\documentclass[psamsfonts,reqno]{amsart}

\usepackage{amssymb,amsfonts}
\usepackage{graphicx}
\usepackage{comment}
\usepackage{scalerel}[2016/12/29]
\usepackage[total={170mm,225mm},top=25mm,left=25mm]{geometry}
\usepackage[all,arc]{xy}
\usepackage{enumerate}
\usepackage{mathrsfs}
\usepackage[linktocpage=true]{hyperref}
\usepackage{bm}
\usepackage{pgfplots}
\usepackage{stmaryrd}
\usepackage{xcolor}
\usepackage[new]{old-arrows}
\usepackage{tikz-cd}
\usepackage{amsmath}
\usepackage{mathtools}
\hypersetup{
    colorlinks,
    linkcolor={blue!100!black},
    citecolor={blue!50!black},
    urlcolor={blue!80!black}
}
\DeclareFontFamily{U}{wncy}{}
\DeclareFontShape{U}{wncy}{m}{n}{<->wncyr10}{}
\DeclareSymbolFont{mcy}{U}{wncy}{m}{n}
\DeclareMathSymbol{\Sh}{\mathord}{mcy}{"58} 
\newcommand\restr[2]{{
  #1 
  |_{#2} 
  }}

\newtheorem{thm}{Theorem}[section]
\newtheorem{cor}[thm]{Corollary}
\newtheorem{prop}[thm]{Proposition}
\newtheorem{lem}[thm]{Lemma}

\newtheorem{quest}[thm]{Question}

\newtheorem{assump}[thm]{Assumptions}
\newtheorem{theorem}{Theorem}

\theoremstyle{definition}
\newtheorem{defn}[thm]{Definition}

\newtheorem{notn}[thm]{Notation}

\theoremstyle{remark}
\newtheorem{rem}[thm]{Remark}

\DeclareMathOperator{\GL}{GL}
\DeclareMathOperator{\SL}{SL}
\DeclareMathOperator{\Gal}{Gal}
\DeclareMathOperator{\der}{der}
\DeclareMathOperator{\ad}{ad}
\DeclareMathOperator{\Ad}{Ad}

\DeclareMathOperator{\Lift}{Lift}

\DeclareMathOperator{\Hom}{Hom}
\DeclareMathOperator{\im}{im}
\DeclareMathOperator{\coker}{coker}

\DeclareMathOperator{\dR}{dR}
\DeclareMathOperator{\WD}{WD}

\DeclareSymbolFont{extraup}{U}{zavm}{m}{n}
\DeclareMathSymbol{\vardiamond}{\mathalpha}{extraup}{87}

\newcommand\sqr{\scaleobj{0.65}{\square}}

\makeatletter
\let\c@equation\c@thm
\makeatother
\numberwithin{equation}{section}

\title{Lifting residual Galois representations with the same semi-simplification}

\author{Stefan Nikoloski}
\address{Department of Mathematics, The Ohio State University, 100 Math Tower, 231 West 18th Avenue, Columbus, OH 43210-1174, USA}
\email{nikoloski.1@osu.edu}


\thanks{I would like to thank Najmuddin Fakhruddin, Toby Gee, and Alice Pozzi for their useful feedback on an earlier draft of this paper. I would like to thank Stefan Patrikis for bringing this project to my attention and for helpful conversations about it. I also thank the referees for their helpful
reports, comments, and suggestions. This work was suppored by NSF grant DMS-2120325.}

\begin{document}

\begin{abstract}

	We consider two modulo $p$ representations $\overline{\rho_1}:\Gamma_F \to G_1(k)$ and $\overline{\rho_2}:\Gamma_F \to G_2(k)$, for $F$ a totally real field, and $G_1$ and $G_2$ possibly distinct reductive groups. We show that under certain conditions both representations admit geometric $p$-adic lifts by allowing extra ramification at the same set of primes. We use this result to answer a version of a question first brought up by Toby Gee and Alice Pozzi. In particular, starting with two modulo $p$ non-semi-simple representations $\overline{\rho_1},\overline{\rho_2}:\Gamma_{\mathbb Q} \to \mathrm{GL}_2(k)$ which have the same semi-simplification, we show under some mild conditions that they are reductions of representations attached to newforms of the same weight $r \ge 2$, the same level $N \ge 1$, and the same Neben character $\varepsilon$.
    
\end{abstract}

\maketitle
\tableofcontents

\section{Introduction}

At around the turn of the century Ravi Ramakrishna (\cite{Ram99} and \cite{Ram02}) developed a purely Galois-theoretic method for producing a geometric $p$-adic lift of a residual representations $\overline{\rho}$. Working with odd absolutely irreducible $\GL_2$-valued representations $\overline{\rho}$ he showed that if a certain Selmer group associated to $\overline{\rho}$ vanishes then $\overline{\rho}$ admits a geometric $p$-adic lift $\rho$. Maybe the more important contribution of Ramakrishna is that he showed under some mild conditions that we can annihilate this Selmer group by allowing ramification at a finite set of primes, and therefore produce a geometric $p$-adic lift $\rho$ of $\overline{\rho}$. The primes where we allow ramification are chosen using the Chebotarev density theorem and they highly depend on $\overline{\rho}$. One is then naturally interested in how much control we can exhibit over this set of primes. One of the many interesting questions in this regard is the following: Given two residual representations $\overline{\rho_1}$ and $\overline{\rho_2}$ can we produce geometric $p$-adic lifts $\rho_1$ and $\rho_2$ of $\overline{\rho_1}$ and $\overline{\rho_2}$, respectively, by allowing ramification at the same set of primes?

In the last $20$ to $25$ years Ramakrishna's lifting method has been vastly generalized and one such generalization is the lifting method of Fakhruddin, Khare, and Patrikis (\cite{FKP21} and \cite{FKP22}). Their lifting method relaxes some of the conditions on the residual representation $\overline{\rho}$ and also allows us to produce a geometric $p$-adic lift of a residual representation valued in a reductive group, and not just $\GL_n$. We will try to give a positive answer to the above question in the setting of the lifting method of Fakhruddin, Khare, and Patrikis. In particular, we will consider residual representation $\overline{\rho_1}$ and $\overline{\rho_2}$ valued in possibly distinct reductive groups.

Let $F$ be a totally real number field and $S$ a finite set of finite places in $F$ containing the places above $p$. Let $\mathcal O$ be the ring of integers of $E$, a finite extension of $\mathbb Q_p$, with a fixed uniformizer $\varpi$ and residue field $k$. For $i=1,2$ let $G_i$ be a smooth group scheme over $\mathcal O$ such that $G_i^0$ is split connected reductive with $G_i/G_i^0$ finite \'etale of order prime to $p$. Let $\overline{\rho_1}:\Gamma_{F,S} \to G_1(k)$ and $\overline{\rho_2}:\Gamma_{F,S} \to G_2(k)$ be continuous residual representations unramified away from $S$. We write $\overline{\rho_i}(\mathfrak g^{\der}_i)$ for the $\mathcal O$-module $\mathfrak g_i^{\der} \otimes_{\mathcal O} k$ with $\Gamma_F$ action given by $\Ad \circ \overline{\rho_i}$. We fix a lift $\mu_i:\Gamma_{F,S} \to G_i/G_i^{\der}(\mathcal O)$ of $\overline{\mu_i} \coloneqq \overline{\rho_i} \pmod{G^{\der}_i}$, which we will moreover assume is de Rham at primes above $p$. We suppose that $\overline{\rho_1}$ and $\overline{\rho_2}$ satisfy the following assumptions:

\begin{assump}
\label{GeneralAssumptions}

\text{ }

\begin{itemize}
    \item $H^1(\Gal(K/F),\overline{\rho_i}(\mathfrak g_i^{\der})^*) = 0$ for $i=1,2$, where $K = F(\overline{\rho_1},\overline{\rho_2},\mu_p)$.
    \item $\overline{\rho_i}(\mathfrak g_i^{\der})$ and $\overline{\rho_j}(\mathfrak g^{\der})^*$ have no common $\mathbb F_p[\Gamma_F]$-subquotient for $i,j=1,2$.

    \item For $i=1,2$ neither $\overline{\rho_i}(\mathfrak g_i^{\der})$ nor $\overline{\rho_i}(\mathfrak g_i^{\der})^*$ contains the trivial representation.
    
    \item For all $\nu \in S$, there is a lift $\rho_{i,\nu}:\Gamma_{F_\nu} \to G_i(\mathcal O)$ with multiplier type $\mu_i$, of $\restr{\overline{\rho_i}}{\Gamma_{F_\nu}}$ for $i=1,2$; and if $\nu \mid p$ this lift can be chosen to be de Rham and Hodge-Tate regular.
    \item $\overline{\rho_i}$ is odd for $i=1,2$, i.e. for every infinite place $\nu$ of $F$, $\dim H^0(\Gamma_{F_\nu},\overline{\rho_i}(\mathfrak g_i^{\der})) = \dim(\mathrm{Flag}_{G_i^{\der}})$.
\end{itemize}
    
\end{assump}

We then prove the following result:

\begin{theorem}[See Theorem \ref{padicLift}]
\label{Theorem A}

    Assume $p \gg_{G_1,G_2} 0$. There exists a finite set $T \supseteq S$ and a finite extension $\mathcal O'$ of $\mathcal O$ such that we can produce geometric lifts $\rho_1:\Gamma_{F,T} \to G_1(\mathcal O')$ with multiplier $\mu_1$ of $\overline{\rho_1}$ and \linebreak $\rho_2:\Gamma_{F,T} \to G_2(\mathcal O')$ with multiplier $\mu_2$ of $\overline{\rho_2}$ by allowing extra ramification at the same set of primes.

\end{theorem}

To achieve this result we will simultaneously lift both $\overline{\rho_1}$ and $\overline{\rho_2}$ using the lifting method of Fakhruddin, Khare, and Patrikis. The lifting method first produces lifts modulo $\varpi^2$, then modulo $\varpi^3$ and so on. We will make sure that even these intermediate lifts are produced by allowing ramification at the same set of primes for both representations. In this general setting it might be impossible to annihilate a Selmer group associated to a residual representation by allowing ramification at some set of auxiliary primes. Therefore, the lifting method of Fakhruddin, Khare, and Patrikis instead relies on the vanishing of relative Selmer groups (see Definition \ref{ModsRelSelmer}) associated to lifts of the residual representation. To even make sense of them we need to first produce lifts modulo high enough power of $\varpi$, which is the reason this lifting method consists of two steps. The first step, called the doubling method, allows us produce lifts of $\overline{\rho_1}$ and $\overline{\rho_2}$ modulo an arbitrary power of $\varpi$ by allowing ramification at the same set of primes. We will choose this set of auxiliary primes from a set of primes with positive density. We do this using a limiting density argument, showing that we can select them so that each intermediate lift has the desired local behavior at each of these primes. We elaborate more on the doubling method in \S 3. In the doubling method whenever we produce a modulo $\varpi^{n+1}$ lift of $\overline{\rho_i}$ starting from a modulo $\varpi^n$ lift of $\overline{\rho_i}$ we must have extra ramification at a some new set of primes. Therefore, even though the doubling method on its own allows us to produce $p$-adic lifts $\rho_1$ and $\rho_2$ of $\overline{\rho_1}$ and $\overline{\rho_2}$, respectively, by allowing ramification at the same set of primes, these lifts will be infinitely ramified, hence not geometric. Because of this, having used the doubling method to produce lifts of $\overline{\rho_1}$ and $\overline{\rho_2}$ modulo high enough power of $\varpi$ we proceed using the second part of the lifting method.

As already mentioned, the second step, called the relative deformation method, relies on the vanishing of relative Selmer groups. We annihilate relative Selmer groups associated to lifts of $\overline{\rho_1}$ and $\overline{\rho_2}$ by allowing extra ramification at a finite set of primes. For our purposes, we need to make sure that we do this using the same set for both representations. The production of primes at which we allow extra ramification uses non-zero classes in the relative Selmer groups. This becomes a big issue in the case where the relative Selmer group for one of the representations is trivial, while for the other it is not. In that case, after allowing ramification at a new prime we can't control the size of the relative Selmer group which was previously trivial. Therefore, it is possible that we enter an endless loop where allowing ramification at one prime will annihilate the relative Selmer group attached to one of the representation, but will make the relative Selmer group attached to the other representations non-zero and so on. In order to overcome this obstacle we use the biggest technical innovation of the paper. In \cite{FKP21} and \cite{FKP22} the authors annihilate the relative Selmer group attached to $\rho_m$, where $\rho_m$ is a modulo $\varpi^m$ lift of $\overline{\rho}$ produced by the doubling method and $m$ is an arbitrary, but fixed integer. In \S 4 we will not fix this integer $m$ beforehand and instead we will let it vary. In such a manner we form a family of relative Selmer groups. We will moreover consider their reductions not only modulo $\varpi$, but modulo higher powers of $\varpi$ as well. By exploiting the relations between the various relative Selmer groups in the family we can find a finite set of primes such that when we allow extra ramification at that set of primes we will lower the dimension of the relative Selmer group which was previously non-zero, while the other relative Selmer group will remain being trivial. Thus, we will eventually annihilate relative Selmer groups attached to lifts of $\overline{\rho_1}$ and $\overline{\rho_2}$ by allowing extra ramification at the same set of auxiliary primes. This will allow us to produce geometric $p$-adic lifts $\rho_1$ and $\rho_2$ of $\overline{\rho_1}$ and $\overline{\rho_2}$, respectively, by allowing ramification at the same set of primes.

\vspace{2 mm}

Theorem \ref{Theorem A}, apart from being interesting on its own for exploring the technical capabilities of the lifting method of Fakhruddin, Khare, and Patrikis, also has arithmetic applications for Galois representations coming from modular forms. Let $f \in S_r(N)$ be an eigenform of weight $r$ and level $N$. To it we can attach a $p$-adic Galois representation $\rho_{f,\nu}:\Gamma_{\mathbb Q} \to \GL_2(E_{f,\nu})$, where $E_f$ is a finite extension of $\mathbb Q$ and $\nu$ is a prime of $E_f$ lying over some integral prime $p$. This Galois representation is uniquely characterized by the property: for primes $\ell \nmid pN$ the representation $\rho_{f,\nu}$ is unramified at $\ell$, and $\mathrm{Tr}(\rho_{f,\nu}(\mathrm{Frob}_\ell)) = a_\ell(f)$, where $a_\ell(f)$ is the eigenvalue of $f$ under the action of the Hecke operator $T_\ell$. The construction of $\rho_{f,\nu}$ is given by Shimura (\cite{Shi71}) when $r=2$, by Deligne (\cite{Del71}) when $r > 2$, and by Deligne and Serre (\cite{DS74}) when $r=1$. It is a standard fact about Galois representations that $E_{f,\nu}^2$ has a stable $\Gamma_{\mathbb Q}$-lattice under the action via $\rho_{f,\nu}$. Therefore, $\rho_{f,\nu}$ is conjugate to an integral representation, which we still label by $\rho_{f,\nu}$, of the form $\Gamma_{\mathbb Q} \to \GL_2(\mathcal O_\nu)$, where $\mathcal O_\nu$ is the ring of integers of $E_{f,\nu}$ with residue field $k_\nu$. In general, we say that a residual representation $\overline{\rho}:\Gamma_{\mathbb Q} \to \GL_2(\overline{\mathbb F_p})$ is \textit{modular} if $\overline{\rho}$ is isomorphic to some $\overline{\rho_{f,\nu}}$, a reduction of a $p$-adic representation attached to an eigenform $f$. The choice of the $\Gamma_{\mathbb Q}$-stable lattice of $E_{f,\nu}^2$ plays an important role in the definition of the reduction $\overline{\rho_{f,\nu}}$. By results of Ribet (\cite{Rib77}) when $r \ge 2$, and Deligne and Serre (\cite{DS74}) when $r = 1$ we know that the integral representation $\rho_{f,\nu}:\Gamma_{\mathbb Q} \to \GL_2(\mathcal O_\nu)$ will be irreducible, and therefore uniquely determined up to conjugacy over $E_{f,\nu}$. However, it is entirely possible that the reduction $\overline{\rho_{f,\nu}}$ will be reducible. As any two such reductions come from integral representations which are conjugates over $E_{f,\nu}$ by the Brauer-Nesbitt theorem their semi-simplifications are isomorphic. In particular, if $\overline{\rho_{f,\nu}}$ is reducible, up to a twist, we can write $\overline{\rho_{f,\nu}}^{\mathrm{ss}} = \overline{\chi} \oplus 1$ for some residual character $\overline{\chi}:\Gamma_{\mathbb Q} \to k_\nu^\times$. However, $\rho_{f,\nu}$ could have multiple non-isomorphic reductions. In fact, by \cite[Proposition $2.1$]{Rib76} if $\overline{\rho_{f,\nu}}^{\mathrm{ss}} = \overline{\chi} \oplus 1$, then there are $\Gamma_{\mathbb Q}$-stable lattices of $E_{f,\nu}^2$ that will give us non-semi-simple reductions of the shape $\begin{pmatrix} \overline{\chi} & * \\ 0 & 1\end{pmatrix}$ and $\begin{pmatrix} 1 & *' \\ 0 & \overline{\chi}\end{pmatrix}$. One is then naturally interested in studying the relations between the extension classes $*$ and $*'$, if any exist. We have the following question which was first asked by Toby Gee and Alice Pozzi:

\begin{quest}
\label{MainQuestion}

Let $\overline{\chi}:\Gamma_\mathbb Q \to \overline{\mathbb F_p}^\times$ be a finitely ramified odd character. Let $\overline{\rho_1} = \begin{pmatrix} \overline{\chi} & * \\ 0 & 1\end{pmatrix}$ and $\overline{\rho_2} = \begin{pmatrix} 1 & *' \\ 0 & \overline{\chi}\end{pmatrix}$ be two finitely ramified residual representations. Under what conditions does there exist an eigenform $f$ and a prime $\nu \mid p$ of $E_f$ such that for a choice of a $\Gamma_{\mathbb Q}$-stable lattice of $E_{f,\nu}^2$ the representation $\rho_{f,\nu}$ has reductions isomorphic to $\overline{\rho_1}$ and $\overline{\rho_2}$?
    
\end{quest}

Here, an odd character means that $\overline{\chi}(c) = -1$, where $c \in \Gamma_{\mathbb Q}$ is the automorphism of $\overline{\mathbb Q}$ induced by complex conjugation. Unless, $\overline{\chi} = \overline{\kappa}^{\pm 1}$, where $\overline{\kappa}$ is the modulo $p$ cyclotomic character, by \cite[Theorem $7.4$]{FKP22} we know that both representation $\overline{\rho_1}$ and $\overline{\rho_2}$ are modular. However, we only know the existence of seemingly unrelated eigenforms for $\overline{\rho_1}$ and $\overline{\rho_2}$. In fact, it is expected that answering Question \ref{MainQuestion} is very hard. Instead, we can try to answer a weaker version of the question, which seems to be more accessible by current lifting methods.

\begin{quest}
\label{WeakerQuestion}

    Keeping the notation of Question \ref{MainQuestion}, when are the residual representations $\overline{\rho_1}$ and $\overline{\rho_2}$ the reductions of representations attached to newforms of the same level, the same weight, and the same Neben character?
    
\end{quest}

Using the lifting method in \cite{FKP22} we can produce $p$-adic lifts $\rho_1$ and $\rho_2$ of $\overline{\rho_1}$ and $\overline{\rho_2}$, respectively, which will be attached to some eigenforms. When a $p$-adic lift is modular the primes at which it is ramified will divide the level of the eigenform attached to it (see \cite{Car89}). Therefore, in order to answer Question \ref{WeakerQuestion} positively it is a paramount that we allow ramification at the same set of additional primes for both representations. The result of Theorem \ref{Theorem A} gives exactly this as long as the representations $\overline{\rho_1}$ and $\overline{\rho_2}$ satisfy Assumptions \ref{GeneralAssumptions}. In \S $5.1$ we will show that this is the case for $\overline{\chi} \neq \overline{\kappa}^{\pm 1}$. 

In addition to this, we need to make sure that the $p$-adic lifts $\rho_1$ and $\rho_2$ exhibit the same local behavior at the primes where the residual representations $\overline{\rho_1}$ and $\overline{\rho_2}$ are already ramified. We remark that it is possible that the extension classes $*$ and $*'$ are ramified at distinct sets of primes. Then, the same would be true for $\overline{\rho_1}$ and $\overline{\rho_2}$, adding another level of complexity. In Assumptions \ref{GeneralAssumptions} we assumed the existence of local $p$-adic lifts of $\restr{\overline{\rho_1}}{\Gamma_{\mathbb Q_\ell}}$ and $\restr{\overline{\rho_2}}{\Gamma_{\mathbb Q_\ell}}$ at all primes $\ell \in S$. We use these local lifts to model the local behavior of the global lifts at each of these primes. Therefore, to prescribe the local behavior of $\rho_1$ and $\rho_2$ at primes in $S$ it is enough to produce local $p$-adic lifts at each of these primes that have the desired local behavior. By the local-global compatibility the power of a prime $\ell$ that will divide the level of a newform $f$ is given by the Artin conductor of $\WD(\restr{\rho_{f,\nu}}{\Gamma_{\mathbb Q_\ell}})$, the Weil-Deligne representation associated to the local representation $\restr{\rho_{f,\nu}}{\Gamma_{\mathbb Q_\ell}}$. Therefore, we want to produce $p$-adic lifts of $\restr{\overline{\rho_1}}{\Gamma_{\mathbb Q_\ell}}$ and $\restr{\overline{\rho_2}}{\Gamma_{\mathbb Q_\ell}}$ whose associated Weil-Deligne representations will have the same Artin conductor. We show in \S $5.2$ that it is always possible to find such $p$-adic local lifts. For primes $\ell \nmid p$ we do this by studying the local universal lifting rings of both $\restr{\overline{\rho_1}}{\Gamma_{\mathbb Q_\ell}}$ and $\restr{\overline{\rho_2}}{\Gamma_{\mathbb Q_\ell}}$, which are explicitly described in \cite[\S 5]{Sho16}. In fact, working on a case-by-base basis we show something stronger: we can find $p$-adic lifts of $\restr{\overline{\rho_1}}{\Gamma_{\mathbb Q_\ell}}$ and $\restr{\overline{\rho_2}}{\Gamma_{\mathbb Q_\ell}}$ which have the same strong inertial type (see Definition \ref{StrongInertialType}). On the other side, for $\ell = p$ we construct by hand $p$-adic lifts of $\restr{\overline{\rho_1}}{\Gamma_{\mathbb Q_p}}$ and $\restr{\overline{\rho_2}}{\Gamma_{\mathbb Q_p}}$ whose Weil-Deligne representations will have the same Artin conductor. These local lifts will also need to satisfy some additional conditions coming from $p$-adic Hodge-Tate theory, in particular we need those lifts to be de Rham. The local lifts we construct will be potentially crystalline and Hodge-Tate regular with Hodge-Tate weights $\{0,r-1\}$ for some integer $r \ge 2$. By the local-global compatibility at the prime $p$ this integer $r$ will be exactly the weight of the newforms associated to the global lifts $\rho_1$ and $\rho_2$. 

The lifting method allows us to prescribe the determinant of the lifts $\rho_1$ and $\rho_2$. In particular, as $\det \overline{\rho_1} = \det \overline{\rho_2} = \overline{\chi}$ we can ask for the $p$-adic lifts $\rho_1$ and $\rho_2$ to have the same determinant. In the language of modular forms this translates to the newforms associated to $\rho_1$ and $\rho_2$ having the same Neben character. 

We then prove the following result

\begin{theorem}[See Theorem \ref{MainResult}]
\label{Theorem B}

    Assume $p \ge 5$. Let $\overline{\rho_1},\overline{\rho_2}:\Gamma_{\mathbb Q,S} \to \GL_2(k)$ be two non-semi-simple residual representations such that

    $$\overline{\rho_1} = \begin{pmatrix} \overline{\chi} & * \\ 0 & 1\end{pmatrix} \quad \quad \mathrm{and} \quad \quad \overline{\rho_2} = \begin{pmatrix} 1 & *' \\ 0 & \overline{\chi}\end{pmatrix}$$
    
    \vspace{2 mm}
    
    \noindent for some odd character $\overline{\chi}$, which moreover is different than $\overline{\kappa}^{\pm 1}$. Then $\overline{\rho_1}$ and $\overline{\rho_2}$ are modular of the same weight $r \ge 2$, the same level $N \ge 1$, and the same Neben character $[\overline{\chi}\overline{\kappa}^{1-r}]$.

\end{theorem}

\section{Notation}

Let $F$ be a totally real number field. We fix an algebraic closure $\overline{F}$ of $F$ and we write $\Gamma_F$ for the absolute Galois group $\Gal(\overline{F}/F)$. For a finite set of primes $S$ of $F$ we let $F(S)$ be the maximal extension of $F$ in $\overline{F}$ that is unramified at primes outside of $S$. We then denote the Galois group $\Gal(F(S)/F)$ by $\Gamma_{F,S}$. For each prime $\nu \in S$ we fix an embedding $\overline{F} \hookrightarrow \overline{F_\nu}$ which gives us an embedding $\Gamma_{F_\nu} \hookrightarrow \Gamma_F$ of local absolute Galois groups into the global absolute Galois group of $F$. In what follows we will often consider tamely ramified lifts, so for an integral prime $p$ and primes $\nu \nmid p$ of $F$ we will work with $\Gal(F_\nu^{\mathrm{tame,p}}/F_\nu)$, which is the maximal tamely ramified with $p$-power ramification quotient of $\Gamma_{F_\nu}$. It is well-known that this group is generated by a lift of the Frobenius $\sigma_\nu$ and a generator of the $p$-part of the tame inertia group $\tau_\nu$, whose interaction is given by the fundamental relation $\sigma_\nu\tau_\nu\sigma_\nu^{-1} = \tau_\nu^{N(\nu)}$. The choice of generator $\tau_\nu$ will factor in some of the computations later on. We choose $\tau_\nu$ as explained in \cite[\S 3]{FKP21}.

We fix an integral prime $p \ge 3$. Let $\mathcal O$ be the ring of integers of a finite extension $E$ of $\mathbb Q_p$ with a uniformizer $\varpi$ and maximal ideal $\mathfrak m = (\varpi)$. We write $k = \mathcal O/\mathfrak m$ for its residue field and we let $e$ be the ramification index of $\mathcal O/\mathbb Z_p$. We let $G$ be a smooth group scheme over $\mathcal O$ such that $G^0$ is split connected reductive and the quotient $G/G^0$, which we label by $\pi_0(G)$, is finite \'etale of order prime to $p$. We write $\mathfrak g$ for its Lie algebra. We let $G^{\der}$ be the derived group of $G^0$ and write $\mathfrak g^{\der}$ for its Lie algebra. We will often abuse the notation and write $\mathfrak g^{\der}$ for $\mathfrak g^{\der} \otimes_\mathcal O k$ when there is not a possibility of confusion. We also let $Z_{G^0}$ and $\mathfrak z_G$ be the center of $G^0$ and its Lie algebra, respectively. We will work with maximal tori $T$ inside $G^0$ and we label their Lie algebras by $\mathfrak t$. For a complete local $\mathcal O$-algebra $\mathcal O'$ with residue field $k'$ we write $\widehat{G}(\mathcal O')$ for the kernel of the reduction $G(\mathcal O') \to G(k')$. For a split maximal torus $T$ of $G^0$ (over $\mathcal O$) and a root $\alpha \in \Phi(G^0,T)$ we let $U_\alpha \subset G^0$ denote the root subgroup that is the image of the root homomorphism (``exponential mapping") $u_\alpha: \mathfrak g_\alpha \to G$. The homomorphism $u_\alpha$ is a $T$-equivariant isomorphism $\mathfrak g_\alpha \to U_\alpha$. See \cite[Theorem $4.1.4$]{Con14}.

In what follows we will often use the exponential isomorphism of groups (see \cite[\S $3.5$]{Til96})

$$\exp: \mathfrak g \otimes_{\mathcal O} I \longrightarrow \ker(G(R) \to G(R/I)),$$

\vspace{2 mm}

\noindent where $R$ is a Noetherian local $\mathcal O$-algebra with residue field $k$, $I$ is a square-zero ideal in $R$, and we view $\mathfrak g \otimes_{\mathcal O} I$ lying inside $\mathfrak g \otimes_{\mathcal O} R$. Moreover, this isomorphism is $G(R/I)$-invariant, where the action on the source is the adjoint action after identifying $\mathfrak g \otimes_{\mathcal O} I$ with $\mathfrak g_{R/I} \otimes_{R/I} I$, while the action on the target is given by lifting to $G(R)$ and conjugating. We will often use the generic notation $\exp$, leaving the ideal $I$ implicit. 

We assume that $p \ge 3$ is very good for $G^{\der}$ (\cite[$\S 1.14$]{Car85}). We also assume that the canonical central isogeny $G^{\der} \times Z_{G^0} \to G^0$ has kernel of order coprime to $p$. We note that we have a $G$-equivariant direct sum decomposition $\mathfrak g = \mathfrak g^{\der} \oplus \mathfrak z_G$. Also, by \cite[$\S 1.16$]{Car85} we have a non-degenerate $G$-invariant trace form $\mathfrak g^{\der} \times \mathfrak g^{\der} \to k$.

Given a group homomorphism $\rho:\Gamma \to H$ and an $H$-module $V$ we write $\rho(V)$ for the associated $\Gamma$-module. More specifically, we will use this whenever we have a continuous representation $\rho_s:\Gamma \to G(\mathcal O/\varpi^s)$ of some topological group $\Gamma$. We will then write $\rho_s(\mathfrak g^{\der})$ for the $\mathcal O$-module $\mathfrak g^{\der} \otimes_{\mathcal O} \mathcal O/\varpi^s$ with the $\Gamma$-action given by $\Ad \circ \rho_s$. For $1 \le r \le s$ our choice of a uniformizer of $\mathcal O$ allows us to view $\mathcal O/\varpi^r$ as a submodule of $\mathcal O/\varpi^s$ by the multiplication by $\varpi^{s-r}$ map. More generally, for $\rho_r \coloneqq \rho_s \mod{\varpi^r}$ we get an inclusion of $\rho_r(\mathfrak g^{\der})$ into $\rho_s(\mathfrak g^{\der})$. Moreover, for an element $X \in \rho_s(\mathfrak g^{\der})$ we write $\mathcal O/\varpi^s[\rho_s] \cdot X$ for the $\mathcal O/\varpi^s[\Gamma]$-submodule of $\rho_s(\mathfrak g^{\der})$ generated by $X$.

We write $\kappa$ for the $p$-adic cyclotomic character. We write $\mathbb Q_p/\mathbb Z_p(1)$ for the abelian group $\mathbb Q_p/\mathbb Z_p$ equipped with a Galois module structure via $\kappa$. We fix an isomorphism between $\mathbb Q_p/\mathbb Z_p(1)$ and $\mu_{p^\infty}(\overline{F})$, which amounts to choosing a compatible collection of $p$-power roots of unity. Using this, for a $\Gamma_F$-module $V$ we can identify the Tate dual $V^* = \Hom(V,\mu_{p^\infty}(\overline{F}))$ with $\Hom(V,\mathbb Q_p/\mathbb Z_p(1))$. Moreover, for a number field $K$ we denote by $K_\infty$ the extension $K(\mu_{p^{\infty}})$.

For a $k[\Gamma_F]$-module $V$ we write $\langle \cdot,\cdot\rangle: V \times V^* \to k(1)$ for the $\Gamma_F$-equivariant, $k$-linear evaluation pairing. For a finite prime $\nu$ of $F$ we have the $k$-linear local Tate pairing $\langle \cdot, \cdot \rangle_\nu : H^1(\Gamma_{F_\nu},M) \times H^1(\Gamma_{F_\nu},M^*) \to k$.

For a finite-dimensional $\mathbb F_p[\Gamma_F]$-module $M$ we write $h^i(\Gamma_F,M)$ for the $\mathbb F_p$-dimension of $H^i(\Gamma_F,M)$.

\section{Doubling Method}

Let $\overline{\rho_1}:\Gamma_{F,S} \to G_1(k)$ and $\overline{\rho_2}:\Gamma_{F,S} \to G_2(k)$ be two residual representations, where $S$ is a finite set of primes in $F$ containing all the primes above $p$. Here, $G_1$ and $G_2$ are group schemes that satisfy the conditions mentioned in \S 2. We now impose some lower bounds on the size of the prime $p$ in terms of the root data of $G_1$ and $G_2$. In particular, we assume that $p$ is large enough so that all the lemmas in Appendix $\hyperref[Appendix]{A}$ are satisfied, remarking that we will apply Lemma \ref{NonInclusionLemma} with $n=4$ and Lemma \ref{TorusandRootLemma} with $r=s=1$. In the general case we leave this bound implicit, but we will compute it explicitly in Lemma \ref{BoundsVerification} for the special case of $\GL_2$-representations. 

We assume that $\overline{\rho_i}$ surjects onto $\pi_0(G_i)$. If not, we can replace $G_i$ with the preimage in $G_i$ of the image of $\overline{\rho_i}$ in $\pi_0(G_i)$, which will not affect the deformation theory of $\overline{\rho_i}$. Let $H \unlhd \Gamma_F$ be the subgroup consisting of elements $\sigma$ such that $\overline{\rho_i}(\sigma)$ lands inside $G_i^0(k)$ for both $i=1,2$. Let $\widetilde{F}$ be the fixed field of $H$. Then $\overline{\rho_1} \times \overline{\rho_2}$ induces an injection $\Gal(\widetilde{F}/F) \hookrightarrow \pi_0(G_1) \times \pi_0(G_2)$. We fix a lift $\mu_i:\Gamma_{F,S} \to G_i/G_i^{\der}(\mathcal O)$ of $\overline{\mu_i} \coloneqq \overline{\rho_i} \pmod{G_i^{\der}}$, which we will moreover assume is de Rham at primes above $p$. We note that as $G_i/G_i^{\der}(k)$ has order prime to $p$ such a lift exists. In particular, we can take the Teichm\"uller lift, which will be finitely ramified and hence de Rham at primes over $p$. Set $K_1 = F(\overline{\rho_1},\mu_p)$, $K_2 = F(\overline{\rho_2},\mu_p)$ and $K = K_1K_2$. We also write $D$ for the largest integer such that $\mu_{p^D}$ is contained in $K$.

As seen in \cite[Theorem $5.2$]{FKP22}, by allowing extra ramification at a finite set of primes we can produce $p$-adic lifts of $\overline{\rho_1}$, and likewise of $\overline{\rho_2}$. We will show that we can produce such lifts by allowing extra ramification at the same set of primes for both representations. To do this we will go step-by-step through the lifting method of \cite{FKP22}. By making some necessary modifications along the way we will show that this is indeed possible. The lifting method consists of two major parts, which are the doubling method and the relative deformation method. In both of these methods, starting with a modulo $\varpi^n$ lift of $\overline{\rho_i}$ we produce a modulo $\varpi^{n+1}$ lift of $\overline{\rho_i}$ with the desired local properties, by possibly allowing ramification at some extra primes. We will show that we can choose the same set of primes for lifts of both $\overline{\rho_1}$ and $\overline{\rho_2}$ in each of these individual steps.

Most of the results in the next two sections are based on results in \cite{FKP21} and \cite{FKP22}. In order to orient the reader and make it easier to notice the modifications we have made in each of them, we directly reference the analogous results in the original lifting result of Fakhruddin, Khare, and Patrikis. Some of our results are direct generalizations of their results to the case of two representations. As it might not be immediately obvious how to generalize some of the technical parts of the proofs when working with multiple representations we prove the results in full details. 

\subsection{Set-up for lifting modulo \texorpdfstring{$\varpi^2$}{ϖ\textasciicircum2}}

In the doubling method we will produce lifts $\rho_{i,n}:\Gamma_{F,S_n} \to G_i(\mathcal O/\varpi^n)$ of $\overline{\rho_i}$ for $i=1,2$ and $n$ large enough. To do this we will need to enlarge the set $S$. In the case of a single representation this is done by allowing ramification at carefully chosen primes which are \textit{trivial} (see \cite[Definition $3.2$]{FKP22}) with respect to $\overline{\rho_i}$. These are precisely the primes in $F$ that split in $K_i$. As we want to allow ramification at the same set of primes for lifts of both $\overline{\rho_1}$ and $\overline{\rho_2}$ we choose primes which are trivial with respect to both of these residual representations. In particular, we will choose primes of $F$ which split in $K$.

We recall that for each prime $\nu$ in $S$ we have fixed an embedding $\overline{F} \hookrightarrow \overline{F_\nu}$. During the lifting process we will enlarge $S$ by primes which are trivial for both $\overline{\rho_1}$ and $\overline{\rho_2}$. At each such prime $\nu$ we will also fix an embedding $\overline{F} \hookrightarrow \overline{F_\nu}$ and hence make sense of $\sigma_\nu$ and $\tau_\nu$ within $\Gamma_F$. We will do this by specifying a decomposition group at $\nu$ inside a finite Galois extension $L/F$ which plays a role in the choice of the prime $\nu$. For a prescribed element $\sigma \in \Gal(L/F)$ using Chebotarev density theorem we can find a positive density set of primes $\nu$ of $F$ for which there exists a prime $\nu_L$ of $L$ that lies over $\nu$ and $\mathrm{Frob}_{\nu_L} = \sigma$. The choice of this prime $\nu_L$ will sufficiently specify a decomposition group at $\nu$.

\begin{notn}

\label{Cocyleimage}

Let $M$ be any of the Galois modules $\overline{\rho_i}(\mathfrak g_i^{\der})$ or $\overline{\rho_i}(\mathfrak g_i^{\der})^*$, and $\phi$ a cohomology class in $H^1(\Gamma_{F,S'},M)$, where $S \subseteq S'$. For a trivial prime $\nu$ with respect to both $\overline{\rho_1}$ and $\overline{\rho_2}$ we can make sense of $\phi(\sigma_\nu)$ and $\phi(\tau_\nu)$ as elements of $M$. As $K$ trivializes the action on $M$ we have that $\restr{\phi}{\Gamma_K}$ is a homomorphism and it doesn't depend on the choice of a cocycle representative of the cohomology class. We let $K_\phi$ be the fixed field of $\restr{\phi}{\Gamma_K}$ over $K$. We note that $K_\phi/F$ is a finite Galois extension. A choice of a decomposition group at $\nu$ inside $\im \restr{\phi}{\Gamma_K} \simeq \Gal(K_\phi/K) \subseteq \Gal(K_\phi/F)$ will determine the image of $\sigma_\nu$ and $\tau_\nu$ under the cocycle $\phi$. Therefore, when we refer to $\phi(\sigma_\nu)$ and $\phi(\tau_\nu)$ we mean that there is a choice of a decomposition group at $\nu$ which gives us the appointed values of $\sigma_\nu$ and $\tau_\nu$ under $\phi$. Moreover, we note that these do not depend on the choice of a cocycle representative of the cohomology class. Indeed, as $\nu$ splits in $K$ we have that $\Gamma_{F_\nu} \subseteq \Gamma_K$. Therefore, $B^1(\Gamma_{F_\nu},M) = 0$ and $\restr{\phi}{\Gamma_{F_\nu}}$ doesn't depend on a cocycle representative.

\end{notn}

At the prime where we will allow extra ramification during the lifting process we will ask for the local lifts to be of special type. With that in mind we define the following lifting functor:

\begin{defn}
\label{LiftFunctorDefinition}

	Let $\nu$ be a prime split in $K$. Fix a split maximal torus $T_i$ of $G_i^0$ (over $\mathcal O$) and a root $\alpha_i \in \Phi(G_i^0,T_i)$. Define $\Lift_{\restr{\overline{\rho_i}}{\Gamma_{F_\nu}}}^{\mu_i,\alpha_i}(R)$ to be the set of lifts of $\restr{\overline{\rho_i}}{\Gamma_{F_\nu}}$ with multiplier $\mu_i$ which are $\widehat{G_i}(R)$-conjugate to a lift $\rho_i:\Gamma_{F_\nu} \to G_i(R)$ satisfying
	
	\begin{itemize}
		\item $\rho_i(\sigma_\nu) \in T_i(R)$.
		\item $\alpha_i(\rho_i(\sigma_\nu)) = N(\nu)$ in $R^\times$.
		\item $\rho_i(\tau_\nu) \in U_{\alpha_i}(R)$.
	
	\end{itemize}

\end{defn}

\begin{rem}

This definition differs slightly from \cite[Definition 3.1]{FKP21}. However, the modification simplifies some of the computations in \cite[$\S 3$]{FKP21} and the relevant results from that section sill hold. In particular the lifting functor is formally smooth.

\end{rem}

\begin{lem}
\label{AnnihilationofSh}

    The set $S$ can be enlarged to a finite set $T$ by adding trivial primes with respect to both $\overline{\rho_1}$ and $\overline{\rho_2}$ such that $\Sh_T^1(\Gamma_{F,T},\overline{\rho_i}(\mathfrak g_i^{\der})^*) = 0$ for $i=1,2$.
    
\end{lem}
\begin{proof}
    We recall that $\Sh_T^1(\Gamma_{F,T},\overline{\rho_i}(\mathfrak g_i^{\der})^*)$ is defined as $\bigcap_{\nu \in T} \ker (H^1(\Gamma_{F,T},\overline{\rho_i}(\mathfrak g_i^{\der})^*) \to H^1(\Gamma_{F_\nu},\overline{\rho_i}(\mathfrak g_i^{\der})^*))$. Therefore, it is enough for $0 \neq \psi \in \Sh_S^1(\Gamma_{F,S},\overline{\rho_i}(\mathfrak g_i^{\der})^*)$ to find a trivial prime $\nu$ for both $\overline{\rho_1}$ and $\overline{\rho_2}$ such that $\restr{\psi}{\Gamma_{F_\nu}} \neq 0$. Indeed, by adding one such prime $\nu$ to $S$ for each $0 \neq \psi \in \Sh_S^1(\Gamma_{F,S},\overline{\rho_i}(\mathfrak g_i^{\der})^*)$ we will get the desired finite set $T$.

    Now, let $0 \neq \psi \in \Sh_S^1(\Gamma_{F,S},\overline{\rho_i}(\mathfrak g_i^{\der})^*)$. As $\Gamma_{F,S}$ is a quotient of $\Gamma_{F}$ we can identify $H^1(\Gamma_{F,S},\overline{\rho_i}(\mathfrak g_i^{\der})^*)$ with the subgroup of $H^1(\Gamma_F,\overline{\rho_i}(\mathfrak g_i^{\der})^*)$ that is trivial on inertia groups outside of $S$. Now from the inflation-restriction sequence we have

    $$0 \longrightarrow H^1(\Gal(K/F),\overline{\rho_i}(\mathfrak g_i^{\der})^*) \longrightarrow H^1(\Gamma_F,\overline{\rho_i}(\mathfrak g_i^{\der})^*) \longrightarrow H^1(\Gamma_K,\overline{\rho_i}(\mathfrak g_i^{\der})^*).$$

    \vspace{2 mm}

    By the first bullet point of Assumptions \ref{GeneralAssumptions} the first term of the exact sequence is $0$, so by the exactness $\restr{\psi}{\Gamma_K} \neq 0$. As above we let $K_\psi$ be the fixed field of $\restr{\psi}{\Gamma_K}$ over $K$, which will be a non-trivial finite Galois extension of $K$. Moreover, as already explained $K_\psi/F$ is also a finite Galois extension. We choose a non-trivial element $\sigma$ of $\im \restr{\psi}{\Gamma_K}  \simeq \Gal(K_\psi/K) \subseteq \Gal(K_\psi/F)$ and by the Chebotarev density theorem we can find a prime $\nu$ in $F$, disjoint from $S$, and a decomposition group at $\nu$ so that $\restr{\sigma_\nu}{{K_\psi}} = \sigma$. By our choice of $\sigma$ we have that $\nu$ is a trivial prime with respect to both $\overline{\rho_1}$ and $\overline{\rho_2}$, and $\restr{\sigma_\nu}{\Gamma_{K_\psi}} \neq 1$, which is exactly what we wanted.

\end{proof}

\begin{rem}

\label{VanishingofSh}

By global duality for our choice of $T$ coming from Lemma \ref{AnnihilationofSh} we have that $\Sh_T^2(\Gamma_{F,T},\overline{\rho_i}(\mathfrak g_i^{\der})) = 0$ for $i=1,2$. The vanishing of this Tate-Shafarevich group tells us that any obstruction to producing multiplier $\mu_i$ lifts of $\overline{\rho_i}$ is local. In other words, it suffices to check that $\restr{\overline{\rho_i}}{\Gamma_{F_\nu}}$ admits a lift with multiplier $\mu_i$ for each $\nu \in T$.

\end{rem}

We suppose that $T$ is a finite enlargement of $S$ as in Lemma \ref{AnnihilationofSh}.

\begin{lem}[{cf. \cite[Lemma $3.4$]{FKP22}}]

    \label{M_ZEnlargement}

    \text{ }
    \begin{enumerate}
        \item There is a finite enlargement by trivial primes for both $\overline{\rho_1}$ and $\overline{\rho_2}$ of $T$ (still labeled by $T$) such that for all cyclic submodules $M_Z \coloneqq \mathbb F_p[\overline{\rho_i}] \cdot Z \subseteq \overline{\rho_i}(\mathfrak g_i^{\der})$ (for $Z \in \overline{\rho_i}(\mathfrak g_i^{\der})$), $\dim_{\mathbb F_p} \Sh^1_T(\Gamma_{F,T},M_Z^*)$ is minimal among such enlargements, i.e. it is equal to $\dim \Sh^1_{T_{\max}}(\Gamma_{F,T},M_Z^*)$, where $T_{\max}$ is the union of $S$ and the set of primes in $F$ which are trivial for both $\overline{\rho_1}$ and $\overline{\rho_2}$.
        
        \item Let $T$ be the enlargement produced in part $(1)$. For any trivial prime $w \notin T$ with respect to both $\overline{\rho_1}$ and $\overline{\rho_2}$, let $L_{w,Z} \subseteq H^1(\Gamma_{F_w},M_Z)$ be the subspace of $1$-cocycles $\phi$ such that $\phi(\tau_w) \in \mathbb F_p \cdot Z$. Then there is an exact sequence:

        $$0 \longrightarrow H^1(\Gamma_{F,T},M_Z) \longrightarrow H^1_{L_{w,Z}}(\Gamma_{F,T \cup w},M_Z) \xrightarrow{\mathrm{ev}_{\tau_w}} \mathbb F_p \cdot Z \longrightarrow 0,$$

        \vspace{2 mm}

        \noindent where $\mathrm{ev}_{\tau_w}$ is the evaluation map $\phi \to \phi(\tau_w)$ and $H^1_{L_{w,Z}}(\Gamma_{F,T \cup w},M_Z)$ is the Selmer group with local conditions $L_{w,Z}$ at $w$ and unramified local conditions away from $T \cup w$.

        \item There is a further enlargement of $T$ produced in part $(1)$ that satisfies the analogous statement in parts $(1)$ and $(2)$ with respect to the cyclic submodules $M_\lambda \coloneqq \mathbb F_p[\overline{\rho_i}] \cdot \lambda \subseteq \overline{\rho_i}(\mathfrak g_i^{\der})^*$ for all $\lambda \in \overline{\rho_i}(\mathfrak g_i^{\der})^*$.
    \end{enumerate}
\end{lem}
\begin{proof}

    \text{ }

    \begin{enumerate}
        \item For any $\psi \in \Sh^1_T(\Gamma_{F,T},M_Z^*)$, but not in $\Sh^1_{T_{\max}}(\Gamma_{F,T},M_Z^*)$ we can find a prime $\nu \in T_{\max} \,\setminus\, T$ such that $\restr{\psi}{\Gamma_{F_\nu}} \neq 0$. We enlarge $T$ to $T'$ by adding one such prime for each $\psi$. As $T$ is a finite set and $M_Z^*$ is a finite module $\Sh_T^1(\Gamma_{F,T},M_Z^*)$ is finite. Therefore, $T'$ is a finite enlargement of $T$. By construction $\Sh_{T'}^1(\Gamma_{F,T'},M_Z^*) = \Sh^1_{T_{\max}}(\Gamma_{F,T},M_Z^*)$ and we deduce that $T'$ is minimal among all such enlargements in the sense explained in the statement of the lemma. Moreover, the equality tell us that any further enlargement of $T'$ will not affect the size of $\Sh_{T'}^1(\Gamma_{F,T'},M_Z^*)$. Hence, we can run this argument for any $M_Z$, of which there are finitely many as $\overline{\rho_i}(\mathfrak g_i^{\der})$ is finite, to get the desired finite enlargement.

        \item We have an exact sequence

        $$0 \longrightarrow H^1(\Gamma_{F,T},M_Z) \longrightarrow H^1_{L_{w,Z}}(\Gamma_{F,T \cup w},M_Z) \xrightarrow{\mathrm{ev}_{\tau_w}} \mathbb F_p \cdot Z.$$

        \vspace{2 mm}

        Since $H^1(\Gamma_{F,T},M_Z)$ and $H_{L_{w,Z}}^1(\Gamma_{F,T \cup w}, M_Z)$ differ only at the local behavior at $w$ and $L_w^{\mathrm{unr}} \subseteq L_{w,Z}$, we impose less strict local conditions at $w$ in the latter giving us the exactness at the first term. As $H^1(\Gamma_{F,T},M_Z)$ will contain exactly the classes of $H^1_{L_{w,Z}}(\Gamma_{F,T \cup w},M_Z)$ which are unramified at $w$, i.e. lie in the kernel of $\mathrm{ev}_{\tau_w}$ we get the exactness at the second term. It remains to show that the map $\mathrm{ev}_{\tau_w}$ is surjective. As the target $\mathbb F_p \cdot Z$ has $\mathbb F_p$-dimension $1$ it suffices to show that $H^1_{L_{w,Z}}(\Gamma_{F,T \cup w},M_Z)$ has a strictly greater $\mathbb F_p$-dimension that $H^1(\Gamma_{F,T},M_Z)$. Applying the Greenberg-Wiles formula once with unramified conditions away from $T$, and once with unramified conditions away from $T \cup w$ and conditions $L_{w,Z}$ at $w$, and subtracting them we get:

        $$h^1_{L_{w,Z}}(\Gamma_{F,T \cup w},M_Z) - h^1(\Gamma_{F,T},M_Z) - h^1_{\{0\}_{\nu \in T} \cup L_{w,Z}^\perp}(\Gamma_{F,T \cup w},M_Z^*) + \dim_{\mathbb F_p} \Sh_T^1(\Gamma_{F,T},M_Z^*) = \dim L_{w,Z} - \dim L_w^{\mathrm{unr}},$$

        \vspace{2 mm}

        \noindent where we used the that the corresponding dual Selmer group to $H^1(\Gamma_{F,T},M_Z)$ is $\Sh_T^1(\Gamma_{F,T},M_Z^*)$. Here $L_{w,Z}^\perp \subseteq H^1(\Gamma_{F_w},M_Z^*)$ is the annihilator of $L_{w,Z}$ under the local Tate duality, and we write $H^1_{\{0\}_{\nu \in T} \cup L_{w,Z}^\perp}(\Gamma_{F,T \cup w},M_Z^*)$ for the Selmer group with trivial local conditions at primes in $T$, local conditions $L_{w,Z}^\perp$ at $w$, and unramified local conditions away from $T \cup w$.
        
        Now, we consider the map $\mathrm{ev}_{\tau_w}$, viewing it as a map from $L_{w,Z}$ to $\mathbb F_p \cdot Z$. We'll show that this is a surjective map. As $w$ is a trivial prime for $\overline{\rho_i}$, $\Gamma_{F_w}$ acts trivially on $M_Z$ and hence $H^1(\Gamma_{F_w},M_Z) = \Hom (\Gamma_{F_w},M_Z)$. Therefore, it is enough to enough to give a map $\Gamma_{F_w} \to M_Z$ that will send $\tau_w$ to a non-trivial element of $\mathbb F_p \cdot Z$. We know that $\Gal(F^{\mathrm{tame},p}_w/F_w) \simeq \mathbb Z_p \rtimes \widehat{\mathbb Z}$, where the first factor is generated by $\tau_w$, the second one by $\sigma_w$, and the semi-direct action is given by the cyclotomic character. We can then define a homomorphism from this Galois group to $\mathbb F_p \cdot Z$ by sending $\sigma_w$ to $0$ and $\tau_w$ to $Z$. This is a well-defined homomorphism because it respects the fundamental relation between $\sigma_w$ and $\tau_w$ since $\kappa(\sigma_w) \equiv 1 \pmod{p}$, as $w$ is a trivial prime for $\overline{\rho_i}$. Therefore, as $\ker (\mathrm{ev}_{\tau_w}) = L_w^{\mathrm{unr}}$ we have that $\dim L_{w,Z} - \dim L_w^{\mathrm{unr}} = 1$. On the other hand

        $$\Sh_T^1(\Gamma_{F,T},M_Z^*) = \Sh^1_{T \cup w}(\Gamma_{F,T \cup w},M_Z^*) \subseteq H^1_{\{0\}_{\nu \in T} \cup L_{w,Z}^\perp}(\Gamma_{F,T \cup w},M_Z^*),$$

        \vspace{2 mm}

        \noindent where the equality follows from the minimality of $T$ by part $(1)$, while the inclusion follows because the two Selmer group differ only by the local conditions at $w$, which are $0$ and $L_{w,Z}^\perp$, respectively. Plugging all this in the equation above we get

        $$h^1_{L_{w,Z}}(\Gamma_{F,T \cup w},M_Z) - h^1(\Gamma_{F,T},M_Z) \ge 1,$$

        \vspace{2 mm}

        \noindent proving the surjectivity of $\mathrm{ev}_{\tau_w}$.

        \item This follows from the analogous proofs of parts $(1)$ and $(2)$ with $M_\lambda$ in place of $M_Z$.

    \end{enumerate}
    
\end{proof}

\begin{rem}

    \label{Enlargement}

    As already noted during the proofs, the conclusions of Lemma \ref{AnnihilationofSh} and Lemma \ref{M_ZEnlargement} will hold for any further enlargements of $T$. We will use this in the future without ever explicitly mentioning it.
    
\end{rem}

We now fix $k$-bases $\{e_{1,b}^*\}_{b \in B_1}$ of $\overline{\rho_1}(\mathfrak g_1^{\der})^*$ and $\{e_{2,b}^*\}_{b \in B_2}$ of $\overline{\rho_2}(\mathfrak g_2^{\der})^*$. Then for each $b \in B_1 \sqcup B_2$ we fix a prime $t_b$, which is trivial with respect to both $\overline{\rho_1}$ and $\overline{\rho_2}$ and we add it to $T$. We also include one trivial prime $t_0$ in $T$. Finally, we fix two more trivial primes $t_1$ and $t_2$, but we do not include them in $T$. We will use these primes to ensure that the fixed fields of certain cocycles are linearly disjoint over $K$. We now write $T' \coloneqq T \,\setminus \, (t_0 \cup \{t_{b}\}_{b\in B_1 \sqcup B_2})$. Since $T$ was arranged to satisfy part $(1)$ of Lemma \ref{M_ZEnlargement} before the addition of these primes we can apply the lemma to $T'$ in place of $T$ and $t_b$ in place of $w$. For a fixed $b \in B_i$ we find a cohomology class

$$\theta_{i,b} \in H^1(\Gamma_{F,T' \cup t_b},M_{e_{i,b}^*})$$

\vspace{2 mm}

\noindent such that $\theta_{i,b}(\tau_{t_b}) = e_{i,b}^*$. Similarly, for any trivial prime $\nu \notin T$ we can find a cohomology class

$$\theta_{i,b}^{(\nu)} \in H^1(\Gamma_{F,T' \cup \nu},M_{e_{i,b}^*})$$

\vspace{2 mm}

\noindent such that $\theta_{i,b}^{(\nu)}(\tau_\nu) = e_{i,b}^*$. We then set

$$\eta_{i,b}^{(\nu)} \coloneqq \theta_{i,b} + \theta_{i,b}^{(\nu)} \in H^1(\Gamma_{F,T' \cup \{t_b,\nu\}},M_{e_{i,b}^*})$$

\vspace{2 mm}

\noindent so that $\eta_{i,b}^{(\nu)}(\tau_\nu) = \eta_{i,b}^{(\nu)}(\tau_{t_b}) = e_{i,b}^*$. We emphasize that from the way these classes were constructed none of them is ramified at $t_0$. We now have the following key technical result:

\begin{prop}[cf. {\cite[Proposition $3.6$]{FKP22}}]
    \label{cokernelgenerators}

    Let $r$ be the dimension (over $\mathbb F_p$) of the cokernel of the restriction map

    $$\Psi_{1,T} : H^1(\Gamma_{F,T},\overline{\rho_1}(\mathfrak g_1^{\der})) \to \bigoplus_{\nu \in T} H^1(\Gamma_{F_\nu},\overline{\rho_1}(\mathfrak g_1^{\der})).$$

    \vspace{2 mm}

    Fix an integer $c \ge D+1$ and a Galois extension $L/F$ containing $K$, unramified outside $T$, and linearly disjoint over $K$ from the composite of $K_\infty$ and the fixed fields of $K_\psi$ of any collection of classes $\psi \in H^1(\Gamma_{F,T},\overline{\rho_1}(\mathfrak g_1^{\der})^*)$. Then there is

    \begin{itemize}
        \item a collection $\{Y_i\}_{i=1}^r$ of elements of $\bigoplus_{\nu \in T} H^1(\Gamma_{F_\nu},\overline{\rho_1}(\mathfrak g_1^{\der}))$ with image equal to an $\mathbb F_p$-basis of $\coker(\Psi_{1,T})$; and for each $i$
        \item a class $q_i \in \ker((\mathbb Z/p^c)^\times \to (\mathbb Z/p^D)^\times)$ that is non-trivial modulo $p^{D+1}$;
        \item a split maximal torus $T_i$, a root $\alpha_i \in \Phi(G^0,T_i)$, and $X_{\alpha_i}$ a non-zero image modulo $\varpi$ of some root vector with respect to $\alpha_i$; and at this point choosing elements $g_1,g_2,\dots,g_r$ of $\Gal(L/K)$ and non-zero $X_1,\dots,X_r \in \mathfrak g^{\der}_2 \otimes k$,
        \item a Chebotarev set $\mathcal C_i$ of trivial primes $\nu \notin T$ and a positive upper-density subset $\mathfrak l_i \subseteq \mathcal C_i$;
        \item for each $\nu \in \mathcal C_i$ a choice of a decomposition group at $\nu$;
        \item for each $\nu \in \mathfrak l_i$ classes $h_1^{(\nu)} \in H^1(\Gamma_{F,T \cup \nu},\overline{\rho_1}(\mathfrak g_1^{\der}))$ and $h_2^{(\nu,t_2)} \in H^1(\Gamma_{F,T \cup \{\nu,t_2\}},\overline{\rho_2}(\mathfrak g_2^{\der}))$;
    \end{itemize}

    such that

    \begin{itemize}
        \item For all $\nu \in \mathcal C_i$, $N(\nu) \equiv q_i \pmod{p^c}$ and the restriction of $\sigma_\nu$ in $\Gal(L/K)$ is $g_i$.
        \item For all $\nu \in \mathfrak l_i$

        \begin{itemize}
            \item $\restr{h_1^{(\nu)}}{T} = Y_i$.
            \item $h_1^{(\nu)}(\tau_\nu) = h_1^{(\nu)}(\tau_{t_0}) = X_{\alpha_i}$.
            \item $h_1^{(\nu)}$ lies in the image of $H^1(\Gamma_{F,T \cup \nu},\mathbb F_p[\overline{\rho_1}] \cdot X_{\alpha_i}) \to H^1(\Gamma_{F,T \cup \nu},\overline{\rho_1}(\mathfrak g_1^{\der}))$.
            \item $\restr{h_2^{(\nu,t_2)}}{T}$ is independent of $\nu \in \mathfrak l_i$.
            \item $h_2^{(\nu,t_2)}(\tau_\nu) = h_2^{(\nu,t_2)}(\tau_{t_0}) = h_2^{(\nu,t_2)}(\tau_{t_2}) = X_i$.
        \end{itemize}
    \end{itemize}
    
    \vspace{2 mm}

    Similarly, for $c \ge D+1$, $L$ as above, and non-zero $Z \in \mathfrak g_1^{\der} \otimes k$ there is a class $q_Z \in \ker((\mathbb Z/p^c)^\times \to (\mathbb Z/p^D)^\times)$ that is non-trivial modulo $p^{D+1}$, and for any choice of $g_Z \in \Gal(L/K)$ and non-zero $X \in \mathfrak g^{\der}_2 \otimes k$, a Chebotarev set $\mathcal C_Z$ of trivial primes (and a choice of decomposition group at each such prime) containing a positive upper-density subset $\mathfrak l_Z$, and for each $\nu \in \mathfrak l_Z$ classes $h_1^{(\nu)} \in H^1(\Gamma_{F,T \cup \nu},\overline{\rho_1}(\mathfrak g_1^{\der}))$ and $h_2^{(\nu,t_2)} \in H^1(\Gamma_{F,T \cup \{\nu,t_2\}},\overline{\rho_2}(\mathfrak g_2^{\der}))$ such that

    \begin{itemize}
        \item For all $\nu \in \mathcal C_Z$, $N(\nu) \equiv q_Z \pmod{p^c}$ and the restriction of $\sigma_\nu$ in $\Gal(L/K)$ is $g_Z$.
        \item For all $\nu \in \mathfrak l_Z$

        \begin{itemize}
            \item $\restr{h_1^{(\nu)}}{T}$ is independent of $\nu \in \mathfrak l_Z$.
            \item $h_1^{(\nu)}(\tau_\nu) = h_1^{(\nu)}(\tau_{t_0}) = Z$.
            \item $h_1^{(\nu)}$ lies in the image of $H^1(\Gamma_{F,T \cup \nu},\mathbb F_p[\overline{\rho_1}] \cdot Z) \to H^1(\Gamma_{F,T \cup \nu},\overline{\rho_1}(\mathfrak g_1^{\der}))$.
            \item $\restr{h_2^{(\nu,t_2)}}{T}$ is independent of $\nu \in \mathfrak l_Z$.
            \item $h_2^{(\nu,t_2)}(\tau_\nu) = h_2^{(\nu,t_2)}(\tau_{t_0}) = h_2^{(\nu,t_2)}(\tau_{t_2}) = X$.

        \end{itemize}
    \end{itemize}

\end{prop}

\begin{rem}

    \label{1and2switch}

    The statement of the proposition holds with the roles of $\overline{\rho_1}$ and $\overline{\rho_2}$, and $t_1$ and $t_2$ reversed. 
    
\end{rem}

\begin{proof}

   Most of the proposition follows directly from \cite[Proposition $3.6$]{FKP22}. All the data, except the class $h_2^{(\nu,t_2)} \in H^1(\Gamma_{F,T \cup \{\nu,t_2\}},\overline{\rho_2}(\mathfrak g_2^{\der}))$, that will satisfy these properties is provided by \cite[Proposition $3.6$]{FKP22}. So it only remains to show the existence of the class $h_2^{(\nu,t_2)}$.

   We choose a non-zero $X_i \in \mathfrak g_2^{\der} \otimes k$ and let $\nu$ be a trivial prime in the positive upper-density set $\mathfrak l_i$ provided by \cite[Proposition $3.6$]{FKP22}. Then, $T \,\setminus\,t_0$ still satisfies part $(1)$ of Lemma \ref{M_ZEnlargement}, so using the lemma we can find $\theta_2^{(\nu)} \in H^1(\Gamma_{F,(T \setminus t_0) \cup \nu},\mathbb F_p[\overline{\rho_2}] \cdot X_i)$ such that $\theta_2^{(\nu)}(\tau_\nu) = X_i$. Similarly, we can find $\theta_2^{(t_0)} \in H^1(\Gamma_{F,T},\mathbb F_p[\overline{\rho_2}] \cdot X_i)$ such that $\theta_2^{(t_0)}(\tau_{t_0}) = X_i$ and $\theta_2^{(t_2)} \in H^1(\Gamma_{F,(T\setminus t_0)\cup t_2},\mathbb F_p[\overline{\rho_2}]\cdot X_i)$ such that $\theta_2^{(t_2)}(\tau_{t_2}) = X_i$. Set $h_2^{(\nu,t_2)} \in H^1(\Gamma_{F,T\cup\{\nu,t_2\}},\overline{\rho_2}(\mathfrak g_2^{\der}))$ to be the image of $\theta_2^{(\nu)} + \theta_2^{(t_0)} + \theta_2^{(t_2)} \in H^1(\Gamma_{F,T \cup \{\nu,t_2\}},\mathbb F_p[\overline{\rho_2}]\cdot X_i)$. As $\theta_2^{(\nu)}$ is the only one of the three that is ramified at $\nu$ we get $h_2^{(\nu,t_2)}(\tau_\nu) = \theta_2^{(\nu)}(\tau_\nu) = X_i$. Similarly, $h_2^{(\nu,t_2)}(\tau_{t_0}) = h_2^{(\nu,t_2)}(\tau_{t_2}) = X_i$. Finally, we group the primes in $\mathfrak l_i$ based on the image $\restr{h_2^{(\nu,t_2)}}{T}$. Since $\bigoplus_{w \in T} H^1(\Gamma_{F_w},\overline{\rho_2}(\mathfrak g_2^{\der}))$ is a finite space using Lemma \ref{DensityLemma} we can find a positive upper-density subset of $\mathfrak l_i$, which we still label by $\mathfrak l_i$, such that $\restr{h_2^{(\nu,t_2)}}{T}$ is independent of $\nu \in \mathfrak l_i$.

   We use the same argument to produce the class $h_2^{(\nu,t_2)}$ in the second part of the lemma with $\mathfrak l_Z$ in place of $\mathfrak l_i$.
    
\end{proof}

The following result will allow us to establish the necessary linear disjointness that will play a crucial part in some of the later results

\begin{lem}
\label{LinearDisjointnessLemma}

For $1 \le j \le n$ let $h_j \in H^1(\Gamma_F,\mathbb F_p[\overline{\rho_{i_j}}] \cdot X_j)$, where $i_j = 1$ or $2$, and $X_j$ is a non-zero element of $\mathfrak g^{\der}_{i_j} \otimes k$. Suppose that for each pair $j \neq s$ there exists a trivial prime $\nu$ with respect to both $\overline{\rho_1}$ and $\overline{\rho_2}$ such that $0 \neq h_j(\tau_\nu) \in \mathbb F_p \cdot X_j$ and $h_s(\tau_\nu) = 0$. Then, their fixed fields $K_{h_1},\dots,K_{h_n}$ are strongly linearly disjoint over $K$, i.e. if $K_h$ is the composite of them all we have

$$\Gal(K_h/K) \simeq \prod_{j=1}^n \Gal(K_{h_j}/K).$$

\vspace{2 mm}

Moreover, $K_h$ is linearly disjoint from $K_\infty$ over $K$.

\end{lem}
\begin{proof}

    This lemma is a slight generalization of \cite[Lemma $3.7$]{FKP22}. We will prove this claim by induction on $n$. When $n=1$ there is nothing to prove. Hence, we assume that $n \ge 2$ and the claim holds for all smaller values of $n$. Let $K_{(\neq n)}$ be the composite of $K_{h_1},\dots,K_{h_{n-1}}$. We will first show that $K_{(\neq n)}$ and $K_{h_n}$ are linearly disjoint over $K$. As both extensions are Galois over $K$ it is enough to show that their intersection is equal to $K$. Let $L \coloneqq K_{h_n} \cap K_{(\neq n)}$ and suppose that $L \neq K$. As $L$ is a Galois extension of $K$ we get that $\Gal(L/K)$ is a non-zero quotient of $\Gal(K_{(\neq n)}/K)$. As $K_{(\neq n)}$ is the composite of $n-1$ fields using the inductive hypothesis we get a surjection:

    $$\prod_{j=1}^{n-1} \Gal(K_{h_j}/K) \simeq \Gal(K_{(\neq n)}/K) \twoheadrightarrow \Gal(L/K).$$

    \vspace{2 mm}

    From this we can find $j \le n-1$ such that the map $\Gal(K_{h_j}/K) \to \Gal(L/K)$ is non-zero. Explicitly this map is given by sending $\sigma \in \Gal(K_{h_j}/K)$ to $\restr{\widetilde{\sigma}}{L}$, where $\tilde{\sigma}$ is the element of $\Gal(K_{(\neq n)}/K)$ that acts as $\sigma$ on $K_{h_j}$ and as the identity on $K_{h_s}$ for $s \neq j$. Moreover, as all the fixed fields are Galois extensions of $F$ this map will be a $\mathbb F_p[\Gamma_F]$-module morphism where the action of $\Gamma_F$ is given by conjugation on both Galois groups.

    Now, by the hypothesis there exists a trivial prime $\nu$ such that $0 \neq h_j(\tau_\nu) \in \mathbb F_p \cdot X_j$ and $h_n(\tau_\nu) = 0$. This means that the restriction of $\tau_\nu$ to $K_{h_j}$ generates $\Gal(K_{h_j}/K)$ as a $\mathbb F_p[\Gamma_F]$-module. Indeed, as $\mathbb F_p[\Gamma_F]$-modules we have $\Gal(K_{h_j}/K) \simeq \im h_j \subseteq \mathbb F_p[\overline{\rho_{i_j}}] \cdot X_j$. As $\im h_j$ contains a non-zero multiple of $X_j$ it has to equal $\mathbb F_p[\overline{\rho_{i_j}}] \cdot X_j$. In particular, $\restr{\tau_\nu}L$ will generate the image of the $\mathbb F_p[\Gamma_F]$-morphism $\Gal(K_{h_j}/K) \to \Gal(L/K)$. As this map is non-zero we have $\restr{\tau_\nu}L \neq \mathrm{id}$. On the other side, by our choice of $\nu$ we have that $h_n(\tau_\nu) =0$, hence $\restr{\tau_\nu}{K_{h_n}} = \mathrm{id}$. As $L \subseteq K_{h_n}$ this means that $\restr{\tau_\nu}L = \mathrm{id}$, which is a contradiction. Therefore, $L = K$ and so $K_{(\neq n)}$ and $K_{h_n}$ are linearly disjoint over $K$. Using the inductive hypothesis 

    $$\Gal(K_h/K) \simeq \Gal(K_{(\neq n)}/K) \times \Gal(K_{h_n}/K) \simeq \prod_{j=1}^{n-1} \Gal(K_{h_j}/K) \times \Gal(K_{h_n}/K),$$

    \vspace{2 mm}

    \noindent which tells us that all the fixed fields are strongly linearly disjoint over $K$.

    The second part of the lemma follows in a similar manner: Let $L = K_h \cap K_\infty$. If $L \neq K$, as above $L$ will be ramified at some trivial prime $\nu$. This is a contradiction since $L$ could only ramify at primes above $p$ by the virtue of being a subfield of $K_\infty$. 
    
\end{proof}

\subsection{Lifting modulo \texorpdfstring{$\varpi^2$}{ϖ\textasciicircum2}}

For each $\nu$ in $T$ and $i=1,2$ we fix local target lifts $\lambda_{i,\nu}:\Gamma_{F_\nu} \to G_i(\mathcal O/\varpi^2)$ with multiplier $\mu_i$ of $\restr{\overline{\rho_i}}{\Gamma_{F_\nu}}$. For the time being we let $\lambda_{i,\nu}$ be just any multiplier $\mu_i$ lifts of $\restr{\overline{\rho_i}}{\Gamma_{F_\nu}}$. We will use these lifts to control the local behavior of lifts of $\overline{\rho_i}$ at each prime in $T$. We will specify them more precisely in Theorem \ref{LiftingModn}. By Remark \ref{VanishingofSh} we know that $\Sh_T^2(\Gamma_{F,T},\overline{\rho_i}(\mathfrak g_i^{\der})) = 0$. Using the existence of local lifts with multiplier $\mu_i$ at each $\nu \in T$ we can produce a lift $\rho'_{i,2}:\Gamma_{F,T} \to G_i(\mathcal O/\varpi^2)$ with multiplier $\mu_i$. The local restriction $\restr{\rho'_{i,2}}{\Gamma_{F_\nu}}$ differs from the target lift $\lambda_{i,\nu} \pmod{\varpi^2}$ by a class $z_{i,\nu} \in H^1(\Gamma_{F_\nu},\overline{\rho_i}(\mathfrak g_i^{\der}))$. More precisely, $\restr{\exp(\varpi z_{i,\nu})\rho_{i,2}'}{\Gamma_{F_\nu}}$ will be strongly equivalent to $\lambda_{i,\nu} \pmod{\varpi^2}$. Therefore, our goal is to find global cohomology classes $h_i$ whose restriction to places in $T$ will be equal to $z_{i,T} \coloneqq (z_{i,\nu})_{\nu \in T} \in \bigoplus_{\nu \in T} H^1(\Gamma_{F_\nu},\overline{\rho_i}(\mathfrak g_i^{\der}))$.

In general, in order to find such cohomology classes we will need to allow extra ramification. We would like the modified lifts $\rho_{i,2} \coloneqq \exp(\varpi h_i)\rho_{i,2}'$ to both be ramified at the same extra primes. In particular, if $\rho_{1,2}$ is ramified at some extra trivial primes the same should be true for $\rho_{2,2}$ and vice versa. Moreover, we like for the ramification in both cases to be ``Steinberg-like". We make that more precise in the next proposition. 

\begin{prop}[cf. {\cite[Proposition $3.8$]{FKP22}}]
\label{DoublingMethod}

There is a finite indexing set $N$, and there is for each $n \in N$ a positive upper-density set $\mathfrak l_n$ of trivial primes in $F$ with respect to both $\overline{\rho_1}$ and $\overline{\rho_2}$, and a pair $(T_{i,n},\alpha_{i,n})$ of a split maximal torus $T_{i,n}$ of $G_i^0$ and a root $\alpha_{i,n} \in \Phi(G^0_i,T_{i,n})$ for $i=1,2$ with the following properties: Fix two $2|N|$-tuples $(A_{1,n},A_{1,n}')_{n \in N}$ and $(A_{2,n},A_{2,n}')_{n \in N}$ of elements in $\widehat{G_1^{\der}}(\mathcal O/\varpi^2)$ and $\widehat{G_2^{\der}}(\mathcal O/\varpi^2)$, respectively. Then there is a $2|N|$-tuple of trivial primes $Q = (\nu_n,\nu_n')_{n \in N}$ disjoint from $T \cup \{t_1,t_2\}$ such that $\{\nu_n,\nu_n'\} \subseteq \mathfrak l_n$ for all $n \in N$, and for $i=1,2$ a class $h_i \in H^1(\Gamma_{F,T \cup Q},\overline{\rho_i}(\mathfrak g_i^{\der}))$ such that

\begin{itemize}
    \item $\restr{h_i}{T} = z_{i,T}$ for $i=1,2$.
    \item $\rho_{i,2}(\tau_{\nu_n}) = u_{\alpha_{i,n}}(X_{i,n}) \pmod{\varpi^2}$ for some $\alpha_{i,n}$-root vector $X_{i,n}$ which is non-zero modulo $\varpi^2$ and likewise $\rho_{i,2}(\tau_{\nu'_n}) = u_{\alpha_{i,n}}(X_{i,n})$; and such that 

    $$\rho_{i,2}(\sigma_{\nu_n}) = A_{i,n} \cdot z_{i,n} \quad \quad \rho_{i,2}(\sigma_{\nu_n'}) = A_{i,n}' \cdot z_{i,n}'$$

    \vspace{2 mm}

    \noindent for $i=1,2$, where $z_{i,n}, z_{i,n}' \in Z_{G_i^0}(\mathcal O/\varpi^2) \cap \widehat{G_i}(\mathcal O/\varpi^2)$ and are determined by $\mu_i(\sigma_{\nu_n})$ and $\mu_i(\sigma_{\nu_n'})$, respectively.

\end{itemize}

\end{prop}

\begin{proof}

    By Lemma \ref{RootSpaceSpan} we can find a finite indexing set $N_{1,\mathrm{span}}$ and for each $n \in N_{1,\mathrm{span}}$ a non-zero element $X_{\alpha_{1,n}} \in \mathfrak g_1^{\der} \otimes k$ which is an image of a root vector with respect to a torus $T_{1,n}$ and a root $\alpha_{1,n} \in \Phi(G_1^0,T_{1,n})$ such that

    $$\sum_{n \in N_{1,\mathrm{span}}} \mathbb F_p[\overline{\rho_1}] \cdot X_{\alpha_{1,n}}= \mathfrak{g}_1^{\der} \otimes k.$$

    \vspace{2 mm}

    Then, by the second part of Proposition \ref{cokernelgenerators} for each $n \in N_{1,\mathrm{span}}$ by taking $L=K$ and $c=D+1$ we can find a positive upper-density set $\mathfrak l_{1,n}$ of trivial primes, a non-trivial $q_{1,n} \in (\mathbb Z/p^{D+1})^\times$ that is trivial modulo $p^D$ with all $\nu \in \mathfrak l_{1,n}$ satisfying $N(\nu) \equiv q_{1,n} \pmod{p^{D+1}}$, and for each $\nu \in \mathfrak l_{1,n}$ a class $h_1^{(\nu)} \in H^1(\Gamma_{F,T \cup \nu},\overline{\rho_1}(\mathfrak g_1^{\der}))$ such that the restriction $\restr{h_1^{(\nu)}}{T} = Y_{1,n}$ is independent of $\nu$, $h_1^{(\nu)}(\tau_\nu) = h_1^{(\nu)}(\tau_{t_0}) = X_{\alpha_{1,n}}$, and $h_1^{(\nu)}$ is the image of an $\mathbb F_p[\overline{\rho_1}] \cdot X_{\alpha_{1,n}}$-valued cocycle; and also a class $h_2^{(\nu,t_2)} \in H^1(\Gamma_{F,T \cup \{\nu,t_2\}},\overline{\rho_2}(\mathfrak g_2^{\der}))$ such that the restriction $\restr{h_2^{(\nu,t_2)}}{T} = Y_{2,n}$ is independent of $\nu$, $h_2^{(\nu,t_2)}(\tau_\nu) = h_2^{(\nu,t_2)}(\tau_{t_0}) = h_2^{(\nu,t_2)}(\tau_{t_2}) = X$, where $X$ is a non-zero element of $\mathfrak g_2^{\der} \otimes k$, which is an image of a root vector modulo $\varpi$ and is independent of the choice of a prime $\nu \in \mathfrak l_n$.

    On the other hand, by the first part of Proposition \ref{cokernelgenerators} with $L=K$ and $c=D+1$ we produce a finite indexing set $N_{1,\coker}$ and a finite set $\{Y_{1,n}\}_{n \in N_{1,\coker}} \in \bigoplus_{w \in T} H^1(\Gamma_{F_w},\overline{\rho_1}(\mathfrak g_1^{\der}))$ whose images form an $\mathbb F_p$-basis of $\coker(\Psi_{1,T})$, and for each $n \in N_{1,\coker}$ a non-zero element $X_{\alpha_{1,n}} \in \mathfrak g_1^{\der} \otimes k$ which is a modulo $\varpi$ image of a root vector with respect to a split maximal torus $T_{1,n}$ and a root $\alpha_{1,n} \in \Phi(G^0_1,T_{1,n})$, a non-trivial $q_{1,n} \in (\mathbb Z/p^{D+1})^\times$ that is trivial modulo $p^D$, a positive upper-density set $\mathfrak l_{1,n}$ of trivial primes with all $\nu \in \mathfrak l_{1,n}$ satisfying $N(\nu) \equiv q_{1,n} \pmod{p^{D+1}}$, and for each $\nu \in \mathfrak l_{1,n}$ a class $h_1^{(\nu)} \in H^1(\Gamma_{F,T \cup \nu},\overline{\rho_1}(\mathfrak g_1^{\der}))$ such that $\restr{h_1^{(\nu)}}{T} = Y_{1,n}$, $h_1^{(\nu)}(\tau_\nu) = h_1^{(\nu)}(\tau_{t_0}) = X_{\alpha_{1,n}}$, and $h_1^{(\nu)}$ is the image of an $\mathbb F_p[\overline{\rho_1}] \cdot X_{\alpha_{1,n}}$ valued cocycle; and also a class $h_2^{(\nu,t_2)} \in H^1(\Gamma_{F,T \cup \{\nu,t_2\}},\overline{\rho_2}(\mathfrak g^{\der}))$ such that the restriction $\restr{h_2^{(\nu,t_2)}}{T} = Y_{2,n}$ is independent of $\nu$, $h_2^{(\nu,t_2)}(\tau_{\nu}) = h_2^{(\nu,t_2)}(\tau_{t_0}) = h_2^{(\nu,t_2)}(\tau_{t_2}) = X$, where $X$ is a non-zero element of $\mathfrak g_2^{\der} \otimes k$, which is an image of a root vector modulo $\varpi$ and is independent of the choice of a prime $\nu \in \mathfrak l_n$.

    \vspace{2 mm}

    Now, we consider the class $z_{1,T}' \coloneqq z_{1,T} - \sum_{n \in N_{1,\mathrm{span}}} Y_{1,n}$. Since $\{Y_{1,n}\}_{n \in N_{1,\coker}}$ form an $\mathbb F_p$-basis of $\coker(\Psi_{1,T})$ we can find $c_{1,n} \in \mathbb F_p$ and $h_1^{\mathrm{old}} \in H^1(\Gamma_{F,T},\overline{\rho_1}(\mathfrak g_1^{\der}))$ such that $z_{1,T}' = \restr{h_1^{\mathrm{old}}}{T} + \sum_{n \in N_{1,\coker}} c_{1,n}Y_{1,n}$. In particular

    $$z_{1,T} = \restr{h_1^{\mathrm{old}}}{T} + \sum_{n \in N_{1,\mathrm{span}}} Y_{1,n} + \sum_{n \in N_{1,\coker}} c_{1,n}Y_{1,n}.$$

    \vspace{2 mm}

    We now set $N_1 \subseteq N_{1,\mathrm{span}} \sqcup N_{1,\coker}$, which is obtained by discarding the indices $n \in N_{1,\coker}$ for which $c_{1,n} = 0$. We emphasize that $N_1$ contains $N_{1,\mathrm{span}}$. By construction, for each $n \in N_1$ and $\nu_{1,n} \in \mathfrak l_{1,n}$ we have $\restr{h_1^{(\nu_{1,n})}}{T} = Y_{1,n}$. Thus, if we rescale $h_1^{(\nu_{1,n})}$ for each $n \in N_1 \cap N_{1,\coker}$ by $c_{1,n}$ we get

    $$z_{1,T} = \restr{h_1^{\mathrm{old}}}{T} + \sum_{n \in N_1} \restr{h_1^{(\nu_{1,n})}}{T}$$

    \vspace{2 mm}

    \noindent for all tuples $\underline{\nu}_1 = (\nu_{1,n})_{n \in N_1} \in \prod_{n \in N_1} \mathfrak l_{1,n}$. Running the same construction with the roles of $\overline{\rho_1}$ and $\overline{\rho_2}$ reversed we find an indexing set $N_2 \subseteq N_{2,\mathrm{span}} \sqcup N_{2,\coker}$, containing $N_{2,\mathrm{span}}$ such that

    $$z_{2,T} = \restr{h_2^{\mathrm{old}}}{T} + \sum_{n \in N_2} \restr{h_1^{(\nu_{2,n})}}{T}$$

    \vspace{2 mm}

    \noindent for all tuples $\underline{\nu}_2 = (\nu_{2,n})_{n \in N_2} \in \prod_{n \in N_2} \mathfrak l_{2,n}$. Then, for any $\underline{\nu}_1,\underline{\nu}'_1 \in \prod_{n \in N_1} \mathfrak l_{1,n}$ and $\underline{\nu}_2,\underline{\nu}'_2 \in \prod_{n \in N_2} \mathfrak l_{2,n}$ we consider classes

    $$h_1 = h_1^{\mathrm{old}} + 2 \sum_{n \in N_1} h_1^{(\nu_{1,n})} - \sum_{n \in N_1} h_1^{(\nu'_{1,n})} + \sum_{n \in N_2} h_1^{(\nu_{2,n},t_1)} - \sum_{n \in N_2} h_1^{(\nu_{2,n}',t_1)} \in H^1(\Gamma_{F,T \cup \{\underline{\nu}_1,\underline{\nu}'_1,\underline{\nu}_2,\underline{\nu}'_2,t_1\}},\overline{\rho_1}(\mathfrak g_1^{\der})),$$

    \vspace{2 mm}

    $$h_2 = h_2^{\mathrm{old}} + 2 \sum_{n \in N_2} h_2^{(\nu_{2,n})} - \sum_{n \in N_2} h_2^{(\nu'_{2,n})} + \sum_{n \in N_1} h_2^{(\nu_{1,n},t_2)} - \sum_{n \in N_1} h_2^{(\nu_{1,n}',t_2)} \in H^1(\Gamma_{F,T \cup \{\underline{\nu}_1,\underline{\nu}'_1,\underline{\nu}_2,\underline{\nu}'_2,t_2\}},\overline{\rho_2}(\mathfrak g_2^{\der})).$$

    \vspace{2 mm}

    We will assume that the tuples $\underline{\nu}_1,\underline{\nu}_1',\underline{\nu}_2,\underline{\nu}_2'$ that do not share any prime and each prime appears at most once within one of the tuples. This isn't a big restriction since we will need to exclude at most finitely many primes from each upper-density set.

    By construction each cohomology class and its corresponding primed version have the same restriction to primes in $T$. Therefore

    $$\restr{h_1}{T} = \restr{h_1^{\mathrm{old}}}{T} + \sum_{n \in N_1} \restr{h_1^{(\nu_{1,n})}}{T} = z_{1,T} \quad \quad \text{and} \quad \quad \restr{h_2}{T} = \restr{h_2^{\mathrm{old}}}{T} + \sum_{n \in N_2} \restr{h_2^{(\nu_{2,n})}}{T} = z_{2,T}.$$

    \vspace{2 mm}

    In particular, both cohomology classes have the desired restriction to primes in $T$. We also note that for each $n \in N_2$ by construction $h_1^{(\nu_{2,n},t_1)}(\tau_{t_1}) = h_1^{(\nu_{2,n}',t_1)}(\tau_{t_1})$. Therefore, the image of $\tau_{t_1}$ under the terms in the last two sums in the definition of $h_1$ are the same and cancel out. As no other term in the definition is ramified at $t_1$ we actually get that $h_1(\tau_{t_1}) = 0$ and in particular $h_1 \in H^1(\Gamma_{F,T \cup \{\underline{\nu}_1,\underline{\nu}'_1,\underline{\nu}_2,\underline{\nu}'_2\}},\overline{\rho_1}(\mathfrak g_1^{\der}))$. Similarly, $h_2 \in H^1(\Gamma_{F,T \cup \{\underline{\nu}_1,\underline{\nu}'_1,\underline{\nu}_2,\underline{\nu}'_2\}},\overline{\rho_2}(\mathfrak g_2^{\der})))$.

    On the other hand, for $\nu_{i,n}$ in one of the tuples we have that $h_1(\tau_{\nu_{i,n}})$ will be a non-zero image modulo $\varpi$ of a certain root vector in $\mathfrak g_1^{\der}$, as exactly one term in the sums is ramified at $\nu_{1,n}$ in that root space. The analogous result for $h_2$ holds, as well. This gives us that both cohomology classes are ramified at primes in $\underline{\nu}_1,\underline{\nu}_1',\underline{\nu}_2,\underline{\nu}_2'$ and moreover in the prescribed root space. This means that $\rho_{1,2}$ and $\rho_{2,2}$ have the desired inertial behavior at the new primes, since the initial lifts $\rho'_{1,2}$ and $\rho_{2,2}'$ are not ramified at any of them. 
    
    We set $N \coloneqq N_1 \sqcup N_2$. Only at this point we fix two $2|N|$-tuples $(A_{1,n},A_{1,n}')_{n \in N}$ and $(A_{2,n},A_{2,n}')_{n \in N}$ of elements in $\widehat{G_1^{\der}}(\mathcal O/\varpi^2)$ and $\widehat{G_2^{\der}}(\mathcal O/\varpi^2)$, respectively. It remains to show that there is a pair of $|N|$-tuples $\underline{\nu} \coloneqq (\underline{\nu}_1,\underline{\nu}_2)$ and $\underline{\nu'} \coloneqq (\underline{\nu}_1',\underline{\nu}_2')$ so that for each $n \in N$ we have that $\rho_{i,2}(\sigma_{\nu_n}) = A_{i,n}$ and $\rho_{i,2}(\sigma_{\nu_n'}) = A_{i,n}'$ modulo the center. We remark that the elements $A_{i,n}$ and $A_{i,n}'$ have to be trivial modulo $\varpi$ because each prime in the tuples is trivial with respect to both $\overline{\rho_1}$ and $\overline{\rho_2}$, so the images of its Frobenius under both of them are trivial. 

    We recall that by Proposition \ref{cokernelgenerators} each prime $\nu_{i,n}$ comes with a choice of a decomposition group. This determines a unique prime $\nu_{i,n,K}$ of $K$ above $\nu_{i,n}$. We let $\mathcal C_{i,n,K}$ and $\mathfrak l_{i,n,K}$ denote the sets of such primes. We note that as both $\mathcal C_{i,n}$ and $\mathfrak l_{i,n}$ consist of primes that split in $K$ the sets $\mathcal C_{i,n,K}$ will still have positive Dirichlet density, while the sets $\mathfrak l_{i,n,K}$ will have positive upper-density. It turns out that it is convenient to work with primes and densities in $K$, so for notational simplicity we drop the subscript whenever it is clear we are working with primes and densities in $K$. The choice of decomposition group at each such prime that comes from Proposition \ref{cokernelgenerators} was specified only in a finite extension of $K$. In what follows we will apply the Chebotarev density theory in extensions of $K$ which will require specification of the image of $\sigma_{\nu_{i,n}}$ in that extension. However, all these extensions will be abelian and hence the image of $\sigma_{\nu_{i,n}}$ in them will depend only on $\nu_{i,n,K}$, which comes from the choice of decomposition group in Proposition \ref{cokernelgenerators}.
    
    The proof will rely on the notion of Dirichlet densities of $N$-tuples of primes, so we recall their definitions. For a number field $L$, we define the Dirichlet density of a subset $P$ of the set of all $N$-tuples of primes in $L$ by
    
    $$\delta(P) = \lim_{s \to 1^+} \frac{\sum_{\underline{\nu} \in P} N(\underline{\nu})^{-s}}{\sum_{\text{all }\underline{\nu}} N(\underline{\nu})^{-s}},$$
    
    \vspace{2 mm}
    
    \noindent assuming this limit exists, where $N(\underline{\nu}) = \prod_{n \in N} N(\nu_n)$. We note that if $P = \prod_{i=1}^n P_n$ is a product of sets of primes $P_n$, where each $P_n$ has Dirichlet density $\delta(P_n)$, then $\delta(P)$ exists and is equal to $\prod_{i=1}^n \delta(P_n)$. We analogously define the upper (Dirichlet) density $\delta^+(P)$ of a set of $N$-tuples of primes by replacing $\lim_{s \to 1^+}$ in the definition for Dirichlet density with $\limsup_{s \to 1^+}$.
    
    Now, by Lemma \ref{DensityLemma} we can find a positive upper density subset $\mathfrak l \subseteq \prod_{n \in N_1} \mathfrak l_{1,n,K} \times \prod_{n \in N_2} \mathfrak l_{2,n,K}$ consisting of $|N|$-tuples $\underline{\nu} = (\underline{\nu}_1,\underline{\nu}_2)$ of primes for which the following are independent of the choice of the tuple in $\mathfrak l$:

    \vspace{1 mm}

    \begin{itemize}
        \item $h_1^{(\mathrm{old})}(\sigma_{\nu_{1,m}})$ and $h_2^{\mathrm{old}}(\sigma_{\nu_{1,m}})$ for all $m \in N_1$.
        \vspace{1 mm}
        \item $\sum_{n \in N_1} h_1^{(\nu_{1,n})}(\sigma_{\nu_{1,m}})$ and $\sum_{n \in N_2} h_1^{(\nu_{2,n},t_1)}(\sigma_{\nu_{1,m}})$ for all $m \in N_1$.
        \vspace{1 mm}
        \item $\sum_{n \in N_2} h_2^{(\nu_{2,n})}(\sigma_{\nu_{1,m}})$ and $\sum _{n \in N_1} h_2^{(\nu_{1,n},t_2)}(\sigma_{\nu_{1,m}})$ for all $n \in N_1$.
        \vspace{2 mm}
        \item $\rho_{1,2}'(\sigma_{\nu_{1,m}}) \pmod{Z_{G^0_1}}$ and $\rho_{2,2}'(\sigma_{\nu_{1,m}}) \pmod{Z_{G^0_2}}$ for all $m \in N_1$.

        \vspace{4 mm}

        \item $h_1^{(\mathrm{old})}(\sigma_{\nu_{2,m}})$ and $h_2^{\mathrm{old}}(\sigma_{\nu_{2,m}})$ for all $m \in N_2$.
        \vspace{1 mm}
        \item $\sum_{n \in N_1} h_1^{(\nu_{1,n})}(\sigma_{\nu_{2,m}})$ and $\sum_{n \in N_2} h_1^{(\nu_{2,n},t_1)}(\sigma_{\nu_{2,m}})$ for all $m \in N_2$.
        \vspace{1 mm}
        \item $\sum_{n \in N_2} h_2^{(\nu_{2,n})}(\sigma_{\nu_{2,m}})$ and $\sum _{n \in N_1} h_2^{(\nu_{1,n},t_2)}(\sigma_{\nu_{2,m}})$ for all $n \in N_2$.
        \vspace{2 mm}
        \item $\rho_{1,2}'(\sigma_{\nu_{2,m}}) \pmod{Z_{G^0_1}}$ and $\rho_{2,2}'(\sigma_{\nu_{2,m}}) \pmod{Z_{G^0_2}}$ for all $m \in N_2$.
        
    \end{itemize}

    \vspace{2 mm}

    This is all possible since the quantities in question take only finitely many values. We note that the last bullet point for both $\underline{\nu}_1$ and $\underline{\nu}_2$ ensures that $\rho_{i,2}(\sigma_{\nu_n})$ modulo the center is determined only by $h_i(\sigma_{\nu_n})$ for each $n \in N$. Therefore, we have reduced the proof to finding $\underline{\nu},\underline{\nu}' \in \mathfrak l$ that satisfy the equations:
    
    \begin{equation}
    	h_i(\sigma_{\nu_n}) = C_{i,n} \quad \quad \text{and} \quad \quad h_i(\sigma_{\nu_n'}) = C_{i,n}' \tag{$\star$}
    \end{equation}
    
    \vspace{2 mm}
    
    \noindent for each $n \in N$ and $i=1,2$, where $C_{i,n}$ and $C_{i,n}'$ are uniquely determined by the $|N|$-tuples $(A_{1,n},A_{1,n}')_{n \in N}$ and $(A_{2,n},A_{2,n}')_{n \in N}$. On the other side, the first three bullet points for both $\underline{\nu}_1$ and $\underline{\nu}_2$ will allow us to get a better control over the images $h_i(\sigma_{\nu_n})$ and $h_i(\sigma_{\nu_n'})$ for each prime $n \in N$ and $i=1,2$. In particular, they will depend on fewer terms that appear in the definition of the cohomology classes $h_1$ and $h_2$. More precisely, if $\underline{\nu} = (\underline{\nu}_1,\underline{\nu}_2)$, $\underline{\nu}' = (\underline{\nu}_1',\underline{\nu}_2') \in \mathfrak l$ we obtain the following $8$ relations:

    \begin{equation}
    \text{For all } m \in N_1 \quad  h_1(\sigma_{\nu_{1,m}}) \text{ is determined by  } \sum_{n \in N_1} h_1^{(\nu'_{1,n})}(\sigma_{\nu_{1,m}}) + \sum_{n \in N_2} h_1^{(\nu'_{2,n},t_1)}(\sigma_{\nu_{1,m}}). \quad \quad \quad \quad \quad  \tag{1}
    \end{equation}
    
    \begin{equation}
        \text{For all } m \in N_1 \quad  h_2(\sigma_{\nu_{1,m}}) \text{ is determined by  } \sum_{n \in N_1} h_2^{(\nu'_{1,n},t_2)}(\sigma_{\nu_{1,m}}) + \sum_{n \in N_2} h_2^{(\nu'_{2,n})}(\sigma_{\nu_{1,m}}). \quad \quad \quad \quad \quad  \tag{2}
    \end{equation}
    
    \begin{equation}
       \text{For all } m \in N_2 \quad  h_1(\sigma_{\nu_{2,m}}) \text{ is determined by  } \sum_{n \in N_1} h_1^{(\nu'_{1,n})}(\sigma_{\nu_{2,m}}) + \sum_{n \in N_2} h_1^{(\nu'_{2,n},t_1)}(\sigma_{\nu_{2,m}}). \quad \quad \quad \quad \quad  \tag{3} 
    \end{equation}
    
    \begin{equation}
    \text{For all } m \in N_2 \quad  h_2(\sigma_{\nu_{2,m}}) \text{ is determined by  } \sum_{n \in N_1} h_2^{(\nu'_{1,n},t_2)}(\sigma_{\nu_{2,m}}) + \sum_{n \in N_2} h_2^{(\nu'_{2,n})}(\sigma_{\nu_{2,m}}). \quad \quad \quad \quad \quad  \tag{4}
    \end{equation}
    
    \begin{equation}
    \text{For all } m \in N_1 \quad  h_1(\sigma_{\nu'_{1,m}}) \text{ is determined by  } 2\sum_{n \in N_1} h_1^{(\nu_{1,n})}(\sigma_{\nu'_{1,m}}) + \sum_{n \in N_2} h_1^{(\nu_{2,n},t_1)}(\sigma_{\nu'_{1,m}}). \quad \quad \quad \quad \tag{5}
    \end{equation}
    
    \begin{equation}
    \text{For all } m \in N_1 \quad  h_2(\sigma_{\nu'_{1,m}}) \text{ is determined by  } \sum_{n \in N_1} h_2^{(\nu_{1,n},t_2)}(\sigma_{\nu'_{1,m}}) + 2 \sum_{n \in N_2} h_2^{(\nu_{2,n})}(\sigma_{\nu'_{1,m}}). \quad \quad \quad \quad   \tag{6}
    \end{equation}
    
    \begin{equation}
    \text{For all } m \in N_2 \quad  h_1(\sigma_{\nu'_{2,m}}) \text{ is determined by  } 2\sum_{n \in N_1} h_1^{(\nu_{1,n})}(\sigma_{\nu'_{2,m}}) + \sum_{n \in N_2} h_1^{(\nu_{2,n},t_1)}(\sigma_{\nu'_{2,m}}). \quad \quad \quad \quad   \tag{7}
    \end{equation}
    
    \begin{equation}
    \text{For all } m \in N_2 \quad  h_2(\sigma_{\nu_{2,m}'}) \text{ is determined by  } \sum_{n \in N_1} h_2^{(\nu_{1,n},t_2)}(\sigma_{\nu'_{2,m}}) + 2\sum_{n \in N_2} h_2^{(\nu_{2,n})}(\sigma_{\nu'_{2,m}}). \quad \quad \quad \quad   \tag{8}
    \end{equation}
    
    \vspace{2 mm} 

    We now fix $\underline{\nu} = (\underline{\nu}_1,\underline{\nu}_2) \in \mathfrak l$. Using these $8$ relations we will define Chebotarev conditions over $K$ (depending on $\underline{\nu}$) for each prime in $\underline{\nu}'$, so that if primes in $\underline{\nu'}$ are chosen to satisfy those Chebotarev conditions then the pair $\underline{\nu}$ and $\underline{\nu}'$ will solve the equations $(\star)$.
    
    Since $N_{1,\mathrm{span}} \subseteq N_1$ we get that any element of $\mathfrak g_1^{\der}$ can be written as a linear combination of values of $h_1^{(\nu_{1,n})}$ for varying $n \in N_1$. Therefore, from relation $(5)$ for a fixed $m \in N_1$ we can prescribe the value of $h_1(\sigma_{\nu_{1,m}'})$ by Chebotarev conditions on $\nu_{1,m}'$ in $K_{h_1^{(\nu_{1,n})}}$ for all $n \in N_1$ and $K_{h_1^{(\nu_{2,n},t_1)}}$ for all $n \in N_2$. Similarly, as $N_{2,\mathrm{span}} \subseteq N_2$ using relation $(6)$ for a fixed $m \in N_1$ we can prescribe the value of $h_2(\sigma_{\nu_{1,m}'})$ by Chebotarev conditions on $\nu_{1,m}'$ in $K_{h_2^{(\nu_{1,n},t_2)}}$ for all $n \in N_1$ and $K_{h_2^{(\nu_{2,n})}}$ for all $n \in N_2$. Analogously, for a fixed $m \in N_2$ from relation $(7)$ we can prescribe the value of $h_1(\sigma_{\nu_{2,m}'})$ by Chebotarev conditions on $\nu_{2,m}'$ in $K_{h_1^{(\nu_{1,n})}}$ for all $n \in N_1$ and $K_{h_1^{(\nu_{2,n},t_1)}}$ for all $n \in N_2$, while by relation $(8)$ we can prescribe the value of $h_2(\sigma_{\nu_{2,m}'})$ by Chebotarev conditions on $\nu_{2,m}'$ in $K_{h_2^{(\nu_{1,n},t_2)}}$ for all $n \in N_1$ and $K_{h_2^{(\nu_{2,n})}}$ for all $n \in N_2$.

    On the other hand, from our earlier construction $\eta_{1,b}^{(\nu_{1,m})}(\tau_{\nu_{1,m}}) = e_{1,b}^*$ and so $\{n_{1,b}^{(\nu_{1,m})}(\tau_{\nu_{1,m}})\}_{b \in B_1}$ is a $k$-basis of $\overline{\rho_1}(\mathfrak g_1^{\der})^*$. Therefore, to prescribe $h_1(\sigma_{\nu_{1,m}})$ it is enough to prescribe $\langle \eta_{1,b}^{(\nu_{1,m})}(\tau_{\nu_{1,m}}),h_1(\sigma_{\nu_{1,m}})\rangle$ for all $b \in B_1$. So using relation $(1)$ for a fixed $m \in N_1$ the value of $h_1(\sigma_{\nu_{1,m}})$ will be determined by

    \begin{equation}
    \sum_{n \in N_1} \langle \eta_{1,b}^{(\nu_{1,m})}(\tau_{\nu_{1,m}}),h_1^{(\nu_{1,n}')}(\sigma_{\nu_{1,m}})\rangle + \sum_{n \in N_2} \langle \eta_{1,b}^{(\nu_{1,m})}(\tau_{\nu_{1,m}}), h_1^{(\nu_{2,n}',t_1)}(\sigma_{\nu_{1,m}}) \rangle \quad \quad \text{for all } b \in B_1. \tag{$\vardiamond$}
    \end{equation}

    \vspace{2 mm}

    We would prescribe the values of each of those pairings using Chebotarev conditions on $\underline{\nu}_1'$ and $\underline{\nu}_2'$. For a fixed $m \in N_1$ from global duality we have the following exact sequence:

    $$H^1(\Gamma_{F,T \cup \{\nu_{1,m},\nu_{1,n}'\}},\overline{\rho_1}(\mathfrak g_1^{\der})) \longrightarrow \bigoplus_{\nu \in T \cup \{\nu_{1,m},\nu_{1,n}'\}} H^1(\Gamma_{F_\nu},\overline{\rho_1}(\mathfrak g_1^{\der})) \longrightarrow H^1(\Gamma_{F,T \cup \{\nu_{1,m},\nu_{1,n}'\}},\overline{\rho_1}(\mathfrak g_1^{\der})^*)^\vee.$$

    \vspace{2 mm}

    From the exactness we get $\sum_{\nu \in T \cup \{\nu_{1,m},\nu_{1,n}'\}} \langle \eta_{1,b}^{(\nu_{1,m})},h_1^{(\nu_{1,n}')}\rangle_\nu = 0$, which by the computation in \cite[Lemma $3.9$]{FKP21} yields

    $$\langle \eta_{1,b}^{(\nu_{1,m})}(\tau_{\nu_{1,m}}),h_1^{(\nu_{1,n}')}(\sigma_{\nu_{1,m}}) \rangle = \langle \eta_{1,b}^{(\nu_{1,m})}(\sigma_{\nu_{1,n}'}),h_1^{(\nu_{1,n}')}(\tau_{\nu_{1,n}'}) \rangle - \sum_{\nu \in T} \langle \eta_{1,b}^{(\nu_{1,m})},h_1^{(\nu_{1,n}')}\rangle_\nu.$$

    \vspace{2 mm}

    By the construction of the cohomology class $h_1^{(\nu_{1,n}')}$ its restriction to primes in $T$ is independent of the choice of prime. Hence, the sum above is independent of $\nu_{1,n}'$. Also $h_1^{(\nu_{1,n}')}(\tau_{\nu_{1,n}'}) = Z_{1,n}$, which is a non-zero scalar multiple of $X_{\alpha_{1,n}}$. Therefore, to prescribe the value on the left it is enough to prescribe $\langle \eta_{1,b}^{(\nu_{1,m})}(\sigma_{\nu_{1,n}'}),Z_{1,n}\rangle$. In a similar manner, using global duality and keeping in mind that $h_1^{(\nu_{2,n}',t_1)}$ has extra ramification at $t_1$ we obtain

    $$\langle \eta_{1,b}^{(\nu_{1,m})}(\tau_{\nu_{1,m}}),h_1^{(\nu_{2,n}',t_1)}(\sigma_{\nu_{1,m}})\rangle = \langle \eta_{1,b}^{(\nu_{1,m})}(\sigma_{\nu_{2,n}'}),h_1^{(\nu_{2,n}',t_1)}(\tau_{\nu_{2,n}'}) \rangle + \langle \eta_{1,b}^{(\nu_{1,m})}(\sigma_{t_1}),h_1^{(\nu_{2,n}',t_1)}(\tau_{t_1}) \rangle - \sum_{\nu \in T} \langle \eta_{1,b}^{(\nu_{1,m})},h_1^{(\nu_{2,n}',t_1)} \rangle_\nu.$$

    \vspace{2 mm}

    As before the restriction of $h_1^{(\nu_{2.n}',t_1)}$ to primes in $T$ is independent of the choice of $\nu_{2,n}'$, so the same is true for the sum above. By construction $h_1^{(\nu_{2,n}',t_1)}(\tau_{\nu_{2,n}'}) = h_1^{(\nu_{2,n}',t_1)}(\tau_{t_1}) = Z_{2,n}$, a non-zero root vector which is independent of $\nu_{2,n}'$. Thus, to prescribe the value on the left it is enough to prescribe $\langle \eta_{1,b}^{(\nu_{1,m})}(\sigma_{\nu_{2,n}'}),Z_{2,n}\rangle$. Combining these results, going back to the equation $(\vardiamond)$, for each $b \in B_1$ we like to prescribe $\sum_{i=1}^2 \sum_{n \in N_i} \langle \eta_{1,b}^{(\nu_{1,m})}(\sigma_{\nu_{i,n}'}),Z_{i,n} \rangle$.

    Now, for any value in $k$ we can find $x \in \mathfrak g_1^{\der} \otimes k$ such that $\langle e_{1,b}^*,x\rangle$ equals that value. As $N_{1,\mathrm{span}} \subseteq N_1$ we have that $\mathfrak g^{\der}_1 = \sum_{n \in N_1} \mathbb F_p[\overline{\rho_1}] \cdot Z_{1,n}$ and so we can write $x = \sum_{n \in N_1} \sum_{j=1}^{a_n} c_{n,j}\beta_{n,j} \cdot Z_{1,n}$, where $c_{n,j} \in \mathbb F_p$ and $\beta_{n,j} \in \Gamma_{F}$. Then

    $$\langle e_{1,b}^*,x\rangle = \left\langle e_{1,b}^*, \sum_{n \in N_1} \sum_{j=1}^{a_n} c_{n,j}\beta_{n,j} \cdot Z_{1,n}\right\rangle = \sum_{n \in N_1} \left \langle e_{1,b}^*, \sum_{j=1}^{a_n}c_{n,j}\beta_{n,j} \cdot Z_{1,n} \right\rangle = \sum_{n \in N_1} \left\langle \sum_{j=1}^{a_n} c_{n,j}'\beta_{n,j}' \cdot e_{1,b}^*, Z_{1,n} \right \rangle,$$

    \vspace{2 mm}

    \noindent where $c_{n,j}' \in \mathbb F_p$ and $\beta_{n,j}' \in \Gamma_F$. More precisely, from the $\Gamma_F$-action on $e_{1,b}^*$ we get $\beta_{n,j}' = \beta_{n,j}^{-1}$ and \linebreak $c_{n,j}' = c_{n,j}\overline{\kappa}(\beta_{n,j})$. Now, for $n \in N_1$ we ask for $\eta_{1,b}^{(\nu_{1,m})}(\sigma_{\nu_{1,n}'}) = \sum_{j=1}^{a_n}c_{n,j}'\beta_{n,j}' \cdot e_{1,b}^*$. We note that the right-hand side lies in the image $\eta_{1,b}^{(\nu_{1,m})}(\Gamma_K)$. Indeed, as $\nu_{1,m}$ is trivial prime we have 
    
    $$\sum_{j=1}^{a_n} c_{n,j}'\beta_{n,j}' \cdot e_{1,b}^* = \sum_{j=1}^{a_n}c_{n,j}'\beta'_{n,j} \cdot \eta_{1,b}^{(\nu_{1,m})}(\tau_{\nu_{1,m}}) = \eta_{1,b}^{(\nu_{1,m})} \left( \prod_{j=1}^{a_n} \beta_{n,j}'\tau_{\nu_{1,m}}^{c_{n,j}'}\beta_{n,j}'^{-1}\right).$$

    \vspace{2 mm}

    This gives us a Chebotarev condition on $\nu_{1,n}'$ in $K_{\eta_{1,b}^{(\nu_{1,m})}}$ for all $b \in B_1$ and $m \in N_1$. Additionally, for $n \in N_2$ we ask for $\eta_{1,b}^{(\nu_{1,m})}(\sigma_{\nu_{2,n}'}) = 0$, which gives us a Chebotarev condition on $\nu_{2,n}'$ in $K_{\eta_{1,b}^{(\nu_{1,m})}}$ for all $b \in B_1$ and $m \in N_1$. Repeating the same argument, relation $(2)$ yields Chebotarev conditions on $\nu_{i,n}'$ in $K_{\eta_{2,b}^{(\nu_{1,m})}}$ for all $b \in B_2, m \in N_1$. Relation $(3)$ gives us a Chebotarev conditions on $\nu_{i,n}'$ in $K_{\eta_{1,b}^{(\nu_{2,m})}}$ for all $b \in B_1$ and $m \in N_2$. Finally, relation $(4)$ gives us Chebotarev conditions on $\nu_{i,n}'$ in $K_{\eta_{2,b}^{(\nu_{2,m})}}$ for all $b \in B_2$ and $m \in N_2$.

    We remark that all these Chebotarev conditions on $\nu_{i,n}'$ coming from the $8$ relations lie in a strongly linearly disjoint extensions over $K$ by Lemma \ref{LinearDisjointnessLemma}. Indeed, for any pair of cocycles we have the desired inertial behavior. For example, all cocycles $h$ are ramified at $t_0$, while neither of the cocycles $\eta$ is. Moreover, each of the cocycles $\eta$ is ramified at a single auxiliary prime $t_b$, while on the other side the cocycles $h$ will either be ramified at a different prime from the upper-density sets mentioned above or one will be ramified at an auxiliary prime $t_1$ or $t_2$, while the other one will not. Therefore, all the Chebotarev conditions are compatible and we can write them as a single Chebotarev condition in their composite, which we denote by $L_{(\underline{\nu}_1,\underline{\nu}_2)}$. We label by $\mathcal C_{(\underline{\nu}_1,\underline{\nu}_2)}$ the tuples of primes $(\underline{\nu}_1',\underline{\nu}_2')$ produced by these Chebotarev conditions in $L_{(\underline{\nu}_1,\underline{\nu}_2)}$.

    Summarizing, we get that if $(\underline{\nu}_1,\underline{\nu}_2) \in \mathfrak l$ and $(\underline{\nu}_1',\underline{\nu}_2') \in \mathfrak l \cap \mathcal C_{(\underline{\nu}_1,\underline{\nu}_2)}$ then by taking $\underline{\nu} = (\underline{\nu}_1,\underline{\nu}_2)$ and $\underline{\nu}' = (\underline{\nu}_1',\underline{\nu_2')}$ we obtain the desired two $2|N|$-tuples of trivial primes. Therefore, it remains to show that $\mathfrak l \cap \mathcal C_{(\underline{\nu}_1,\underline{\nu}_2)}$ is non-empty for some $(\underline{\nu}_1,\underline{\nu}_2) \in \mathfrak l$. Suppose that no such tuple of trivial primes exists. This means that for any integer $s$ we can find a subset $\{(\underline{\nu}_1,\underline{\nu}_2),\cdots,(\underline{\nu}_{2s-1},\underline{\nu}_{2s})\} \subseteq \mathfrak l$ such that $(\mathfrak l \setminus \mathfrak l_k) \cap \mathcal C_{(\underline{\nu}_{2k-1},\underline{\nu}_{2k})} = \emptyset$ for all $k \le s$. Here $\mathfrak l_k$ consists of the tuples of primes in $\mathfrak l$ that share at least one prime with the tuple $(\underline{\nu}_{2k-1},\underline{\nu}_{2k})$. As remarked above this set has density $0$. Moreover, we assume that no prime repeats and appears in multiple tuples $(\underline{\nu}_{2k-1},\underline{\nu}_{2k})$, which can be easily arranged as $\mathfrak l$ has positive density. Then

    $$\mathfrak l \setminus \bigcup_{k=1}^s \mathfrak l_k \subseteq \bigcap_{k=1}^s \overline{\mathcal C_{(\underline{\nu}_{2k-1},\underline{\nu}_{2k})}},$$

    \vspace{2 mm}

    \noindent where we take complement among the $|N|$-tuples of primes in $K$ that split over $F$. We recall that $\mathcal C_{(\underline{\nu}_{2k-1},\underline{\nu}_{2k})}$ consists of $|N|$-tuple of primes $(\underline{\nu}_1',\underline{\nu}_2')$ with each $\nu_{i,n}'$ coming from a Chebotarev condition in $L_{(\underline{\nu}_{2k-1},\underline{\nu}_{2k})}$. We note that $L_{(\underline{\nu}_{2k-1},\underline{\nu}_{2k})}$ is the composite of fixed fields of cocycles each of which is totally ramified at some prime in the tuple $(\underline{\nu}_{2k-1},\underline{\nu}_{2k})$, while for other such tuples the extensions are a composite of fixed fields of cocycles neither of which is ramified at primes in $(\underline{\nu}_{2k-1},\underline{\nu}_{2k})$. Therefore, by Lemma \ref{LinearDisjointnessLemma} all the fields $L_{(\underline{\nu}_{2k-1},\underline{\nu}_{2k})}$ are strongly linearly disjoint over $K$. Using this from the inclusion-exclusion principle

    \begin{align*}
        \delta \left( \bigcap_{k=1}^s\overline{\mathcal C_{(\underline{\nu}_{2k-1},\underline{\nu}_{2k})}} \right) &= 1 - \sum_{k=1}^s \delta\left(\mathcal C_{(\underline{\nu}_{2k-1},\underline{\nu}_{2k})}\right) - \cdots (-1)^n \prod_{k=1}^s \delta\left(\mathcal C_{(\underline{\nu}_{2k-1},\underline{\nu}_{2k})}\right) = \prod_{k=1}^s \left(1 - \delta\left(\mathcal C_{(\underline{\nu}_{2k-1},\underline{\nu}_{2k})}\right)\right).
    \end{align*}

    \vspace{2 mm}
    
    Now, if $d_k \coloneqq [L_{(\underline{\nu}_{2k-1},\underline{\nu}_{2k})}:K]$ we have that $\mathcal C_{(\underline{\nu}_{2k-1},\underline{\nu}_{2k})}$ has density at least $d_k^{-|N|}$. Indeed, for each prime in the $|N|$-tuple we are picking from a set of density at least $d_k^{-1}$. From the construction of the fields $L_{(\underline{\nu}_{2k-1},\underline{\nu}_{2k})}$ we can give a uniform upper bound $D$ on $d_k$ in terms of $p$, $\dim \mathfrak g^{\der}_1$, $\dim \mathfrak g^{\der}_2$, and $|N|$, as these quantities give us an upper bound on the images of the cocycles and how many cocycles were used in the definition of $L_{(\underline{\nu}_{2k-1},\underline{\nu}_{2k})}$. Hence 

    $$\delta^+(\mathfrak l) = \delta^+\left(\mathfrak l \setminus \bigcup_{k=1}^s \mathfrak l_k \right) \le \delta \left( \bigcap_{k=1}^s\overline{\mathcal C_{(\underline{\nu}_{2k-1},\underline{\nu}_{2k})}} \right) = \prod_{k=1}^s \left(1 - \delta\left(\mathcal C_{(\underline{\nu}_{2k-1},\underline{\nu}_{2k})}\right)\right) \le (1-D^{-|N|})^s.$$

    \vspace{2 mm}

    However, by our assumption we can take $s$ as large as we want. Therefore, we can make the right-hand side arbitrarily close to $0$, which will contradict the fact that $\mathfrak l$ has positive upper-density. Hence, we can find $(\underline{\nu}_1,\underline{\nu}_2) \in \mathfrak l$ and $(\underline{\nu}_1',\underline{\nu}_2') \in \mathfrak l \cap \mathcal C_{(\underline{\nu}_1,\underline{\nu}_2)}$ consisting of distinct primes and we are done as explained above.
    
\end{proof}

\subsection{Lifting modulo \texorpdfstring{$\varpi^n$}{ϖ\textasciicircum n}}

We will now inductively produce modulo $\varpi^n$ lifts of $\overline{\rho_1}$ and $\overline{\rho_2}$ by generalizing the proof of Proposition \ref{DoublingMethod}. We do that one step at a time passing from a modulo $\varpi^{n-1}$ lifts $\rho_{1,n-1}$ and $\rho_{2,n-1}$ to modulo $\varpi^n$ lifts. We let $L_{i,n-1} \coloneqq K(\rho_{i,n-1}(\mathfrak g_i^{\der}))$ and $L_{n-1} \coloneqq L_{1,n-1}L_{2,n-1}$. Making sure that $L_{1,n-1}$ and $L_{2,n-1}$ are linearly disjoint over $K$ at each step will play an important role in the proof. We will make sure that the chosen lifts satisfy this property. We note that this property is automatically satisfied at the first step, i.e. when passing from $\overline{\rho_i}$ to $\rho_{i,2}$. Indeed, for $n=2$ we have $L_{1,1}=L_{2,1}=K$, as $K$ trivializes the action on $\overline{\rho_i}$, and therefore also the adjoint action on $\overline{\rho_i}(\mathfrak g^{\der})$. To perform the lifting at each step we will use Proposition \ref{cokernelgenerators} with $L=L_{n-1}$ and $c = \max \{D+1,\lceil \frac ne \rceil\}$. However, we first need to verify the linear disjointness in the hypothesis of the proposition. To that end, we have the following lemma, which is inspired by {\cite[Lemma $4.3$]{FKP22}}:

\begin{lem}

    \label{LinearDisjointness}

    Let $\rho_{1,n-1}:\Gamma_{F,T_{n-1}} \to G_1(\mathcal O/\varpi^{n-1})$ and $\rho_{2,n-1}:\Gamma_{F,T_{n-1}} \to G_2(\mathcal O/\varpi^{n-1})$ be lifts of $\overline{\rho_1}$ and $\overline{\rho_2}$, respectively, where $T_{n-1}$ is a finite enlargement of $T$ by trivial primes. Furthermore, assume that the images of $\rho_{1,n-1}$ and $\rho_{2,n-1}$ contain $\widehat{G_1^{\der}}(\mathcal O/\varpi^{n-1})$ and $\widehat{G_2^{\der}}(\mathcal O/\varpi^{n-1})$, respectively, and $L_{1,n-1}$ and $L_{2,n-1}$ are linearly disjoint over $K$. Then $L_{n-1}$ is linearly disjoint over $K$ from the composite of $K_\infty$ with the fixed fields $K_\psi$ for any collection of classes $\psi \in H^1(\Gamma_{F,T_{n-1}},\rho_i(\mathfrak g_i^{\der})^*)$ 
    
\end{lem}

\begin{proof}

    From the linear disjointness of $L_{1,n-1}$ and $L_{2,n-1}$ over $K$ we get

    $$\Gal(L_{n-1}/K) \simeq \Gal(L_{1,n-1}/K) \times \Gal(L_{2,n-1}/K) \simeq \widehat{G_1^{\der}}(\mathcal O/\varpi^{n-1}) \times \widehat{G_2^{\der}}(\mathcal O/\varpi^{n-1}),$$

    \vspace{2 mm}

    \noindent where the second isomorphism follows from the assumption on the images of $\rho_{1,n-1}$ and $\rho_{2,n-2}$. Any abelian quotient of it will be a quotient of $\Gal(L_{n-1}/K)^{\mathrm{ab}}$, which is isomorphic to $\widehat{G_1^{\der}}(\mathcal O/\varpi^2) \times \widehat{G_2^{\der}}(\mathcal O/\varpi^2)$ by Lemma \ref{AbelianizationLemma}, which in turn as an $\mathbb F_p[\Gamma_F]$-module is isomorphic to $\overline{\rho_1}(\mathfrak g_1^{\der}) \oplus \overline{\rho_2}(\mathfrak g_2^{\der})$ under the exponential map.

    Now, we consider the composite of $K_\infty$ and fields $K_\psi$ for any collection of classes $\psi \in H^1(\Gamma_{F,T_{n-1}},\overline{\rho_i}(\mathfrak g_i^{\der})^*)$ for $i=1,2$. Suppose that the intersection of this composite and $L_{n-1}$ is a field $M$, which properly contains $K$. By taking smaller $M$, if necessary, we can assume that $\Gal(M/K)$ is a simple $\mathbb F_p[\Gamma_F]$-module. Then $\Gal(M/K)$ will be a quotient of the Galois group of the composite over $K$, which is isomorphic to a subgroup of $\Gal(K_\infty/K) \times \prod \Gal(K_\psi/K)$. Thus, $\Gal(M/K)$ is a simple $\mathbb F_p[\Gamma_F]$-subquotient of this direct product. By the Jordan-H\"older Theorem all the composite factors are independent of the composition series, so we can compute the simple subquotients of each factor separately. As $\Gal(K_\psi/K)$ is a submodule of $\overline{\rho_i}(\mathfrak g_i^{\der})^*$, every simple subquotient of $\Gal(K_\psi/K)$ will also be a simple subquotient of $\overline{\rho_i}(\mathfrak g_i^{\der})^*$. On the other side, as $K$ contains $\mu_p$ we have that $\Gal(K_\infty/K) \simeq \mathbb Z_p$, where the $\Gamma_F$-action is trivial by \cite[Lemma $2.1$]{FKP22}. Therefore, all of its simple subquotients are of the form $\mathbb Z/p$ with the trivial action.

    By the virtue of being a quotient of $\Gal(L_{n-1}/K)^{\mathrm{ab}}$, $\Gal(M/K)$  is an $\mathbb F_p[\Gamma_F]$-quotient of $\overline{\rho_1}(\mathfrak g_1^{\der}) \oplus \overline{\rho_2}(\mathfrak g_2^{\der})$. As $\Gal(M/K)$ is a simple $\mathbb F_p[\Gamma_F]$-module it has to be a quotient of one of $\overline{\rho_1}(\mathfrak g_1^{\der})$ or $\overline{\rho_2}(\mathfrak g_2^{\der})$. Therefore, $\Gal(M/K)$ can't be isomorphic to a simple $\mathbb F_p[\Gamma_F]$-subquotient of $\overline{\rho_i}(\mathfrak g_i^{\der})^*$, as that will contradict the second bullet point of Assumptions \ref{GeneralAssumptions}. Thus, $\Gal(M/K)$ is isomorphic to $\mathbb Z/p$ with the trivial $\Gamma_F$-action, and one of $\overline{\rho_1}(\mathfrak g_1^{\der})$ or $\overline{\rho_2}(\mathfrak g_2^{\der})$ has a trivial $\mathbb F_p[\Gamma_F]$-quotient. Then, the composition $\overline{\rho_i}(\mathfrak g_i^{\der}) \twoheadrightarrow \mathbb Z/p \hookrightarrow k$ tells us that $\Hom_{\mathbb F_p[\Gamma_F]}(\overline{\rho_i}(\mathfrak g_i^{\der}),k)$ is non-trivial. By definition, this is the $\Gamma_F$-invariant submodule of $\Hom_k(\overline{\rho_i}(\mathfrak g_i^{\der}),k)$. However, by the trace pairing $\Hom_k(\overline{\rho_i}(\mathfrak g_i^{\der}),k) \simeq \overline{\rho_i}(\mathfrak g_i^{\der})$, which yields that $\overline{\rho_i}(\mathfrak g_i^{\der})$ contains the trivial representation, contradicting the third bullet point of Assumptions \ref{GeneralAssumptions}. Hence, we deduce that such $M$ doesn't exist and we obtain the desired linear disjointness over $K$.
    
\end{proof}

\begin{thm}[cf. {\cite[Theorem $4.4$]{FKP22}}]

\label{LiftingModn}

    Let $T \supseteq S$ be the finite set of primes we have constructed in the discussion above. Then, there exists a sequence of finite sets of primes of $F$, $T = T_1\subset T_2 \subset T_3 \subset \dots \subset T_n \subset \cdots$, and for each $n \ge 2$ a lift $\rho_{i,n}:\Gamma_{F,T_n} \to G_i(\mathcal O/\varpi^n)$ of $\overline{\rho_i}$ with multiplier $\mu_i$ for $i = 1,2$, such that $\rho_{i,n} = \rho_{i,n+1} \pmod{\varpi^n}$ for all $n$. These systems of lifts satisfy the following properties:

    \begin{enumerate}
        \item $\rho_{1,n}$ is ramified at $\nu \in T_n \setminus S$ if and only if $\rho_{2,n}$ is ramified at $\nu$. Moreover, for such a prime $\nu$ there is a split maximal torus $T_{i,\nu}$ and a root $\alpha_{i,\nu} \in \Phi(G_i^0,T_{i,\nu})$ such that $\restr{\rho_{i,n}}{\Gamma_{F_\nu}} \in \Lift_{\restr{\overline{\rho_i}}{\Gamma_{F_\nu}}}^{\mu_i,\alpha_{i,\nu}}(\mathcal O/\varpi^n)$. In addition, one of the following two properties holds:
            \begin{enumerate}
                \item For some $s \le eD$ (in particular $s \le n$), $\rho_{i,s}(\tau_\nu)$ is a non-trivial element of $U_{\alpha_{i,\nu}}(\mathcal O/\varpi^s)$ and for all $n' \ge s$, $\restr{\rho_{i,n'}}{\Gamma_{F_\nu}}$ is a $\widehat{G_i}(\mathcal O)$-conjugate to the reduction modulo $\varpi^{n'}$ of a fixed lift $\rho_{i,\nu}:\Gamma_{F_\nu} \to G_i(\mathcal O)$ of $\restr{\rho_{i,s}}{\Gamma_{F_\nu}}$. Moreover, we can choose this lift to correspond to a formally smooth point in the generic fiber of the lifting ring of $\restr{\overline{\rho_i}}{\Gamma_{F_\nu}}$.
                \item For $s = eD$ (in particular $s < n$), $\restr{\rho_{i,s}}{\Gamma_{F_\nu}}$ is trivial modulo the center, while $\alpha_{i,\nu}(\rho_{i,s+1}(\sigma_\nu)) \equiv N(\nu) \not \equiv 1 \pmod {\varpi^{s+1}}$ and $\beta(\rho_{i,s+1}(\sigma_\nu)) \not \equiv 1 \pmod{\varpi^{s+1}}$ for all roots $\beta \in \Phi(G_i^0,T_{i,\nu})$.
            \end{enumerate}

        \item For all $\nu \in S$, $\restr{\rho_{i,n}}{\Gamma_{F_\nu}}$ is strictly equivalent to $\rho_{i,\nu} \pmod{\varpi^n}$.
        \item The fixed fields $L_{1,n} = K(\rho_{1,n}(\mathfrak g_1^{\der}))$ and $L_{2,n} = K(\rho_{2,n}(\mathfrak g_2^{\der}))$ are linearly disjoint over $K$.
        \item $\im(\rho_{i,n})$ contains $\widehat{G_i^{\der}}(\mathcal O/\varpi^n)$ for $i=1,2$.
    \end{enumerate}
    
\end{thm}

\begin{proof}

	We will inductively lift $\overline{\rho_i}$ to $\rho_{i,n}:\Gamma_{F,T_n}\to G_i(\mathcal O/\varpi^n)$ that satisfies the conclusion of the theorem, at each stage enlarging the ramification set $T_n$. Because the data in the inductive step will differ slightly depending on the size of $n$ we will do this lifting in two steps by using two separate inductions. In the first step we will produce a lift $\rho_{i,eD}$ of $\overline{\rho_i}$ at which point we will stop the first induction. In the second induction, starting with the lift $\rho_{eD}$ we will produce the desired lift $\rho_{i,n}$. We remark that primes added during the first induction will satisfy part $(1a)$ of the theorem, while primes added during the second induction will satisfy part $(1b)$ of the theorem.
	
	We first note that $\overline{\rho_i}$ satisfies the statement of the theorem. This follows immediately, since no prime in $T \setminus S$ is ramified in either of the representations, for each $\nu \in S$ we are given a lift $\rho_{i,\nu}$ of $\restr{\overline{\rho_i}}{\Gamma_{F_\nu}}$ by Assumptions \ref{GeneralAssumptions}, and in this case $L_{1,1} = L_{2,1} = K$ and $\widehat{G_i^{\der}}(\mathcal O/\varpi)$ is the trivial group. This verifies the base step of the induction. Hence for $eD \ge n \ge 2$ we can assume that we have already constructed a lift $\rho_{i,n-1}:\Gamma_{F,T_{n-1}} \to G_i(\mathcal O/\varpi^{n-1})$ of $\overline{\rho_i}$ as in the theorem. We remark that for each prime $\nu \in T_{n-1} \setminus S$ at which $\rho_{i,n-1}$ is ramified we are given a torus $T_{i,\nu}$ and a root $\alpha_{i,\nu} \in \Phi(G_i^0,T_{i,\nu})$. We will often change this pair to some $\widehat{G_i}(\mathcal O)$-conjugate of it without changing the notation; see \cite[Remark $5.15$]{FKP21}.
	
	We will first add some primes to $T_{n-1}$ that split in $L_{n-1} = K(\rho_{1,n-1}(\mathfrak g_1^{\der}),\rho_{2,n-1}(\mathfrak g_2^{\der}))$ and introduce unramified multiplier $\mu_i$ lifts $\lambda_{i,\nu}$ in order to make sure properties $(3)$ and $(4)$ are satisfied. We label this enlarged set by $T_{n-1}'$. Let $g_1,\dots, g_s$ be generators of $\ker(G_1^{\der}(\mathcal O/\varpi^n) \to G_1^{\der}(\mathcal O/\varpi^{n-1}))$. Next, we pick primes $\nu_1,\dots,\nu_s$ that split in $L_{n-1}$, which implies that the restrictions of $\rho_{1,n-1}$ and $\rho_{2,n-1}$ to $\Gamma_{F_{\nu_k}}$ are trivial modulo the center. We take $\lambda_{1,\nu_k}:\Gamma_{F_{\nu_k}} \to G_1^{\der}(\mathcal O/\varpi^n)$ to be the unramified representation that sends $\sigma_{\nu_k}$ to $g_k$. On the other side, we let $\lambda_{2,\nu_k}:\Gamma_{F_{\nu_k}} \to G_2^{\der}(\mathcal O/\varpi^n)$ be the trivial representation. Similarly, for generators $h_1,\dots,h_t$ of $\ker(G_2^{\der}(\mathcal O/\varpi^n) \to G_2^{\der}(\mathcal O/\varpi^{n-1}))$ we can find another set of primes $\{w_1,\dots, w_t\}$ that split in $L_{n-1}$, and we then choose $\lambda_{1,w_k}:\Gamma_{F_{w_k}} \to G_1^{\der}(\mathcal O/\varpi^n)$ to be the trivial representation, and $\lambda_{2,w_k}:\Gamma_{F_{w_k}} \to G_2^{\der}(\mathcal O/\varpi^n)$ to be the unramified representation that sends $\sigma_{w_k}$ to $h_k$.
	
	We now note that if we produce any modulo $\varpi^n$ lift $\rho_{i,n}$ of $\rho_{i,n-1}$ so that its restrictions to the local absolute Galois groups at these new primes is equal to the lifts described above then properties $(3)$ and $(4)$ will be satisfied. Indeed, let $g \in \widehat{G_1^{\der}}(\mathcal O/\varpi^n)$. Then, from the inductive hypothesis there exists $\alpha \in \Gamma_F$ such that $\rho_{1,n}(\alpha) \equiv \rho_{1,n-1}(\alpha) \equiv g \pmod{\varpi^{n-1}}$. Thus, $g(\rho_{1,n}(\alpha))^{-1} \in \ker(G_1^{\der}(\mathcal O/\varpi^n) \to G_1^{\der}(\mathcal O/\varpi^{n-1}))$. Hence, there exists $\beta$, a product of $\sigma_{\nu_k}$ such that $g(\rho_{1,n}(\alpha))^{-1} = \rho_{1,n}(\beta)$. Thus, $\rho_{1,n}(\beta\alpha) = g$ and we deduce that $\rho_{1,n}(\Gamma_F)$ contains $\widehat{G_1^{\der}}(\mathcal O/\varpi^n)$. Similarly, we deduce the analogous result for $\rho_{2,n}$.
	
	On the other hand, $\Gal(L_{i,n}/K) \simeq \widehat{G_i^{\der}}(\mathcal O/\varpi^n)$ as this is the maximal possible Galois group and it has to be achieved as property $(4)$ is satisfied. Therefore, to show that $L_{1,n}$ and $L_{2,n}$ are linearly disjoint over $K$ it is enough to show that for any $g \in \widehat{G_1^{\der}}(\mathcal O/\varpi^n)$ we can find $\sigma \in \Gamma_F$ such that $\rho_{1,n}(\sigma) = g$ in $\widehat{G_1^{\der}}(\mathcal O/\varpi^n)$, while $\rho_{2,n}(\sigma)$ is trivial modulo the center; and similarly for any $h \in \widehat{G_2^{\der}}(\mathcal O/\varpi^n)$ we need to find $\sigma \in \Gamma_F$ such that $\rho_{1,n}(\sigma)$ is trivial modulo the center, and $\rho_{2,n}(\sigma) = h$ in $\widehat{G_2^{\der}}(\mathcal O/\varpi^n)$. The existence of such $\sigma$ follows from the choice of the primes $\nu_1,\dots,\nu_s$, and $w_1,\dots,w_t$ above. In particular, in the previous paragraph we can choose $\alpha$ to be a product of $\sigma_{\nu_k}$ for the primes $\nu_k$ that were added during the previous lifting steps, as their images under $\rho_{1,n-1}$ will generate $\widehat{G_1^{\der}}(\mathcal O/\varpi^{n-1})$. Therefore, $\rho_{1,n}(\beta\alpha) = g$ in $\widehat{G^{\der}}(\mathcal O/\varpi^n)$, while $\rho_{2,n}(\beta\alpha)$ is trivial modulo the center, as $\alpha$ and $\beta$ are products of $\sigma_{\nu_k}$ and all the local lifts $\lambda_{2,\sigma_{\nu_k}}$ are trivial modulo the center. We get the analogous result with the roles of $\rho_{1,n}$ and $\rho_{2,n}$, and primes $\nu_k$ and $w_k$ reversed. Before moving on we remark that by \cite[Lemma $6.15$]{FKP21} this needs to be done for only finitely many times. By this for large enough $n$ we have that if the images of the Frobenii $\sigma_{\nu_k}$ under $\rho_{1,n}$ generate $\widehat{G_1^{\der}}(\mathcal O/\varpi^n)$ then the images of the same Frobenii under $\rho_{1,m}$ will generate $\widehat{G_1^{\der}}(\mathcal O/\varpi^m)$ for all $m \ge n$. On the other hand, their images under $\rho_{2,m}$ will still be trivial. The analogous statement holds for $\rho_{2,n}$ and $\widehat{G_2^{\der}}(\mathcal O/\varpi^n)$, as well.
	
	Now for $\nu \in T_{n-1}$ we fix local lifts $\lambda_{i,\nu}:\Gamma_{F_\nu} \to G_i(\mathcal O/\varpi^n)$ of $\restr{\rho_{i,n-1}}{\Gamma_{F_\nu}}$ as follows:
	
	\begin{itemize}
		\item If $\nu \in S$, as $\rho_{i,n-1}$ satisfies the statement of the theorem, by part $(2)$ of the theorem we have that $\restr{\rho_{i,n-1}}{\Gamma_{F_\nu}}$ is a $\widehat{G_i}(\mathcal O)$-conjugate to the given lift $\rho_{i,\nu}$ modulo $\varpi^{n-1}$. We take $\lambda_{i,\nu}$ to be that $\widehat{G_i}(\mathcal O)$-conjugate of $\rho_{i,\nu} \pmod{\varpi^n}$.
		\item If $\nu \in T_{n-1} \setminus S$ and neither $\rho_{1,n-1}$ nor $\rho_{2,n-1}$ is ramified at $\nu$ we take $\lambda_{i,\nu}$ to be any multiplier $\mu_i$ unramified lift of $\restr{\rho_{i,n-1}}{\Gamma_{F_\nu}}$.
		\item If $\nu \in T_{n-1} \setminus S$ and both $\rho_{1,n-1}$ and $\rho_{2,n-1}$ are ramified at $\nu$, then as $n \le eD$ we have that part $(1a)$ of the theorem will be satisfied for this prime $\nu$. Then, by the inductive hypothesis for some $s < n$ we have that $\rho_{i,s}(\tau_\nu)$ is a non-trivial element of $U_{\alpha_{i,\nu}}(\mathcal O/\varpi^s)$ and we are given a $\widehat{G_i}(\mathcal O)$-conjugate to a lift $\rho_{i,\nu}$ of $\restr{\rho_{i,s}}{\Gamma_{F_\nu}}$ which was introduced in the $s$-th stage of the induction. We replace $\rho_{i,\nu}$ by this $\widehat{G_i}(\mathcal O)$-conjugate of it, and we take $\lambda_{i,\nu} = \rho_{i,\nu} \pmod{\varpi^n}$.
	
	\end{itemize}
	
	By Lemma \ref{AnnihilationofSh} we have $\Sh_T^1(\Gamma_{F,T},\overline{\rho_i}(\mathfrak g_i^{\der})^*) = 0$, hence $\Sh_{T_{n-1}'}^1(\Gamma_{F,T_{n-1}'},\overline{\rho_i}(\mathfrak g_i^{\der})^*) = 0$. Therefore, by global duality and the existence of local lifts we can find a lift $\rho_{i,n}':\Gamma_{F,T_{n-1}'} \to G_i(\mathcal O/\varpi^n)$ of $\rho_{i,n-1}$. We then define a class $z_{i,T_{n-1}'} = (z_\nu)_{\nu \in T_{n-1}'} \in \bigoplus_{\nu \in T_{n-1}'} H^1(\Gamma_{F_\nu},\overline{\rho_i}(\mathfrak g_i^{\der}))$ such that for all $\nu \in T_{n-1}'$ we have

    	$$\exp(\varpi^{n-1}z_{i,T_{n-1}'})\restr{\rho_{i,n}'}{\Gamma_{F_\nu}} = \lambda_{i,\nu}.$$

    	\vspace{2 mm}

    	If $z_{i,T_{n-1}'}$ is the restriction of a global cohomology class $h_i \in H^1(\Gamma_{F,T_{n-1}'},\overline{\rho_i}(\mathfrak g_i^{\der}))$ for both $i=1,2$ then we set $\rho_{i,n} \coloneqq \exp(\varpi^{n-1}h_i)\rho_{i,n}'$ and $T_n = T_{n-1}'$. In this case, $\rho_{i,n}$ will be a lift of $\rho_{i,n-1}$ such that $\restr{\rho_{i,n}}{\Gamma_{F_\nu}}$ is a $\ker(G_i^{\der}(\mathcal O/\varpi^n) \to G_i^{\der}(\mathcal O/\varpi^{n-1}))$-conjugate to the fixed local lift $\lambda_{i,\nu}$ for all $\nu \in T_n$. By the choice of the local lifts properties $(1)$ and $(2)$ are satisfied and we are done. 
	
	If not, we will enlarge the set $T_{n-1}'$ using the analogue to Proposition \ref{DoublingMethod}. We would like to achieve this by adding trivial primes $\nu$ such that $\rho_{i,n}(\sigma_\nu)$ would be trivial modulo the center. We briefly explain how this is done. As in the proposition we have a finite indexing set $N$ and for each $m \in N$ we have Chebotarev sets $\mathcal C_{j,m}$ with positive upper-density subsets $\mathfrak l_{j,m}$ with various cohomology classes for each $\nu \in \mathfrak l_{j,m}$. We produce these by invoking Proposition \ref{cokernelgenerators} with $T_{n-1}'$ in place of $T$, $c = \max\{D+1,\lceil \frac ne \rceil\}$, $L_{n-1}=K(\rho_{1,n-1}(\mathfrak g_1^{\der}),\rho_{2,n-2}(\mathfrak g_2^{\der}))$, and with both $\overline{\rho_1}$ and $\overline{\rho_2}$. The notation $\mathfrak l_{j,m}$ means that these are the positive upper-density sets produced by invoking Proposition \ref{cokernelgenerators} with $\overline{\rho_j}$. The hypothesis of the proposition is satisfied since $c \ge D+1$, $L_{n-1}$ is unramified outside of $T_{n-1}'$, and we have the desired linear disjointness over $K$ from Lemma \ref{LinearDisjointness}. We note that for each application of the proposition it produces classes $q_{j,m} \in \ker((\mathbb Z/p^c)^\times \to (\mathbb Z/p^D)^\times)$, a split maximal tori $T_{j,m}^{(1)}$ of $G_1^0$, and $T_{j,m}^{(2)}$ of $G_2^0$, together with roots $\alpha_{j,m}^{(1)} \in \Phi(G_1^0,T_{j,m}^{(1)})$, and $\alpha_{j,m}^{(2)} \in \Phi(G_2^0,T_{j,m}^{(2)})$. After these are produced we can choose the elements $g_{j,m} \in \Gal(L_{n-1}/K)$. We choose them all to be trivial. 
	
	At this point we prescribe the $2|N|$-tuples $(A_{i,m},A_{i,m}')_{m \in N}$ of elements in $\widehat{G^{\der}}(\mathcal O/\varpi^n)$. We choose them all to be trivial. We need to make sure that this choice is allowable. In particular, we need to make sure that $\mathfrak l_{j,m}$ contains positive-density sets of primes $\{\nu_m\}$ and $\{\nu_m'\}$ such that $\rho_{i,n-1}(\sigma_{\nu_m}) = A_{i,m}$ in $G_i^{\der}(\mathcal O/\varpi^{n-1})$ and $\rho_{i,n-1}(\sigma_{\nu_m'}) = A_{i,m}'$ in $G_i^{\der}(\mathcal O/\varpi^{n-1})$ for both $i=1,2$. This is guaranteed for all primes $\nu \in \mathfrak l_{j,m}$ by our choice of $g_{j,m}$ above, which ensures that $\rho_{i,n-1}(\sigma_\nu)$ is trivial modulo the center. 
	
	Now, by the limiting argument of Proposition \ref{DoublingMethod} there is a $2|N|$-tuple of trivial primes $Q = (\nu_m,\nu_m')_{m \in N}$ disjoint from $T_{n-1}'$ and for $i=1,2$ a class $h_i \in H^1(\Gamma_{F,T_{n-1}'\cup Q},\overline{\rho_i}(\mathfrak g_i^{\der}))$ such that $\restr{h_i}{T_{n-1}'} = z_{i,T_{n-1}'}$, and for all $m \in N$ and $i=1,2$ there is a pair $(T_{i,m},\alpha_{i,m})$ of a split maximal torus $T_{i,m}$ of $G^0_i$ and a root $\alpha_{i,m} \in \Phi(G^0_i,T_{i,m})$ such that $\rho_{i,n}(\tau_{\nu_m}) = \rho_{i,n}(\tau_{\nu_m'}) = u_{\alpha_{i,m}}(X_{i,m})$ for some $\alpha_{i,m}$-root vector $X_{i,m}$ that is non-zero modulo $\varpi^n$,  $\rho_{i,n}(\sigma_{\nu_m}) = A_{i,m}$ and $\rho_{i,n}(\sigma_{\nu_m'}) = A_{i,m}'$ modulo the center, for $\rho_{i,n} = \exp(\varpi^{n-1}h_i)\rho_{i,n}':\Gamma_{F,T_{n-1}'\cup Q} \to G_i(\mathcal O/\varpi^n)$.
	
	By construction the lifts $\rho_{i,n}$ would satisfy part $(1a)$ of the theorem for any prime in $T_{n-1}'$. It remains to show that this is the case for primes in $Q$ as well. Each prime $\nu \in Q$ comes with an associated split maximal torus $T_{i,\nu}$ of $G^0_i$ and a root $\alpha_{i,\nu} \in \Phi(G^0_i,T_{i,\nu})$. As $\rho_{i,n}(\sigma_\nu)$ is trivial modulo the center, so is its image under the root $\alpha_{i,\nu}$. On the other side, as $\nu$ is a trivial prime we have that $N(\nu)$ is trivial modulo $p^D=\varpi^{eD}$ and in particular modulo $\varpi^n$. Hence $\restr{\rho_{i,n}}{\Gamma_{F_\nu}} \in \Lift_{\restr{\overline{\rho_i}}{\Gamma_{F_\nu}}}^{\mu_i,\alpha_{i,\nu}}(\mathcal O/\varpi^n)$. Now, by \cite[Lemma $3.7$]{FKP21} we can produce a lift $\rho_{i,\nu} \in \Lift_{\restr{\overline{\rho_i}}{\Gamma_{F_\nu}}}^{\mu_i,\alpha_{i,\nu}}(\mathcal O)$ of $\restr{\rho_{i,n}}{\Gamma_{F_\nu}}$ which moreover defines a formally smooth point in the generic fiber of the trivial inertial type lifting ring for $\restr{\rho_i}{\Gamma_{F_\nu}}$. This is exactly the $p$-adic local lift for which part $(1a)$ of the theorem is satisfies for the prime $\nu$. We are then done by taking $T_n = T_{n-1}' \cup Q$.
	
	That concludes our first induction by which we have produced lifts $\rho_{i,eD}:\Gamma_{F,T_{eD}}\to G_i(\mathcal O/\varpi^{eD})$ as in the statement of the theorem. Before starting with the second induction we will first produce lifts $\rho_{i,eD+1}$ by hand. As in the inductive step above, if needed, we first enlarge the set $T_{eD}$ to a set $T_{eD}'$ by introducing unramified multiplier $\mu_i$ lifts $\lambda_{i,\nu}$ at each new prime. Again, the reason for this is to ensure that the lifts $\rho_{i,eD+1}$ satisfy properties $(3)$ and $(4)$. At primes $\nu \in T_{eD}$ we fix local lifts $\lambda_{i,\nu}$ as before, recalling that primes in $T_{eD}$ at which both $\rho_{i,eD}$ ramify will satisfy part $(1a)$ of the theorem. As before we might need to allow extra ramification at some finite set of primes in order to produce lifts $\rho_{i,eD+1}$. We do this in the same way as above, by using the analogue of Proposition \ref{DoublingMethod}. Again, we have a finite indexing set $N$ and for each $m \in N$ by invoking Proposition \ref{cokernelgenerators} we produce Chebotarev sets $\mathcal C_{j,m}$ with positive upper-density subsets $\mathfrak l_{j,m}$. We still choose the elements $g_{j,m} \in \Gal(L_{eD}/K)$ to be trivial. The only difference in comparison to the argument in the first induction will be how we choose the $2|N|$-tuples $(A_{i,m},A_{i,m}')_{m \in N}$ of elements in $\widehat{G^{\der}}(\mathcal O/\varpi^{eD+1})$. We recall that each prime $\nu \in \mathfrak l_{j,m}$ comes with associated tori $T_{j,m}^{(i)}$ of $G_i^0$, roots $\alpha_{j,m}^{(i)} \in \Phi(G_i^0,T_{j,m}^{(i)})$, and a non-trivial class $q_{j,m} \in \ker((\mathbb Z/p^{D+1})^\times \to (\mathbb Z/p^D)^\times)$. We then choose $A_{i,m}$ and $A_{i,m}'$ to be equal to an element $t_{j,m}^{(i)} \in T_{j,m}^{(i)}$ such that $t_{j,m}^{(i)} \pmod{\varpi^{eD}}$ is trivial modulo the center, $\alpha_{j,m}^{(i)}(t_{j,m}^{(i)}) \equiv q_{j,m} \pmod{\varpi^{eD+1}}$, and $\beta(t_{j,m}^{(i)}) \not \equiv 1 \pmod{\varpi^{eD+1}}$ for all $\beta \in \Phi(G^0_i,T_{j,m}^{(i)})$. Such element exists by Lemma \ref{TorusElementLemma} for $p \gg_{G_i} 0$. We are allowed to choose this values for $A_{i,m}$ and $A_{i,m}'$ since it is trivial in $G^{\der}_i(\mathcal O/\varpi^{eD})$ and the same is true of $\rho_{i,eD}(\sigma_\nu)$ for any $\nu \in \mathfrak l_{j,m}$ by our choice of $g_{j,m} \in \Gal(L_{eD}/K)$. Again, by the limiting argument of Proposition \ref{DoublingMethod} we can produce the desired lifts $\rho_{i,eD+1}$ by allowing extra ramification at a $2|N|$-tuple of trivial primes $Q = (\nu_m,\nu_m')_{m \in N}$. By construction primes $\nu \in Q$ will satisfy the statement of part $(1b)$ of the theorem.  
	
	We now start the second induction which will produce lifts $\rho_{i,n}$ modulo $\varpi^n$ for an arbitrarily large $n$. We start our induction from $n=eD+1$. In the intermediate step above we have constructed lifts $\rho_{i,eD+1}$ which satisfy the statement of the theorem. Hence, the base step of the induction is verified. We now suppose that for some $n \ge eD+2$ we have constructed a lift $\rho_{i,n-1}:\Gamma_{F,T_{n-1}} \to G_i(\mathcal O/\varpi^{n-1})$ of $\overline{\rho_i}$ as in the statement of the theorem, using the outlined procedure. As before, if needed, we first enlarge $T_{n-1}$ to a set $T_{n-1}'$ by introducing unramified multiplier $\mu_i$ lifts $\lambda_{i,\nu}$ at each new prime, guaranteeing that the lifts $\rho_{i,n}$ will satisfy properties $(3)$ and $(4)$ of the theorem. For each $\nu \in T_{n-1}$ we fix local lifts $\lambda_{i,\nu}:\Gamma_{F_\nu}\to G_i(\mathcal O/\varpi^n)$ of $\restr{\rho_{i,n-1}}{\Gamma_{F_\nu}}$ as follows:
	
	\begin{itemize}
		\item If $\nu \in S$, as in the first induction, as $\rho_{i,n-1}$ satisfies the statement of the theorem, by part $(2)$ of the theorem we have that $\restr{\rho_{i,n-1}}{\Gamma_{F_\nu}}$ is a $\widehat{G_i}(\mathcal O)$-conjugate to the given lift $\rho_{i,\nu}$ modulo $\varpi^{n-1}$. We take $\lambda_{i,\nu}$ to be this $\widehat{G_i}(\mathcal O)$-conjugate of $\rho_{i,\nu} \pmod{\varpi^n}$.
		\item If $\nu \in T_{n-1} \setminus S$ and neither $\rho_{1,n-1}$ nor $\rho_{2,n-1}$ is ramified at $\nu$ similarly as in the first induction we take $\lambda_{i,\nu}$ to be any multiplier $\mu_i$ unramified lift of $\restr{\rho_{i,n-1}}{\Gamma_{F_\nu}}$.
		\item Let $\nu \in T_{eD} \setminus S$ and both $\rho_{1,n-1}$ and $\rho_{2,n-1}$ be ramified at $\nu$. These are the primes where we have allowed extra ramification during the first induction above. As remarked before these will satisfy part $(1a)$ of the theorem. Therefore, by the inductive hypothesis for some $s \le eD$ we have that $\rho_{i,s}(\tau_\nu)$ is a non-trivial element of $U_{\alpha_{i,\nu}}(\mathcal O/\varpi^s)$ and we are given a $\widehat{G_i}(\mathcal O)$-conjugate to a lift $\rho_{i,\nu}$ of $\restr{\rho_{i,s}}{\Gamma_{F_\nu}}$ which was introduced in the $s$-th stage of the induction. We replace $\rho_{i,\nu}$ by this $\widehat{G_i}(\mathcal O)$-conjugate of it, and we take $\lambda_{i,\nu} = \rho_{i,\nu} \pmod{\varpi^n}$.
		\item Finally we consider primes $\nu \in T_{n-1} \setminus T_{eD}$ at which both $\rho_{1,n-1}$ and $\rho_{2,n-1}$ ramify. These are the primes that we have allowed extra ramification during the second induction and they satisfy part $(1b)$ of the theorem. By the inductive hypothesis $\restr{\rho_{i,eD}}{\Gamma_{F_\nu}}$ is trivial modulo the center, while $\restr{\rho_{i,n-1}}{\Gamma_{F_\nu}}$ is in a general position as described in the statement of the theorem. We then choose $\lambda_{i,\nu}$ to be any lift of $\restr{\rho_{i,n-1}}{\Gamma_{F_\nu}}$ that will still satisfy these conditions. Such a lift exists by Lemma \ref{GeneralLiftLemma}. We moreover choose $\lambda_{i,\nu}$ to have multiplier $\mu_i$. As $\mu_i$ isn't ramified at $\nu$ to achieve this it is enough to choose $\lambda_{i,\nu}$ such that $\lambda_{i,\nu}(\sigma_\nu) \pmod{G_i^{\der}} = \mu_i(\sigma_\nu)$. This is possible by the final part of Lemma \ref{GeneralLiftLemma}.
	
	\end{itemize}
	
	We might need to allow extra ramification at some finite set of primes to produce lifts $\rho_{i,n}$, which we do by using the analogue of Proposition \ref{DoublingMethod}. As before, we have a finite indexing set $N$ and for each $m \in N$ by invoking Proposition \ref{cokernelgenerators} we produce Chebotarev sets $\mathcal C_{j,m}$ with positive upper-density subsets $\mathfrak l_{j,m}$. Each prime $\nu \in \mathfrak l_{j,m}$ comes with associated tori $T_{j,m}^{(i)}$ of $G_i^0$, roots $\alpha_{j,m}^{(i)} \in \Phi(G_i^0,T_{j,m}^{(i)})$, and a class $q_{j,m} \in \ker((\mathbb Z/p^{c})^\times \to (\mathbb Z/p^D)^\times)$, which is non-trivial modulo $p^{D+1}$. We then choose $g_{j,m} \in \Gal(L_{n-1}/K)$ to be the element corresponding to $(t_{j,m}^{(1)},t_{j,m}^{(2)})$ via the isomorphism
	
	$$\Gal(L_{n-1}/K) \simeq \Gal(L_{1,n-1}/K) \times \Gal(L_{2,n-1}/K) \simeq \widehat{G_1^{\der}}(\mathcal O/\varpi^{n-1}) \times \widehat{G_2^{\der}}(\mathcal O/\varpi^{n-1}),$$

        \vspace{2 mm}

        \noindent where for $i=1,2$, $t_{j,m}^{(i)} \in T_{j,m}^{(i)}(\mathcal O/\varpi^{n-1})$ is trivial modulo $\varpi^{eD}$ and satisfies $\alpha_{j,m}^{(i)}(t_{j,m}^{(i)}) \equiv q_{j,m} \not \equiv 1\pmod{\varpi^{n-1}}$ and $\beta(t_{j,m}^{(i)}) \not \equiv 1 \pmod{\varpi^{eD+1}}$ for all $\beta \in \Phi(G^0_i,T_{j,m}^{(i)})$. As $q_{j,m}$ is trivial modulo $\varpi^{eD}$, but not modulo $\varpi^{eD+1}$ such elements exist by Lemma \ref{TorusElementLemma} for $p \gg_{G_i}0$. We remark that this means that $\rho_{i,n-1}(\sigma_\nu) = t_{j,m}^{(i)}$ modulo the center for all $\nu \in \mathfrak l_{j,m}$. We then choose the $2|N|$-tuples $(A_{i,m},A_{i,m}')_{m \in N}$ of elements in $\widehat{G^{\der}}(\mathcal O/\varpi^n)$ by asking for $A_{i,m}$ and $A_{i,m}'$ to be lifts of $t_{j,m}^{(i)}$ that will satisfy the same properties modulo $\varpi^n$. Such lifts exist by arguing as in the inductive step of Lemma \ref{TorusElementLemma}. It's clear that we are allowed to choose these values for $A_{i,m}$ and $A_{i,m}'$. Finally, by the limiting argument of Proposition \ref{DoublingMethod} we can produce the desired lifts $\rho_{i,n}$ by allowing extra ramification at a $2|N|$-tuple of trivial primes $Q = (\nu_m,\nu_m')_{m \in N}$. By construction primes $\nu \in Q$ will satisfy the statement of part $(1b)$ of the theorem.  

\end{proof}

\section{Relative deformation theory}

\subsection{Relative Selmer groups}

Let $n$ be a large enough integer, in particular we want $n \ge eD$ and $n$ to be divisible by the ramification index $e$. By Theorem \ref{LiftingModn} we produce lifts $\rho_{1,n}:\Gamma_{F,S'} \to G_1(\mathcal O/\varpi^n)$ and $\rho_{2,n}:\Gamma_{F,S'} \to G_2(\mathcal O/\varpi^n)$ of $\overline{\rho_1}$ and $\overline{\rho_2}$, respectively, writing $S'$ in place of $T_n$.

\begin{lem}
\label{rhoNProperties}

    \text{ }

    \begin{enumerate}
        \item $\rho_{i,r}(\mathfrak g^{\der}_i)$ and $\rho_{j,s}(\mathfrak g_j^{\der})^*$ have no common $\mathbb F_p[\Gamma_F]$-subquotient for $i,j = 1,2$ and $r,s \le n$.
        \item $\rho_{i,r}(\mathfrak g_i^{\der})$ and $\rho_{i,r}(\mathfrak g_i^{\der})^*$ do not contain the trivial representation for any $i=1,2$ and $r \le n$.
    \end{enumerate}
    
\end{lem}
\begin{proof}

    \text{ }

    \begin{enumerate}
        \item We will show that the simple $\mathbb F_p[\Gamma_F]$-subquotients of       $\rho_{i,r}(\mathfrak g_i^{\der})$ are the simple $\mathbb F_p[\Gamma_F]$-subquotients of $\overline{\rho_i}(\mathfrak g_i^{\der})$, possibly with different multiplicities. This is trivially true for $r=1$. Now, suppose that $r \ge 2$ and consider the exact sequence of $\mathbb F_p[\Gamma_F]$-modules

        $$0 \longrightarrow \rho_{i,r-1}(\mathfrak g_i^{\der}) \longrightarrow \rho_{i,r}(\mathfrak g_i^{\der}) \longrightarrow \overline{\rho_i}(\mathfrak g_i^{\der}) \longrightarrow 0.$$

        \vspace{2 mm}

        From the Jordan-H\"older theorem the simple $\mathbb F_p[\Gamma_F]$-subquotients of $\rho_{i,r}(\mathfrak g_i^{\der})$ are exactly the simple $\mathbb F_p[\Gamma_F]$-subquotients of $\rho_{i,r-1}(\mathfrak g_i^{\der})$ and $\overline{\rho_i}(\mathfrak g_i^{\der})$. Inductively, these are the simple $\mathbb F_p[\Gamma_F]$-subquotients of $\overline{\rho_i}(\mathfrak g_i^{\der})$. By a similar argument, the simple $\mathbb F_p[\Gamma_F]$-subquotients of $\rho_{j,s}(\mathfrak g_j^{\der})^*$ are also the simple $\mathbb F_p[\Gamma_F]$-subquotients of $\overline{\rho_j}(\mathfrak g_j^{\der})^*$. Therefore, any common $\mathbb F_p[\Gamma_F]$-subquotient of $\rho_{i,r}(\mathfrak g_i^{\der})$ and $\rho_{j,s}(\mathfrak g_j^{\der})^*$ would yield a common $\mathbb F_p[\Gamma_F]$-subquotient of $\overline{\rho_i}(\mathfrak g_i^{\der})$ and $\overline{\rho_j}(\mathfrak g_j^{\der})^*$, which contradicts the second bullet point of Assumptions \ref{GeneralAssumptions}.

        \item We will prove this statement by induction on $r$. The base case $r=1$ is true by the third bullet point of Assumptions \ref{GeneralAssumptions}. Now, let $r \ge 2$ and suppose that the statement of the lemma is true for $r-1$. The exact sequence in the previous part induces an exact sequence on $\Gamma_F$-invariants

        $$0 \longrightarrow H^0(\Gamma_F,\rho_{i,r-1}(\mathfrak g_i^{\der})) \longrightarrow H^0(\Gamma_F,\rho_{i,r}(\mathfrak g_i^{\der})) \longrightarrow H^0(\Gamma_F,\overline{\rho_i}(\mathfrak g_i^{\der})).$$

        \vspace{2 mm}

        By the inductive hypothesis and the third bullet point of Assumptions \ref{GeneralAssumptions} both end terms are trivial, hence $H^0(\Gamma_F,\rho_{i,r}(\mathfrak g_i^{\der})) = 0$. Thus, $\rho_{i,r}(\mathfrak g_i^{\der})$ doesn't contain the trivial representation. We get the analogous result for $\rho_{i,r}(\mathfrak g_i^{\der})^*$ in a similar manner.

    \end{enumerate}
    
\end{proof}

For each prime $\nu$ in the finite set of primes $S'$ and $1 \le r \le n$ we will define subgroups $Z_{i,r,\nu} \subseteq Z^1(\Gamma_{F_\nu},\rho_{i,r}(\mathfrak g_i^{\der}))$ such that

\begin{itemize}
    \item Each $Z_{i,r,\nu}$ contains the coboundaries $B^1(\Gamma_{F_\nu},\rho_{i,r}(\mathfrak g_i^{\der}))$.
    \item For $a + b \le n$, the inclusion and the reduction maps induce short exact sequences

    \begin{equation}
    \label{stares}
    0 \longrightarrow Z_{i,a,\nu} \longrightarrow Z_{i,a+b,\nu} \longrightarrow Z_{i,b,\nu} \longrightarrow 0. \tag{$\ast$}
    \end{equation}

    \vspace{2 mm}
    
\end{itemize}

As primes in $S'$ were added during different steps of the lifting method and they satisfy different properties we will have to define these cocycles subgroups differently. We do that as follows:

\begin{itemize}
    \item For each prime $\nu$ in the initial set $S$ by the fourth bullet point of Assumptions \ref{GeneralAssumptions} we are given a $p$-adic lift $\rho_{i,\nu}:\Gamma_{F_\nu} \to G_i(\mathcal O)$ with multiplier $\mu_i$ of $\restr{\overline{\rho_i}}{\Gamma_{F_\nu}}$. We then fix an irreducible component $\overline{R_{i,\nu}}[1/\varpi]$ of the generic fiber of the lifting ring $R^{\sqr,\mu_i}_{\restr{\overline{\rho_i}}{\Gamma_{F_\nu}}}[1/\varpi]$ (for $\nu \nmid p$) or of the inertial type $\tau_i$ and $p$-adic Hodge type $\textup{\textbf{v}}_i$ lifting ring $R^{\sqr,\mu_i,\tau_i,\textup{\textbf{v}}_i}_{\restr{\overline{\rho_i}}{\Gamma_{F_\nu}}}[1/\varpi]$ (for $\nu \mid p$; see \cite[Proposition $3.0.12$]{Bal12} for the construction of this ring), which contains the specified lift $\rho_{i,\nu}$. The lift $\rho_{i,\nu}$ might not necessarily correspond to a formally smooth point, however by \cite[Lemma $4.9$]{FKP21} we can find a formally smooth point which is $\varpi$-adically arbitrarily close to it. More precisely, for any integer $s$ we can find a lift $\rho_{i,\nu}':\Gamma_{F_\nu} \to G_i(\mathcal O')$ for some finite extension $\mathcal O'$ of $\mathcal O$, which defines a formally smooth point of $\overline{R_{i,\nu}}[1/\varpi]$ and $\rho_{i,\nu}' \equiv \rho_{i,\nu} \pmod{\varpi^{s}}$. We choose such an approximation and we replace $\rho_{i,\nu}$ with $\rho_{i,\nu}'$ and $\mathcal O$ with $\mathcal O'$, while retaining the original notation. This might increase the ramification index $e$, and subsequently our lower bound on $n$. However, we assume that we have done this replacement before starting with the lifting method. We note that the ring $\mathcal O$ will not be enlarged any further during the lifting method. Finally, we produce the desired cocycles subgroups using \cite[Proposition $4.7$]{FKP21}.

    \item Let $\nu$ be a prime in $S' \setminus S$ for which part $(1a)$ of Theorem \ref{LiftingModn} is satisfied. We have a $p$-adic lift $\rho_{i,\nu}:\Gamma_{F_\nu} \to G_i(\mathcal O)$ of $\restr{\overline{\rho_i}}{\Gamma_{F_\nu}}$ which corresponds to a formally smooth point in the generic fiber of the lifting ring of $\restr{\overline{\rho_i}}{\Gamma_{F_\nu}}$. Then, again we use \cite[Proposition $4.7$]{FKP21} to produce the subgroups $Z_{i,r,\nu}$. 

    \item Next, let $\nu$ be a prime in $S' \setminus S$ for which part $(1b)$ of Theorem \ref{LiftingModn} is satisfied. We then have an explicit description of the local behavior of $\rho_{i,n}$ at $\nu$. We define the subgroups $Z_{i,r,\nu}$ using \cite[Lemma $3.5$]{FKP21}.

    \item Finally, at the primes $\nu$ in $S' \setminus S$ at which neither $\rho_{1,n}$ nor $\rho_{2,n}$ is ramified we let $Z_{i,r,\nu}$ be the subspace of unramified cocycles.

\end{itemize}

\vspace{2 mm}

We let $L_{i,r,\nu} \subseteq H^1(\Gamma_{F_\nu},\rho_{i,r}(\mathfrak g_i^{\der}))$ be the image of $Z_{i,r,\nu}$, and we let $L_{i,r,\nu}^\perp \subseteq H^1(\Gamma_{F_\nu},\rho_{i,r}(\mathfrak g_i^{\der})^*)$ be the annihilator of $L_{i,r,\nu}$ under the local duality pairing. We label the collection $\{L_{i,r,\nu}\}_{\nu \in S'}$ by $\mathcal L_{i,r,S'}$ and similarly the collection $\{L_{i,r,\nu}^\perp\}_{\nu \in S'}$ by $\mathcal L_{i,r,S'}^\perp$.

\begin{defn}
\label{rSelmerDef}

The \textit{$r$-th Selmer group} $H_{\mathcal L_{i,r,S'}}^1(\Gamma_{F,S'},\rho_{i,r}(\mathfrak g_i^{\der}))$ is defined to be

$$\ker \left( H^1(\Gamma_{F,S'},\rho_{i,r}(\mathfrak g_i^{\der})) \longrightarrow \bigoplus_{\nu \in S'} \frac{H^1(\Gamma_{F_\nu},\rho_{i,r}(\mathfrak g_i^{\der}))}{L_{i,r,\nu}}\right).$$

\vspace{2 mm}

We analogously define the \textit{$r$-th dual Selmer group} $H^1_{\mathcal L_{i,r,S'}^\perp} H^1(\Gamma_{F,S'},\rho_{i,r}(\mathfrak g_i^{\der})^*)$.

\end{defn}

\begin{lem}
\label{BalancedSelmer}

The local conditions $\mathcal L_{i,r,S'}$ for $1 \le r \le n$ are \textit{balanced}, i.e.

$$\left|H^1_{\mathcal L_{i,r,S'}}(\Gamma_{F,S'},\rho_{i,r}(\mathfrak g_i^{\der}))\right| = \left|H^1_{\mathcal L_{i,r,S'}^\perp}(\Gamma_{F,S'},\rho_{i,r}(\mathfrak g_i^{\der})^*)\right|.$$
    
\end{lem}

\begin{proof}

    Applying the Greenberg-Wiles formula with the local conditions $\mathcal L_{i,r,S'}$ yields

    $$\frac{|H^1_{\mathcal L_{i,r,S'}}(\Gamma_{F,S'},\rho_{i,r}(\mathfrak g_i^{\der}))|}{|H^1_{\mathcal L_{i,r,S'}^\perp}(\Gamma_{F,S'},\rho_{i,r}(\mathfrak g_i^{\der})^*)|} = \frac{|H^0(\Gamma_F,\rho_{i,r}(\mathfrak g_i^{\der}))|}{|H^0(\Gamma_F,\rho_{i,r}(\mathfrak g_i^{\der})^*)|} \prod_{\nu \in S'} \frac{|L_{i,r,\nu}|}{|H^0(\Gamma_{F_\nu},\rho_{i,r}(\mathfrak g_i^{\der}))|} \prod_{\nu \mid \infty} \frac 1{|H^0(\Gamma_{F_\nu},\rho_{i,r}(\mathfrak g_i^{\der}))|}.$$

    \vspace{2 mm}

    By part $(2)$ of Lemma \ref{rhoNProperties} both the numerator and the denominator of the first fraction on the right-hand side are $1$ and so that factor is equal to $1$. By the dimension computations in \cite[Lemma $3.5$]{FKP21} and \cite[Proposition $4.7$]{FKP21} all the factors that correspond to primes in $S'$ that do not divide $p$ will be $1$, as well. On the other hand, at primes $\nu \mid p$ as we used Hodge-Tate regular $p$-adic lifts to define the local conditions $L_{i,r,\nu}$ the parabolic associated to $\textup{\textbf{v}}_i$ will be Borel. Therefore, by \cite[Proposition $4.7$]{FKP21} we have that $|L_{i,r,\nu}| = |H^0(\Gamma_{F_\nu},\rho_{i,r}(\mathfrak g_i^{\der}))| \cdot |\mathcal O/\varpi^r|^{[F_\nu,\mathbb Q_p] \cdot \dim_k(\mathfrak n_i)}$, where $\mathfrak n_i$ is the Lie algebra of the unipotent radical of the Borel subgroup of $G_i$ associated to $\textup{\textbf{v}}_i$. Finally, at the infinite places, the assumption that $\overline{\rho_i}$ is odd tells us that $|H^0(\Gamma_{F_\nu},\rho_{i,r}(\mathfrak g_i^{\der}))| = |\mathcal O/\varpi^r|^{\dim_k(\mathfrak n_i)}$. This reduces the product on the right-hand side to

    $$\prod_{\nu \mid p} |\mathcal O/\varpi^r|^{[F_\nu,\mathbb Q_p]\cdot \dim_k(\mathfrak n_i)} \cdot \prod_{\nu \mid \infty} |\mathcal O/\varpi^r|^{-\dim_k(\mathfrak n_i)} = |\mathcal O/\varpi^r|^{(-[F:\mathbb Q] + \sum_{\nu \mid p} [F_\nu,\mathbb Q_p]) \cdot \dim_k(\mathfrak n_i)} = 1,$$

    \vspace{2 mm}

    \noindent where we used that $F$ is a totally real field and so the number of non-Archimedean places is $[F:\mathbb Q]$. This gives us that $|H^1_{\mathcal L_{i,r,S'}}(\Gamma_{F,S'},\rho_{i,r}(\mathfrak g_i^{\der}))| = |H^1_{\mathcal L_{i,r,S'}^\perp}(\Gamma_{F,S'},\rho_{i,r}(\mathfrak g_i^{\der})^*)|$, i.e. the local conditions $\mathcal L_{i,r,S'}$ are balanced.
    
\end{proof}

\begin{lem}[{\cite[Lemma $6.1$]{FKP21}}]
\label{SelmerExactSeq}

For any integers $a,b \ge 1$ and $a+b \le n$ there are exact sequences

$$H_{\mathcal L_{i,a,S'}}^1(\Gamma_{F,S'},\rho_{i,a}(\mathfrak g_i^{\der})) \longrightarrow H_{\mathcal L_{i,a+b,S'}}^1(\Gamma_{F,S'},\rho_{i,a+b}(\mathfrak g_i^{\der})) \longrightarrow H_{\mathcal L_{i,b,S'}}^1(\Gamma_{F,S'},\rho_{i,b}(\mathfrak g_i^{\der}))$$

\vspace{2 mm}

\noindent and

$$H_{\mathcal L_{i,a,S'}^\perp}^1(\Gamma_{F,S'},\rho_{i,a}(\mathfrak g_i^{\der})^*) \longrightarrow H_{\mathcal L_{i,a+b,S'}^\perp}^1(\Gamma_{F,S'},\rho_{i,a+b}(\mathfrak g_i^{\der})^*) \longrightarrow H_{\mathcal L_{i,b,S'}^\perp}^1(\Gamma_{F,S'},\rho_{i,b}(\mathfrak g_i^{\der})^*).$$

\end{lem}
\begin{proof}

    We start with the exact sequence

    $$0 \longrightarrow \rho_{i,a}(\mathfrak g_i^{\der}) \longrightarrow \rho_{i,a+b}(\mathfrak g_i^{\der}) \longrightarrow \rho_{i,b}(\mathfrak g_i^{\der}) \longrightarrow 0.$$

    \vspace{2 mm}

    It induces a long exact sequence on cohomology. In particular, we get an exact sequence

    $$H^1(\Gamma_{F,S'},\rho_{i,a}(\mathfrak g_i^{\der})) \longrightarrow H^1(\Gamma_{F,S'},\rho_{i,a+b}(\mathfrak g_i^{\der})) \longrightarrow H^1(\Gamma_{F,S'},\rho_{i,b}(\mathfrak g_i^{\der})).$$

    \vspace{2 mm}

    For any cohomology class in $H_{\mathcal L_{i,a,S'}}^1(\Gamma_{F,S'},\rho_{i,a}(\mathfrak g_i^{\der}))$ and any $\nu \in S'$ its restriction to $\Gamma_{F_\nu}$ will have a cocycle representative in $Z_{i,a,\nu}$. From the exact sequence (\ref{stares}) in the definition of $Z_{i,r,\nu}$ we have the restriction to $\Gamma_{F_\nu}$ of its image in $H^1(\Gamma_{F,S'},\rho_{i,a+b}(\mathfrak g_i^{\der}))$ will have a representative in $Z_{i,a+b,\nu}$. In particular, the image of the cohomology class will be an element of $H_{\mathcal L_{i,a+b,S'}}^1(\Gamma_{F,S'},\rho_{i,a+b}(\mathfrak g_i^{\der}))$. A similar reasoning tells us that the reduction modulo $\varpi^b$ sends $H_{\mathcal L_{i,a+b,S'}}^1(\Gamma_{F,S'},\rho_{i,a+b}(\mathfrak g_i^{\der}))$ to $H^1_{\mathcal L_{i,b,S'}}(\Gamma_{F,S'},\rho_{i,b}(\mathfrak g_i^{\der}))$, so the composition in the statement of the lemma is well-defined. 

    From the exact sequence on cohomology we get that elements in $H_{\mathcal L_{i,a,S'}}^1(\Gamma_{F,S'},\rho_{i,a}(\mathfrak g_i^{\der}))$ are sent to $0$ in $H_{\mathcal L_{i,b,S'}}^1(\Gamma_{F,S},\rho_{i,b}(\mathfrak g_i^{\der}))$. Conversely, let $\phi \in Z^1(\Gamma_{F,S'},\rho_{i,a+b}(\mathfrak g_i^{\der}))$ represent a class in $H^1_{\mathcal L_{i,a+b,S'}}(\Gamma_{F,S'},\rho_{i,a+b}(\mathfrak g_i^{\der}))$ that is sent to $0$ in $H_{\mathcal L_{i,b,S'}}^1(\Gamma_{F,S'},\rho_{i,b}(\mathfrak g_i^{\der}))$. Therefore, the image of $\phi$ in $Z^1(\Gamma_{F,S'},\rho_{i,b}(\mathfrak g_i^{\der}))$ is a coboundary. All such coboundaries are of the form $\sigma \to \sigma X - X$ for some $X \in \rho_{i,b}(\mathfrak g_i^{\der})$. So, from the surjectivity we can find $X' \in \rho_{i,a+b}(\mathfrak g_i^{\der})$ such that $X' \pmod{\varpi^b} = X$. Thus, $\phi' \coloneqq \phi - (\sigma X' - X')$ is sent to $0$ in $Z^1(\Gamma_{F,S'},\rho_{i,b}(\mathfrak g^{\der}))$. By the choice of $\phi$ we have that $\restr{\phi}{\Gamma_{F_\nu}} \in Z_{i,a+b,\nu}$ for all $\nu \in S'$. On the other side, by our initial assumption $Z_{i,a+b,\nu}$ contains all the coboundaries, so $\restr{\phi'}{\Gamma_{F_\nu}} \in Z_{i,a+b,\nu}$, too. Now, we consider the commutative diagram of exact sequence

    $$\begin{tikzcd}
0 \arrow[rr] &  & {Z^1(\Gamma_{F,S'},\rho_{i,a}(\mathfrak g_i^{\der}))} \arrow[rr] \arrow[dd]          &  & {Z^1(\Gamma_{F,S'},\rho_{i,a+b}(\mathfrak g_i^{\der}))} \arrow[rr] \arrow[dd]          &  & {Z^1(\Gamma_{F,S'},\rho_{i,b}(\mathfrak g_i^{\der}))} \arrow[dd]          \\
             &  &                                                                               &  &                                                                                   &  &                                                                    \\
0 \arrow[rr] &  & {\bigoplus_{v \in S'} Z^1(\Gamma_{F_v},\rho_{i,a}(\mathfrak g_i^{\der}))} \arrow[rr] &  & {\bigoplus_{v \in S'} Z^1(\Gamma_{F_v},\rho_{i,a+b}(\mathfrak g_i^{\der}))} \arrow[rr] &  & {\bigoplus_{v \in S'} Z^1(\Gamma_{F_v},\rho_{i,b}(\mathfrak g_i^{\der}))}
\end{tikzcd}$$

\vspace{2 mm}

    By the exactness of the top row as $\phi'$ is sent to $0$ in $Z^1(\Gamma_{F,S'},\rho_{i,b}(\mathfrak g_i^{\der}))$ we can find $\phi_a \in Z^1(\Gamma_{F,S'},\rho_{i,a}(\mathfrak g_i^{\der}))$ that is sent to $\phi'$. On the other hand, using the right commutative square we get that $\restr{\phi'}{\Gamma_{F_\nu}}$ is sent to $0$ in $Z^1(\Gamma_{F_\nu},\rho_{i,b}(\mathfrak g_i^{\der}))$ for all $\nu \in S'$. Thus, by the exactness of the bottom row for each $\nu \in S'$ we can find $\phi_{a,\nu} \in Z^1(\Gamma_{F_\nu},\rho_{i,a}(\mathfrak g_i^{\der}))$, each mapping to $\restr{\phi'}{\Gamma_{F_\nu}}$. As $\restr{\phi'}{\Gamma_{F_\nu}} \in Z_{i,a+b,\nu}$ using the exact sequence (\ref{stares}) we get that $\phi_{a,\nu} \in Z_{i,a,\nu}$. Finally, using the left commutative square we get that $\restr{\phi_a}{\Gamma_{F_\nu}}$ is also sent to $\restr{\phi'}{\Gamma_{F_\nu}}$ for all $\nu \in S'$. So, by the injectivity of the map we have $\restr{\phi_a}{\Gamma_{F_\nu}} = \phi_{a,\nu}$ for all $\nu \in S'$. Thus, $\phi_a$ represents a cohomology class in $H_{\mathcal L_{i,a,S'}}^1(\Gamma_{F,S'},\rho_{i,a}(\mathfrak g_i^{\der}))$, which gives us the exactness of the sequence. 

    \vspace{2 mm}

    The exact sequence for the dual Selmer groups follows similarly, by replacing $Z_{i,r,\nu}$ with the preimages $Z_{i,r,\nu}^\perp$ of $L_{i,\nu,r}^\perp$ in $Z^1(\Gamma_{F_\nu},\rho_{i,r}(\mathfrak g_i^{\der})^*)$. All we need to prove is the existence of an exact sequence

    $$0 \longrightarrow Z_{i,b,\nu}^\perp \longrightarrow Z_{i,a+b,\nu}^\perp \longrightarrow Z_{i,a,\nu}^\perp.$$

    \vspace{2 mm}

    Let $f^*:Z^1(\Gamma_{F_\nu},\rho_{i,b}(\mathfrak g_i^{\der})^*) \hookrightarrow Z^1(\Gamma_{F_\nu},\rho_{i,a+b}(\mathfrak g_i^{\der})^*)$ be the map induced by reduction modulo $\varpi^b$. We will use the notation $[\,\cdot\,]$ for the cohomology class generated by a cocycle in a suitable cohomology group. Let $\psi \in Z_{i,b,\nu}^\perp$. For any $\phi \in Z_{i,a+b,\nu}$ by the functoriality of the Tate pairing $[f^*(\psi)] \cup [\phi] = [\psi] \cup [\phi \pmod{\varpi^b}] = 0$, since $[\phi \pmod{\varpi^b}] \in L_{i,b,\nu}$ and $[\psi]$ lies in its annihilator. Therefore, $[f^*(\psi)] \in L_{i,a+b,\nu}^\perp$ and $f^*(\psi) \in Z_{i,a+b,\nu}^\perp$. This tells us that the image of $Z_{i,b,\nu}^\perp$ lands into $Z_{i,a+b,\nu}^\perp$, giving us the exactness at the first term. 
    
    On the other hand, let $g^*:Z^1(\Gamma_{F_\nu},\rho_{i,a+b}(\mathfrak g^{\der})^*) \to Z^1(\Gamma_{F_\nu},\rho_{i,a}(\mathfrak g^{\der})^*)$ be the map induced by multipilication by $\varpi^b$. Let $\psi \in Z_{i,a+b,\nu}^\perp$. For any $\phi \in Z_{i,a,\nu}$ by the functoriality of the Tate pairing we have $[g^*(\psi)] \cup [\phi] = [\psi] \cup [\varpi^b\phi] = 0$, since $[\varpi^b\phi] \in L_{i,a+b,\nu}$ and $[\psi] \in L_{i,a+b,\nu}^\perp$. Hence, $[g^*(\psi)] \in L_{i,a,\nu}^\perp$ and $g^*(\psi) \in Z_{i,a,\nu}^\perp$. Therefore, $g^*$ sends $Z_{i,a+b,\nu}^\perp$ to $Z_{i,a,\nu}^\perp$, so the last map is well-defined.
    
    It remains to prove the exactness at the second term, i.e. $\im f^* = \ker g^*$. Directly from the construction we have $g^* \circ f^* = 0$, giving us one inclusion. Conversely let $\psi \in Z_{i,a+b,\nu}^\perp$ such that $g^*(\psi) = 0$. By the exactness on the level of cocycles this means that $\psi = f^*(\psi')$ for some $\psi' \in Z^1(\Gamma_{F_\nu},\rho_{i,b}(\mathfrak g^{\der}_i)^*)$. We claim that $\psi'$ is actually an element of $Z_{i,b,\nu}^\perp$. Let $\phi'$ be any element of $Z_{i,b,\nu}$. From the exact sequence (\ref{stares}) it is the image of some $\phi \in Z_{i,a+b,\nu}$ under the reduction modulo $\varpi^b$ map. Then

    $$[\psi'] \cup [\phi'] = [\psi'] \cup [\phi \pmod{\varpi^b}] = [f^*(\psi')] \cup [\phi] = [\psi] \cup [\phi] = 0,$$

    \vspace{2 mm}

    \noindent as $[\phi] \in L_{i,a+b,\nu}$ and $[\psi] \in L_{i,a+b,\nu}^\perp$. Thus, $[\psi'] \in L_{i,b,\nu}^\perp$ and so $\psi' \in Z_{i,b,\nu}^\perp$, giving us the exactness at the second term.
    
\end{proof}

\begin{defn}
\label{ModsRelSelmer}

    For $1 \le s \le r \le n$ we define the \textit{$(r,s)$-relative Selmer group} as

    $$\overline{H^1_{(r,s)}(\Gamma_{F,S'},\rho_{i,r}(\mathfrak g_i^{\der}))} \coloneqq \im \left(H_{\mathcal L_{i,r,S'}}^1(\Gamma_{F,S'},\rho_{i,r}(\mathfrak g_i^{\der})) \to H^1_{\mathcal L_{i,s,S'}}(\Gamma_{F,S'},\rho_{i,s}(\mathfrak g_i^{\der}))\right),$$

    \vspace{2 mm}

    \noindent where the map is induced by the reduction modulo $\varpi^s$ map. We analogously define the \textit{$(r,s)$-relative dual Selmer group} $H^1_{(r,s)}(\Gamma_{F,S'},\rho_{i,r}(\mathfrak g_i^{\der})^*)$.
    
\end{defn}

\begin{rem}
\label{Mod1RelSelmer}

    When $s = 1$ in the definition of the $(r,s)$-relative Selmer group we use the alternative notation $\overline{H^1_{\mathcal L_{i,r,S'}}(\Gamma_{F,S'},\rho_{i,r}(\mathfrak g_i^{\der}))}$, and we call it the $r$-th relative Selmer group. Given an element $\phi$ in the $r$-th Selmer group we write $\overline{\phi}$ for its image in the corresponding $r$-th relative Selmer group. We use the same convention for the $r$-th relative dual Selmer group. 
    
\end{rem}

\begin{lem}[cf. {\cite[Lemma $6.3$]{FKP21}}]
\label{BalancedRelativeSelmer}

The $r$-th relative Selmer and dual Selmer groups are also balanced, i.e.

$$\dim \left(\overline{H_{\mathcal L_{i,r,S'}}^1(\Gamma_{F,S'},\rho_{i,r}(\mathfrak g_i^{\der}))}\right) = \dim \left(\overline{H_{\mathcal L_{i,r,S'}^\perp}^1(\Gamma_{F,S'},\rho_{i,r}(\mathfrak g_i^{\der})^*)}\right).$$
    
\end{lem}

\begin{proof}

The claim is true for $r=1$, since the relative Selmer groups are just the modulo $\varpi$ Selmer and dual Selmer group and they have the same dimension by Lemma \ref{BalancedSelmer}. Hence we can assume that $r \ge 2$. Using the exact sequence in Lemma \ref{SelmerExactSeq} with $a=r-1$ and $b=1$ we get

\begin{align*}
    \left|\,\overline{H_{\mathcal L_{i,r,S'}}^1(\Gamma_{F,S'},\rho_{i,r}(\mathfrak g_i^{\der}))}\,\right| &= \left|\,\im (H_{\mathcal L_{i,r,S'}}^1(\Gamma_{F,S'},\rho_{i,r}(\mathfrak g_i^{\der})) \to H_{\mathcal L_{i,1,S'}}^1(\Gamma_{F,S'},\overline{\rho_i}(\mathfrak g_i^{\der})))\,\right| \\
    & = \frac{\left|\, H_{\mathcal L_{i,r,S'}}^1(\Gamma_{F,S'},\rho_{i,r}(\mathfrak g_i^{\der})) \,\right|}{\left|\,\ker (H_{\mathcal L_{i,r,S'}}^1(\Gamma_{F,S'},\rho_{i,r}(\mathfrak g_i^{\der})) \to H_{\mathcal L_{i,1,S'}}^1(\Gamma_{F,S'},\overline{\rho_i}(\mathfrak g_i^{\der})))\,\right|} \\
    & = \frac{\left|\,H_{\mathcal L_{i,r,S'}}^1(\Gamma_{F,S'},\rho_{i,r}(\mathfrak g_i^{\der}))\,\right|}{\left|\,\im (H_{\mathcal L_{i,r-1,S'}}^1(\Gamma_{F,S'},\rho_{i,r-1}(\mathfrak g_i^{\der})) \to H_{\mathcal L_{i,r,S'}}^1(\Gamma_{F,S'},\rho_{i,r}(\mathfrak g_i^{\der})))\,\right|}.
\end{align*}

\vspace{2 mm}

On the other hand, from the long exact sequence on cohomology we get an exact sequence

$$H^0(\Gamma_{F,S'},\overline{\rho_i}(\mathfrak g_i^{\der})) \longrightarrow H^1(\Gamma_{F,S'},\rho_{i,r-1}(\mathfrak g_i^{\der})) \longrightarrow H^1(\Gamma_{F,S'},\rho_{i,r}(\mathfrak g_i^{\der})).$$

\vspace{2 mm}

By the third bullet point of Assumptions \ref{GeneralAssumptions} the first term vanishes and so the second map is an inclusion. This means that the denominator of the last line above is equal to $\left|H_{\mathcal L_{i,r-1,S'}}^1(\Gamma_{F,S'},\rho_{i,r-1}(\mathfrak g_i^{\der}))\right|$. We get an analogous result for the relative dual Selmer group. Finally, by Lemma \ref{BalancedSelmer} the local conditions are balanced for $1 \le r \le n$, so we have

\begin{align*}
    \left|\,\overline{H_{\mathcal L_{i,r,S'}}^1(\Gamma_{F,S'},\rho_{i,r}(\mathfrak g_i^{\der}))}\,\right|&= \frac{\left|\,H_{\mathcal L_{i,r,S'}}^1(\Gamma_{F,S'},\rho_{i,r}(\mathfrak g_i^{\der}))\,\right|}{\left|\,H_{\mathcal L_{i,r-1,S'}}^1(\Gamma_{F,S'},\rho_{i,r-1}(\mathfrak g_i^{\der}))\,\right|} \\
    & = \frac{\left|\,H_{\mathcal L_{i,r,S'}^\perp}^1(\Gamma_{F,S'},\rho_{i,r}(\mathfrak g_i^{\der})^*)\,\right|}{\left|\,H_{\mathcal L_{i,r-1,S'}^\perp}^1(\Gamma_{F,S'},\rho_{i,r-1}(\mathfrak g_i^{\der})^*)\,\right|} \\
    &= \left|\,\overline{H_{\mathcal L_{i,r,S'}^\perp}^1(\Gamma_{F,S'},\rho_{i,r}(\mathfrak g_i^{\der})^*)}\,\right|.
\end{align*}
    
\end{proof}

\subsection{Annihilating the relative (dual) Selmer group}

For $1 \le r \le n$ we define $F_r$ to be the splitting field $F(\overline{\rho_1},\overline{\rho_2},\rho_{1,r}(\mathfrak g_1^{\der}),\rho_{2,r}(\mathfrak g_2^{\der}))$. We also define $F_n^* = F_n(\mu_{p^{n/e}})$.

\begin{lem}[cf. {\cite[Lemma $6.4$]{FKP21}}]
\label{VanishingofH1FN*}

$H^1(\Gal(F_n^*/F),\rho_{i,r}(\mathfrak g_i^{\der})^*) = 0$ for all $i=1,2$ and $r \le n$.
    
\end{lem}
\begin{proof}

We will first prove the lemma for $r=1$. We consider the inflation-restriction sequence

$$ H^1(\Gal(K/F),\overline{\rho_i}(\mathfrak g_i^{\der})^*) \longrightarrow H^1(\Gal(F_n^*/F),\overline{\rho_i}(\mathfrak g_i^{\der})^*) \longrightarrow H^1(\Gal(F_n^*/K),\overline{\rho_i}(\mathfrak g_i^{\der})^*)^{\Gal(K/F)}.$$

\vspace{2 mm}

By the first bullet point of Assumptions \ref{GeneralAssumptions} the first group vanishes. On the other hand, as $\Gamma_K$ fixes $\overline{\rho_i}$ and $\mu_p$ it acts trivially on $\overline{\rho_i}(\mathfrak g_i^{\der})^*$. Therefore

$$H^1(\Gal(F_n^*/K),\overline{\rho_i}(\mathfrak g_i^{\der})^*)^{\Gal(K/F)} = \Hom(\Gal(F_n^*/K),\overline{\rho_i}(\mathfrak g_i^{\der})^*)^{\Gal(K/F)} = \Hom_{\Gamma_F}(\Gal(F_n^*/K),\overline{\rho_i}(\mathfrak g_i^{\der})^*).$$

\vspace{2 mm}

We recall that $\rho_{i,n}$ satisfies properties $(3)$ and $(4)$ of Theorem \ref{LiftingModn}. Then, by Lemma \ref{LinearDisjointness} we get

$$\Gal(F_n^*/K) \simeq \Gal(L_1/K) \times \Gal(L_2/K) \times \Gal(K(\mu_{p^t})/K) \simeq \widehat{G_1^{\der}}(\mathcal O/\varpi^n) \times \widehat{G_2^{\der}}(\mathcal O/\varpi^n) \times \mathbb Z/p^{t-D},$$

\vspace{2 mm} 

\noindent where $t$ is the largest integer such that $\mu_{p^t} \subset F_n^*$, and for the last factor we used the fact that $D$ is the largest integer such that $K$ contains $\mu_{p^D}$. Any $\Gamma_F$-equivariant homomorphism from $\Gal(F_n^*/K)$ to $\overline{\rho_i}(\mathfrak g_i^{\der})^*$ will have to factor through the abelianization of $\Gal(F_n^*/K)$, which by Lemma \ref{AbelianizationLemma}, as an $\mathbb F_p[\Gamma_F]$-module will be isomorphic to $\overline{\rho_1}(\mathfrak g_1^{\der}) \oplus \overline{\rho_2}(\mathfrak g_2^{\der}) \oplus \mathbb Z/p^{t-D}$. Hence

$$\Hom_{\Gamma_F}(\Gal(F_n^*/K),\overline{\rho_i}(\mathfrak g_i^{\der})^*) = $$

$$\Hom_{\mathbb F_p[\Gamma_F]}(\overline{\rho_1}(\mathfrak g_1^{\der}),\overline{\rho_i}(\mathfrak g_i^{\der})^*) \oplus \Hom_{\mathbb F_p[\Gamma_F]}(\overline{\rho_2}(\mathfrak g_2^{\der}),\overline{\rho_i}(\mathfrak g_i^{\der})^*) \oplus \Hom_{\mathbb F_p[\Gamma_F]}(\mathbb Z/p^{t-D},\overline{\rho_i}(\mathfrak g_i^{\der})^*),$$

\vspace{2 mm}

\noindent which by Assumptions \ref{GeneralAssumptions} is equal to $0$. From the exact sequence we deduce that $H^1(\Gal(F_n^*/F),\overline{\rho_i}(\mathfrak g_i^{\der})^*) = 0$. 

Now, suppose that $r \ge 2$ and we have proved the claim for $r-1$. Then, from the long exact sequence on cohomology we get an exact sequence

$$H^1(\Gal(F_n^*/F),\rho_{i,r-1}(\mathfrak g_i^{\der})^*) \longrightarrow H^1(\Gal(F_n^*/K),\rho_{i,r}(\mathfrak g_i^{\der})^*) \longrightarrow H^1(\Gal(F_n^*/K),\overline{\rho_i}(\mathfrak g_i^{\der})^*).$$

\vspace{2 mm}

By the inductive hypothesis the end terms are trivial, hence so is the middle one, which proves the claim.
    
\end{proof}

Recall that we have defined the field $\widetilde{F}$ to be the smallest extension of $F$ such that $\overline{\rho}_i(\Gamma_{\widetilde{F}}) \subseteq G_i^0(k)$ for both $i=1,2$. We now let $\Gamma$ be the inverse image in $G_1^{\ad}(\mathcal O) \times G_2^{\ad}(\mathcal O)$ of $\Ad \circ (\overline{\rho_1} \times \overline{\rho_2})(\Gamma_{\widetilde{F}}) \subseteq G_1^{\ad}(k) \times G_2^{\ad}(k)$. We now apply \cite[Corollary B.$2$]{FKP21} to the $p$-adic Lie group $\Gamma$, the $\Gamma$-module $\mathrm{Lie}(\Gamma) = \mathfrak g_1^{\der} \oplus \mathfrak g_2^{\der}$, and $m=1$. The corollary will produce an integer $M_1$ such that for $m \ge M_1$ the reduction map 

$$H^1(\Gamma,\mathfrak g_i^{\der} \otimes_{\mathcal O} \mathcal O/\varpi^m) \longrightarrow H^1(\Gamma,\mathfrak g_i^{\der} \otimes_{\mathcal O} k)$$

\vspace{2 mm}

\noindent is the zero map, where the action is given via the $i$-th factor. We remark that this integer $M_1$ depends only on the images of the residual representation $\overline{\rho_1}$ and $\overline{\rho_2}$. We now fix $m \ge M_1$, which moreover is divisible by $e$. If needed, we increase the value of $n$, the integer we introduced at the beginning of \S $4$, so that from now on it satisfies $n \ge 16m$.

\begin{lem}[cf. {\cite[Lemma $6.4$]{FKP21}}]
\label{H^1ZeroMap}

The map $H^1(\Gal(F_n^*/F),\rho_{i,m}(\mathfrak g_i^{\der})) \to H^1(\Gal(F_n^*/F),\overline{\rho_i}(\mathfrak g_i^{\der}))$ is the zero map for $i=1,2$.
    
\end{lem}
\begin{proof}

    We consider the inflation-restriction sequence

    $$0 \longrightarrow H^1(\Gal(F_n/F),\rho_{i,m}(\mathfrak g_i^{\der})) \longrightarrow H^1(\Gal(F_n^*/F),\rho_{i,m}(\mathfrak g_i^{\der})) \longrightarrow H^1(\Gal(F_n^*/F_n),\rho_{i,m}(\mathfrak g_i^{\der}))^{\Gal(F_n/F)}.$$

    \vspace{2 mm}

    Since $m \le n$ the $\Gal(F_n^*/F_n)$-action on $\rho_{i,m}(\mathfrak g_i^{\der})$ is trivial. Therefore, the last term is equal to \linebreak $\Hom_{\Gamma_F}(\Gal(F_n^*/F_n),\rho_{i,m}(\mathfrak g_i^{\der}))$. As $F(\mu_{p^{n/e}})/F$ is abelian by \cite[Lemma $2.1$]{FKP22} the $\Gamma_F$-action on $\Gal(F_n^*/F_n)$ is trivial. On the other hand, by part $(b)$ of Lemma \ref{rhoNProperties} $\rho_{i,m}(\mathfrak g_i^{\der})$ doesn't contain the trivial representation, which means that the group of $\Gamma_F$-equivariant homomorphisms from $\Gal(F_n^*/F_n)$ to $\rho_{i,m}(\mathfrak g_i^{\der})$ is trivial. Hence, we obtain an isomorphism $H^1(\Gal(F_n/F),\rho_{i,m}(\mathfrak g_i^{\der})) \simeq H^1(\Gal(F_n^*/F),\rho_{i,m}(\mathfrak g_i^{\der}))$. From the functoriality of the inflation-restriction sequence we get the commutative diagram

    $$\begin{tikzcd}
    {H^1(\Gal(F_n/F),\rho_{i,m}(\mathfrak g_i^{\der}))} \arrow[d] \arrow[rr] &  & {H^1(\Gal(F_n^*/F),\rho_{i,m}(\mathfrak g_i^{\der}))} \arrow[d] \\
    {H^1(\Gal(F_n/F),\overline{\rho_i}(\mathfrak g_i^{\der}))} \arrow[rr]    &  & {H^1(\Gal(F_n^*/F),\overline{\rho_i}(\mathfrak g_i^{\der}))}                          
    \end{tikzcd}$$

    \vspace{2 mm}

    By the above observation the horizontal maps are isomorphisms. Therefore, in order to show that the second vertical map is the zero map it is enough to show that for the first one. Furthermore, we note that both $\Gal(\widetilde{F}/F)$ and $\Gal(F_n/\widetilde{F}(\rho_{1,n}(\mathfrak g_1^{\der}),\rho_{2,n}(\mathfrak g_2^{\der})))$ are groups of order coprime to $p$. For the first one, from the definition of $\widetilde{F}$ the map $\overline{\rho_1} \times \overline{\rho_2}$ induces an injection $\Gal(\widetilde{F}/F) \hookrightarrow \pi_0(G_1) \times \pi_0(G_2)$. By our assumption on $G_1$ and $G_2$ the groups $\pi_0(G_i)$ and $\pi_0(G_i)$ have order coprime to $p$, hence the same is true for $\Gal(\widetilde{F}/F)$. For the second one, we note that the map $\overline{\rho_1} \times \overline{\rho_2}$ induces an injection $\Gal(F_n/\widetilde{F}(\rho_{1,n}(\mathfrak g_1^{\der}),\rho_{2,n}(\mathfrak g_2^{\der}))) \hookrightarrow Z_{G_1^0}(k) \times Z_{G_2^0}(k)$. To see this we start with the obvious injection $\Gal(F_n/F(\rho_{1,n}(\mathfrak g_1^{\der}),\rho_{2,n}(\mathfrak g_2^{\der}))) \hookrightarrow G_1(k) \times G_2(k)$ induced by $\overline{\rho_1} \times \overline{\rho_2}$. By the definition of $\widetilde{F}$ any $\sigma$ that fixes it must have image in $G_i^0(k)$ under $\overline{\rho_i}$. This yields an injection $\Gal(F_n/\widetilde{F}(\rho_{1,n}(\mathfrak g_1^{\der}),\rho_{2,n}(\mathfrak g_2^{\der}))) \hookrightarrow G_1^0(k) \times G_2^0(k)$. Moreover, any $\sigma$ that fixes $\widetilde{F}(\rho_{1,n}(\mathfrak g_1^{\der}),\rho_{2,n}(\mathfrak g_2^{\der}))$ must have an image under $\rho_{i,n}$ that lies in the kernel of $\Ad \circ \rho_{i,n}$, which is exactly $Z_{G_i}(G^0_{i})(\mathcal O/\varpi^n)$, the centralizer of the connected component. By our previous observation this image already lies in $G^0_i(\mathcal O/\varpi^n)$ so it must lie in the center $Z_{G_i^0}(\mathcal O/\varpi^n)$. Hence, its image under $\overline{\rho_i}$ lies in $Z_{G_i^0}(k)$ giving us the desired injection. By our assumptions on $G_1$ and $G_2$ we know that $Z_{G_i^0}(k)$ has order coprime to $p$, so the same is true for $\Gal(F_n/\widetilde{F}(\rho_{1,n}(\mathfrak g_1^{\der}),\rho_{2,n}(\mathfrak g_2^{\der})))$. Therefore, from the inflation-restriction sequence it is enough to prove the claim with $\Gal(\widetilde{F}(\rho_{1,n}(\mathfrak g_1^{\der}),\rho_{2,n}(\mathfrak g_2^{\der}))/\widetilde{F})$ in place of $\Gal(F_n/F)$. We now consider the injection 

    $$\Gal(\widetilde{F}(\rho_{1,n}(\mathfrak g_1^{\der}),\rho_{2,n}(\mathfrak g_2^{\der}))/\widetilde{F}) \hookrightarrow \Gal(\widetilde{F}(\rho_{1,n}(\mathfrak g_1^{\der}))/\widetilde{F}) \times \Gal(\widetilde{F}(\rho_{2,n}(\mathfrak g_2^{\der}))/\widetilde{F}) \xhookrightarrow{\rho_{1,n} \times \rho_{2,n}} G_1^{\ad}(\mathcal O/\varpi^n) \times G_2^{\ad}(\mathcal O/\varpi^n).$$

    \vspace{2 mm}

    Hence, we can identify $\Gal(\widetilde{F}(\rho_{1,n}(\mathfrak g_1^{\der}),\rho_{2,n}(\mathfrak g_2^{\der}))/\widetilde{F})$ with a subgroup $\overline{\Gamma}$ of $G_1^{\ad}(\mathcal O/\varpi^n) \times G_2^{\ad}(\mathcal O/\varpi^n)$. With this notation we want to prove that $H^1(\overline{\Gamma},\mathfrak g_i^{\der} \otimes_{\mathcal O} \mathcal O/\varpi^m) \to H^1(\overline{\Gamma},\mathfrak g_i^{\der} \otimes_{\mathcal O} k)$ is the zero map, where the $\overline{\Gamma}$-action is the adjoint action via the $i$-th factor. 

    We claim that the inverse image of $\overline{\Gamma}$ in $G_1^{\ad}(\mathcal O) \times G_2^{\ad}(\mathcal O)$ is the group $\Gamma$, recalling that $\Gamma$ was defined as the inverse image in $G_1^{\ad}(\mathcal O) \times G_2^{\ad}(\mathcal O)$ of $\Ad \circ (\overline{\rho_1} \times \overline{\rho_2})(\Gamma_{\widetilde{F}}) \subseteq G_1^{\ad}(k) \times G_2^{\ad}(k)$. Clearly, the inverse image of $\overline{\Gamma}$ will be contained in $\Gamma$. Conversely, let $(g_1,g_2) \in \Gamma$. This means that there exists $\sigma \in \Gamma_{\widetilde{F}}$ such that $g_i \pmod{\varpi} = \Ad(\overline{\rho_i}(\sigma))$ for $i=1,2$. Writing $\overline{g_i}$ for $g_i \pmod{\varpi^n}$ we have that $\Ad(\rho_{i,n}(\sigma))\overline{g_i}^{-1} \pmod{\varpi} = 1$, and so $\Ad(\rho_{i,n}(\sigma))\overline{g_i}^{-1} \in \widehat{G^{\ad}_i}(\mathcal O/\varpi^n) = \widehat{G^{\der}_i}(\mathcal O/\varpi^n)$, with the last equality coming from our assumption that the central isogeny $G_i^{\der} \times Z_{G_i^0} \to G_i^0$ has kernel of order coprime to $p$. By parts $(3)$ and $(4)$ of Theorem \ref{DoublingMethod} we can find $\alpha \in \Gamma_F$ such that $\Ad(\rho_{i,n}(\sigma))\overline{g_i}^{-1} = \Ad(\rho_{i,n}(\alpha))$ for both $i=1,2$. We also note that $\alpha \in \Gamma_{\widetilde{F}}$ since $\overline{\rho_i}(\alpha) = 1$ modulo the center. Thus, $(g_1,g_2)$ is the preimage of $\Ad(\rho_{1,n}\times\rho_{2,n})(\sigma\alpha^{-1}) \in \overline{\Gamma}$ in $G_1^{\ad}(\mathcal O) \times G_2^{\ad}(\mathcal O)$. Therefore, $\overline{\Gamma}$ is a quotient of $\Gamma$. So, from the functoriality of the inflation-restriction sequence we obtain the commutating diagram

    $$\begin{tikzcd}
0 \arrow[rr] &  & H^1(\overline{\Gamma},\mathfrak g_i^{\der} \otimes_\mathcal O \mathcal O/\varpi^m) \arrow[rr] \arrow[d] &  & {H^1(\Gamma,\mathfrak g_i^{\der} \otimes_\mathcal O \mathcal O/\varpi^m)} \arrow[d] \\
0 \arrow[rr] &  & {H^1(\overline{\Gamma},\mathfrak g_i^{\der} \otimes_\mathcal O k)} \arrow[rr]        &  & {H^1(\Gamma,\mathfrak g_i^{\der} \otimes_\mathcal O k)}          
\end{tikzcd}$$

\vspace{2 mm}

    Thus, to show that the first vertical map is zero it is enough to show that the second vertical map is the zero map. This follows from our choice of $M_1$ and the requirement that $m \ge M_1$.
    
\end{proof}

In what follows our goal is to annihilate the $m$-th relative Selmer groups associated to $\rho_{1,m}$ and $\rho_{2,m}$. As in the doubling method we want to achieve this by allowing extra ramification at the same primes for both representations. We define the following set of auxiliary primes:

\begin{defn}
\label{PrimesQ_n}

Let $Q_{n,m}$ be the set of trivial primes $\nu$ of $F$ that satisfy the following properties:

\begin{itemize}
    \item $N(\nu) \equiv 1 \pmod{\varpi^m}$, but $N(\nu) \not \equiv 1 \pmod{\varpi^{m+1}}$.
    \item For $i=1,2$ $\restr{\rho_{i,n}}{\Gamma_{F_\nu}}$ is unramified (with multiplier $\mu$) and $\restr{\rho_{i,m}}{\Gamma_{F_\nu}}$ is trivial modulo the center.
    \item There exists a split maximal torus $T_i$ of $G_i$ and a root $\alpha_i \in \Phi(G^0_i,T_i)$ such that $\rho_{i,n}(\sigma_\nu) \in T_i(\mathcal O/\varpi^n)$ and $\alpha_i(\rho_{i,n}(\sigma_\nu)) \equiv \kappa(\sigma_\nu) \equiv N(\nu) \pmod{\varpi^n}$ for $i=1,2$.
    \item For $i=1,2$ and all roots $\beta \in \Phi(G^0_i,T_i)$ we have $\beta(\rho_{i,m+1}(\sigma_\nu)) \not \equiv 1 \pmod{\varpi^{m+1}}$.
\end{itemize}
    
\end{defn}

The set $Q_{m,n}$ is given by a Chebotarev condition in $F_n^*/F_m^*$. The first property translates to a condition on the Frobenius in $K(\mu_{p^{n/e}})$. We only ask for $\sigma_\nu$ to be trivial in $K(\mu_{p^{m/e}})$, but non-trivial in $K(\mu_{p^{m/e+1}})$. The rest of the properties can be arranged by giving a Chebotarev condition in $\Gal(K(\rho_{1,n}(\mathfrak g_1^{\der}),\rho_{2,n}(\mathfrak g_2^{\der}))/K)$. By parts $(3)$ and $(4)$ of Theorem \ref{LiftingModn} this Galois group is isomorphic to $\widehat{G_1^{\der}}(\mathcal O/\varpi^n) \times \widehat{G_2^{\der}}(\mathcal O/\varpi^n)$. We then ask for $\sigma_\nu$ to lie in the set $\{(g_1,g_2) \mid g_i \in T_i(\mathcal O/\varpi^n)$ for some maximal torus $T_i$ of $G_i$, $g_i \equiv 1 \pmod{\varpi^m}$, for some $\alpha_i \in \Phi(G^0_i,T_i)$ we have $\alpha_i(g_i) \equiv N(\nu) \pmod{\varpi^n}$, and $\beta(g_i) \not \equiv 1 \pmod{\varpi^{m+1}}$ for all $\beta \in \Phi(G^0_i,T_i)\}$. As we've seen in the proof of Theorem \ref{LiftingModn} such elements $g_i$ exist by Lemma \ref{TorusElementLemma}. We remark that this is indeed a Chebotarev condition, i.e. this set is closed under conjugation. In fact, if $g_i$ satisfies the desired properties with respect to $(T_i,\alpha_i)$, then $gg_ig^{-1}$ will satisfy the same properties with respect to $\Ad(g)(T_i,\alpha_i)$ for any $g \in \widehat{G_i^{\der}}(\mathcal O/\varpi^n)$. Finally, by the linear disjointness of Lemma \ref{LinearDisjointness} these two Chebotarev conditions are compatible. 

At each prime $\nu$ in $Q_{n,m}$ we want to produce lifts of the form $\Lift_{\restr{\overline{\rho_i}}{\Gamma_{F_\nu}}}^{\mu_i,\alpha_i}$ (see Definition \ref{LiftFunctorDefinition}) for $i=1,2$. To achieve this we will define local conditions $L_{i,r,\nu}^{\alpha_i}$, which we do using the cocycle subgroups $Z_{i,r,\nu}^{\alpha_i}\subseteq Z^1(\Gamma_{F_\nu},\rho_{i,r}(\mathfrak g_i^{\der}))$ produced by \cite[Lemma $3.5$]{FKP21}.

We now assume that $\mathfrak g_i^{\der}$ consists of a single $\pi_0(G_i)$-orbit of simple factors. As shown in \cite[Theorem $6.11$]{FKP21} in order to produce $p$-adic lifts of $\rho_{1,n}$ and $\rho_{2,n}$ it is enough to do that in this special case. We now have the following analogue of \cite[Proposition $6.8$]{FKP21} which we use to select the primes at which we will allow extra ramification in order to annihilate the $m$-th relative Selmer groups attached to both $\rho_{1,m}$ and $\rho_{2,m}$.

\begin{prop}
\label{SelectionofPrimes}

Let $Q$ be any finite set of $Q_{n,m}$. For both $i=1,2$ let $\phi_i \in H^1_{\mathcal L_{i,m,S'\cup Q}}(\Gamma_{F,S'\cup Q},\rho_{i,m}(\mathfrak g_i^{\der}))$ and $\psi_i \in H^1_{\mathcal L_{i,m,S'\cup Q}^\perp}(\Gamma_{F,S'\cup Q},\rho_{i,m}(\mathfrak g_i^{\der})^*)$ be such that $\overline{\phi_i} \neq 0$ and $\overline{\psi_i} \neq 0$. There exists a prime $\nu \in Q_{n,m} \setminus Q$, with associated tori $T_1$ and $T_2$, and roots $\alpha_1 \in \Phi(G^0_1,T_1)$ and $\alpha_2 \in \Phi(G^0_2,T_2)$ such that

\begin{itemize}
    \item $\restr{\overline{\psi_i}}{\Gamma_{F_\nu}} \notin L_{i,1,\nu}^{\alpha_i,\perp}$ for both $i=1,2$.
    \item $\restr{\phi_i}{\Gamma_{F_\nu}} \notin L_{i,m,\nu}^{\alpha_i}$ for both $i=1,2$.
\end{itemize}
\end{prop}

\begin{proof}

    We first note that we are proving asymmetric results for cocycles and dual cocycles. In particular, we can only prove $\restr{\phi_i}{\Gamma_{F_\nu}} \notin L_{i,m,\nu}^{\alpha_i}$, but not $\restr{\overline{\phi_i}}{\Gamma_{F_\nu}} \notin L_{i,1,\nu}^{\alpha_i}$ for both $i=1,2$. This is a consequence of the differences between Lemma \ref{VanishingofH1FN*} and Lemma \ref{H^1ZeroMap}.
    
    We claim that $\restr{\overline{\psi_i}}{\Gamma_{F_n^*}}$ and $\restr{\phi_i}{\Gamma_{F_n^*}}$ are all non-zero. The claim about the dual cocycles $\overline{\psi_i}$ follows directly from Lemma \ref{VanishingofH1FN*} and the inflation-restriction sequence. On the other hand, if $\restr{\phi_i}{\Gamma_{F_n^*}}$ is zero, then from the inflation-restriction sequence $\phi_i$ corresponds to a class in $H^1(\Gal(F_n^*/F),\rho_{i,m}(\mathfrak g_i^{\der}))$. But, then by Lemma \ref{H^1ZeroMap} its restriction $\overline{\phi_i}$ will be zero, contradicting our hypothesis on $\phi_i$.


    We now claim that we can find $\gamma_2 \in \Gamma_{F_n^*}$ such that $\phi_1(\gamma_2),\phi_2(\gamma_2),\overline{\psi_1}(\gamma_2)$, and $\overline{\psi_2}(\gamma_2)$ are all non-zero. We know that $F_n^*(\phi_1)/F_n^*$ and $F_n^*(\phi_2)/F_n^*$ are both non-trivial Galois extensions. We can then find an element in $\Gal(F_n^*(\phi_1,\phi_2)/F_n^*)$ which will act non-trivially on both subextensions $F_n^*(\phi_i)$. This is always possible for the composite of two non-trivial extensions. It can be easily done on a case-by-case basis depending on whether one of the extensions is contained in the other or not. Similarly, we can find an element $\Gal(F_n^*(\overline{\psi_1},\overline{\psi_2})/F_n^*)$ which acts non-trivially on both $F_n^*(\overline{\psi_i})$. Next, we note that these two extensions are linearly disjoint over $F_n^*$. Indeed

    $$\Gal(F_n^*(\phi_1,\phi_2)/F_n^*) \hookrightarrow \Gal(F_n^*(\phi_1)/F_n^*) \times \Gal(F_n^*(\phi_2)/F_n^*) \simeq \im(\restr{\phi_1}{\Gamma_{F_n^*}}) \times \im(\restr{\phi_2}{\Gamma_{F_n^*}}) \subseteq \rho_{1,m}(\mathfrak g_1^{\der}) \oplus \rho_{2,m}(\mathfrak g_2^{\der}).$$

    \vspace{2 mm}

    Similarly, $\Gal(F_n^*(\overline{\psi_1},\overline{\psi_2})/F_n^*)$ is an $\mathbb F_p[\Gamma_F]$-submodule of $\overline{\rho_1}(\mathfrak g_1^{\der})^* \oplus \overline{\rho_2}(\mathfrak g_2^{\der})^*$. Hence, if the intersection of $F_n^*(\phi_1,\phi_2)$ and $F_n^*(\overline{\psi_1},\overline{\psi_2})$ is a non-trivial extension of $F_n^*$ it will give rise to a common $\mathbb F_p[\Gamma_F]$-subquotient of $\rho_{i,m}(\mathfrak g_i^{\der})$ and $\overline{\rho_j}(\mathfrak g_j^{\der})^*$, which contradicts part $(1)$ of Lemma \ref{rhoNProperties}. This linear disjointness allows us to find the desired element $\gamma_2$ in $\Gal(F_n^*(\phi_1,\phi_2,\overline{\psi_1},\overline{\psi_2})/F_n^*)$. We now claim that we can find a split maximal torus $T_i$ of $G_i$ and $\alpha_i \in \Phi(G_i^0,T_i)$ such that for both $i=1,2$ we have

    $$\phi_i(\gamma_2) \notin \ker(\restr{\alpha_i}{\mathfrak t_i}) \oplus \bigoplus_{\beta \in \Phi(G_i^0,T)} (\mathfrak g_i)_\beta \quad \quad \quad \text{and} \quad\quad\quad\overline{\psi_i}(\gamma_2) \notin ((\mathfrak g_i)_{\alpha_i})^\perp.$$

    \vspace{2 mm}

    We will first show that it is enough to prove this claim when $G_i$ is connected and simple. We fix a maximal torus $T_i$ of $G_i$ and a root $\alpha_i \in \Phi(G_i^0,T_i)$. By our assumptions $\mathfrak g_i^{\der}$ consists of a single $\pi_0(G_i)$-orbit of simple factors and $\Gamma_F$ surjects onto $\pi_0(G_i)$. Then, from the $\Gamma_F$-equivariance of $\restr{\phi_i}{\Gamma_{F_n^*}}$ and $\restr{\overline{\psi_i}}{\Gamma_{F_n^*}}$ we get that all of them have a non-trivial projection onto each simple factor of $\mathfrak g_i^{\der}$. Therefore, after replacing $\mathfrak g_i^{\der}$ with the simple factor that contains $(\mathfrak g_i^{\der})_{\alpha_i}$, and $\restr{\phi_i}{\Gamma_{F_n^*}}$ and $\restr{\overline{\psi_i}}{\Gamma_{F_n^*}}$  with their non-trivial projections on this simple factor, we can assume that $\mathfrak g_i^{\der}$ is simple. Thus, by taking $G_i^0$ in place of $G_i$ we can assume that $G_i$ is connected and simple in the (purely Lie-theoretic) proof of the claim above.
    
    We note that $\phi_i(\gamma_2)$ might be trivial modulo $\varpi$, however, as it is non-zero we can write $\phi_i(\gamma_2) = \varpi^{r_i}X_i$ for some $r_i < m$ and $X_i \in \mathfrak g_i^{\der}$ such that $X_i \pmod{\varpi} \neq 0$. By Lemma \ref{TorusandRootLemma} for $p \gg_{G_1} 0$ we can find $\overline{g} \in G_1(k)$ such that $\Ad(\overline{g}^{-1})\overline{X_1} \notin \ker(\restr{\alpha_1}{\mathfrak t_1}) \oplus \bigoplus (\mathfrak g_1)_\beta$ and $\Ad(\overline{g}^{-1})\overline{\psi_1}(\gamma_2) \notin ((\mathfrak g_1)_{\alpha_1})^\perp$. Using the smoothness of $G_1$ we can lift $\overline{g}$ to $g \in G_1(\mathcal O/\varpi^m)$. Then, writing $(T_g,\alpha_g)$ for $\Ad(g)(T_1,\alpha_1)$ we get that $\phi(\gamma_2) \notin \ker(\restr{\alpha_g}{\mathfrak t_g}) \oplus \bigoplus_{\beta \in \Phi(G_i^0,T)} (\mathfrak g_1)_\beta$ and $\overline{\psi}(\gamma_2) \notin ((\mathfrak g_1)_{\alpha_g})^\perp$. Replacing $(T_1,\alpha_1)$ with $(T_g,\alpha_g)$ we get exactly what we want. A similar argument allows us to find a split maximal torus $T_2$ of $G_2$ and $\alpha_2 \in \Phi(G_2,T_2)$ such that $\phi_2(\gamma_2) \notin \ker(\restr{\alpha_2}{\mathfrak t_2}) \oplus \bigoplus (\mathfrak g_2)_\beta$ and $\overline{\psi_2}(\gamma_2) \notin ((\mathfrak g_2)_{\alpha_2})^\perp$

    As explained in the discussion following Definition \ref{PrimesQ_n} using a Chebotarev condition in $F_n^*/F$ we can find an element $\gamma_1 \in \Gamma_F$ which will produce primes that satisfy the conditions of the definition with respect to the specific $(T_1,\alpha_1)$ and $(T_2,\alpha_2)$. Now, by Lemma \ref{NonInclusionLemma}, under the assumption that $p \ge 5$ there exist integers $0 \le a \le 4$ such that

    $$\phi_1(\gamma_2^a\gamma_1) = a\phi_1(\gamma_2) + \phi_1(\gamma_1) \notin \ker (\restr{\alpha_1}{\mathfrak t_1}) \oplus \bigoplus_{\beta \in \Phi(G^0_1,T_1)} (\mathfrak g_1)_\beta;$$

    $$\phi_2(\gamma_2^a\gamma_1) = a\phi_2(\gamma_2) + \phi_2(\gamma_1) \notin \ker (\restr{\alpha_2}{\mathfrak t_2}) \oplus \bigoplus_{\beta \in \Phi(G^0_2,T_2)} (\mathfrak g_2)_\beta;$$

    $$\overline{\psi_1}(\gamma_2^a\gamma_1) = a\overline{\psi_1}(\gamma_2) + \overline{\psi_1}(\gamma_1)\notin ((\mathfrak g_1)_{\alpha_1})^\perp;$$

    $$\overline{\psi_2}(\gamma_2^a\gamma_1) = a\overline{\phi_2}(\gamma_2) + \overline{\psi_2}(\gamma_1) \notin ((\mathfrak g_2)_{\alpha_2})^\perp.$$

    \vspace{2 mm}

    We explain in more detail how we apply Lemma \ref{NonInclusionLemma},. The rings $R_i$ are either $k$ or $\mathcal O/\varpi^m$, the modules $M_i$ are the adjoint and the dual adjoint representations, the submodule $M_i'$ are the specified hyperplanes we want to avoid, while the elements $y_i$ are the images of $\gamma_1$ under each cocycle. In the notation of the lemma we take the elements $x_i$ to be the images of $\gamma_2$ under each of the cocycles.

    Finally, by \cite[Lemma $6.7$]{FKP21} any prime $\nu$ whose Frobenius is equal to $\gamma_2^{a}\gamma_1$ in $\Gal(F_n^*(\phi_1,\phi_2,\overline{\psi_1},\overline{\psi_2})/F)$ will satisfy the desired conditions with respect to $(T_1,\alpha_1)$ and $(T_2,\alpha_2)$. As such primes satisfy a Chebotarev condition we pick them from a positive density set. In particular, we can choose $\nu$ such that $\nu \notin Q$.
    
\end{proof}

We now annihilate the relative Selmer groups attached to lifts of $\overline{\rho_1}$ and $\overline{\rho_2}$ by allowing extra ramification at primes in a finite subset $Q$ of $Q_{n,m}$. To ease off the notation and make clearer the procedure of annihilating the relative Selmer groups simultaneously we use the notation $(a,b)$, where $a$ is the $k$-dimension of $\overline{H^1_{\mathcal L_{1,m,S'}}(\Gamma_{F,S'},\rho_{1,m}(\mathfrak g_1^{\der}))}$, and $b$ is the $k$-dimension of $\overline{H^1_{\mathcal L_{2,m,S'}}(\Gamma_{F,S'},\rho_{2,m}(\mathfrak g_2^{\der}))}$. We leave the set of ramifying primes implicit in this notation, using the same notation even after we allow extra ramification at some finite set $Q$. Our goal is to get to the case $(0,0)$ after allowing some extra ramification. We will first show how to reduce to the special case $(1,0)$.

Suppose that we are in the case $(a,b)$, where both $a$ and $b$ are non-zero. Using Proposition \ref{SelectionofPrimes} we produce a prime $\nu \in Q_{n,m}$ satisfying the conditions in its statement. Then, arguing as in \cite[Theorem $6.9$]{FKP21} by allowing ramification at $\nu$ with local conditions $L_{i,r,\nu}^{\alpha_i}$ we can decrease the size of the $m$-th Selmer groups attached to $\rho_{1,m}$ and $\rho_{2,m}$ (not necessarily the dimension of the $m$-th relative Selmer groups). We can continue running the same argument until one of the $m$-th relative Selmer group is zero. Thus, without loss of generality we can assume that we are in the case $(a,0)$, where $a \neq 0$. Running a version of Proposition \ref{SelectionofPrimes} for one representation (this is exactly the statement of \cite[Proposition $6.8$]{FKP21}) we can produce a prime $\nu \in Q_{n,m}$ such that allowing ramification at $\nu$ will decrease the size of the $m$-th Selmer group $H^1_{\mathcal L_{1,m,S'}}(\Gamma_{F,S'},\rho_{1,m}(\mathfrak g_1^{\der}))$, while we note that $a$ might not necessarily decrease. If $b=0$ even after allowing ramification at $\nu$ then we are in the same situation as before, but having decreased the size of the $m$-th Selmer group for the first representation. We note that we can't stay in this case indefinitely, as eventually we will annihilate the $m$-th Selmer group for the first representation. Thus, after allowing ramification at some finite set of primes we end up in the case $(0,0)$, which is what we want, or $(a',b)$ for some $a' \le a$ and $b \neq 0$. The next lemma tells us that $b$ must be equal to $1$.

\begin{lem}
\label{DimensionofRelSelmer}

For $i=1,2$ let $\overline{H^1_{\mathcal L_{i,r,S'}}(\Gamma_{F,S'},\rho_{i,r}(\mathfrak g_i^{\der}))}$ have $k$-dimension $d$. Then for $\nu \in Q_{n,r}$ with local conditions $L_{i,r,\nu}^{\alpha_i}$ the relative Selmer group $\overline{H^1_{\mathcal L_{i,r,S' \cup \nu}}(\Gamma_{F,S' \cup \nu},\rho_{i,r}(\mathfrak g_i^{\der}))}$ has $k$-dimension $\le d+1$.

\end{lem}
\begin{proof}

    Set $L_{i,r,\nu}' \coloneqq L_{i,r,\nu}^{\alpha_i} + L_{i,r,\nu}^{\mathrm{unr}}$. We then have a commuting diagram of exact sequences

    \begin{center}
        \begin{tikzcd}
            0 \arrow[r] & {H^1_{\mathcal L_{i,r,S'}}(\Gamma_{F,S'},\rho_{i,r}(\mathfrak g_i^{\der}))} \arrow[r] \arrow[d] & {H^1_{\mathcal L_{i,r,S'} \cup L_{i,r,\nu}'}(\Gamma_{F,S' \cup \nu},\rho_{i,r}(\mathfrak g_i^{\der}))} \arrow[d] \arrow[r] & {L_{i,r,\nu}'/L_{i,r,\nu}^{\mathrm{unr}}} \arrow[d] \\
            0 \arrow[r] & {H^1_{\mathcal L_{i,1,S'}}(\Gamma_{F,S'},\overline{\rho_i}(\mathfrak g_i^{\der}))} \arrow[r]                       & {H^1_{\mathcal L_{i,1,S'} \cup L_{i,1,\nu}'}(\Gamma_{F,S' \cup \nu},\overline{\rho_i}(\mathfrak g_i^{\der}))} \arrow[r]           & {L_{i,1,\nu}'/L_{i,1,\nu}^{\mathrm{unr}}}          
        \end{tikzcd}
    \end{center}
    
    \vspace{2 mm}

    Here the vertical maps are just reductions modulo $\varpi$, while the last horizontal maps are restrictions to $\Gamma_{F_\nu}$. As $\restr{\rho_{i,r}}{\Gamma_{F_\nu}}$ is trivial modulo the center, by \cite[Lemma $6.7$]{FKP21} ${L_{i,r,\nu}'/L_{i,r,\nu}^{\mathrm{unr}}}$ is isomorphic to $\mathcal O/\varpi^r$ and moreover is generated by the ramified cocycle in $L_{i,r,\nu}^{\alpha_i}$. (see \cite[Lemma $3.5$]{FKP21}). Therefore, for any two classes $\phi_1,\phi_2 \in H^1_{\mathcal L_{i,r,S'} \cup L_{i,r,\nu}'}(\Gamma_{F,S' \cup \nu},\rho_{i,r}(\mathfrak g_i^{\der}))$ we can find $a \in \mathcal O/\varpi^r$ such that $\phi_1 = a\phi_2 + \phi$ (or possibly $\phi_2 = a\phi_1 + \phi$) for some $\phi \in H^1_{\mathcal L_{i,r,S'}}(\Gamma_{F,S'},\rho_{i,r}(\mathfrak g_i^{\der}))$. Therefore, $\overline{\phi_1}$ and $\overline{\phi_2}$ differ by a scalar multiple modulo an element of $\overline{H^1_{\mathcal L_{i,r,S'}}(\Gamma_{F,S'},\rho_{i,r}(\mathfrak g_i^{\der}))}$. By our assumption this relative Selmer group has $k$-dimension $d$, so we deduce that $\overline{H^1_{\mathcal L_{i,r,S'} \cup L_{i,r,\nu}'}(\Gamma_{F,S' \cup \nu},\rho_{i,r}(\mathfrak g_i^{\der}))}$ has $k$-dimension $\le d+1$.  The same would be true for its subspace $\overline{H^1_{\mathcal L_{i,r,S' \cup \nu}}(\Gamma_{F,S' \cup \nu},\rho_{i,r}(\mathfrak g_i^{\der}))}$.
    
\end{proof}

If $a' = 0$ the we have reduced to the special case, which is what we want. Otherwise, we are in the case $(a',1)$ and both $m$-th relative Selmer groups are non-trivial. As in the first case we have considered above, by allowing ramification at primes produced by Proposition \ref{SelectionofPrimes} we can keep reducing the sizes of the $m$-th Selmer groups until one of the $m$-th relative Selmer groups is $0$. Then we end up either in the case $(0,1)$, which is what we want, or $(a'',0)$ for some $a'' \le a' \le a$. We note that it is possible that equality holds, i.e. we haven't decreased the size of the $m$-th relative Selmer group for the first representation. However, allowing ramification at each of the extra primes will decrease the size of the $m$-th Selmer group for the first representation. Therefore, we are in a similar situation as the one we have started, i.e. $(a,0)$, but the size of the $m$-th Selmer group for the first representation is smaller. As noticed previously we can't stay in this case indefinitely, so we will eventually end up in the cases $(0,0)$, $(1,0)$, or $(0,1)$.

It remains to consider the special case $(1,0)$. To simplify the notation we move to $S'$ all the primes where we allowed extra ramification during the possible reduction to this case. We take care of the special case using the following result:

\begin{prop}
\label{1-0Annihilation}

Suppose that $\overline{H^1_{\mathcal L_{1,m,S'}}(\Gamma_{F,S'},\rho_{1,m}(\mathfrak g_1^{\der}))}$ has dimension $1$, while $\overline{H^1_{\mathcal L_{2,m,S'}}(\Gamma_{F,S'},\rho_{2,m}(\mathfrak g_2^{\der}))}$ is trivial. Then, there exists a finite set of primes $Q$, disjoint from $S'$ and consisting of primes in $Q_{n,r}$ for varying $m \le r \le 4m$ such that 

$$\overline{H^1_{\mathcal L_{i,4m,S' \cup Q}}(\Gamma_{F,S' \cup Q},\rho_{i,4m}(\mathfrak g_i^{\der}))} = 0 \quad \quad \quad \text{for } i =1,2.$$
   
\end{prop}
\begin{proof}

    By our assumption $\overline{H^1_{\mathcal L_{1,4m,S'}}(\Gamma_{F,S'},\rho_{1,4m}(\mathfrak g_1^{\der}))}$ will have dimension $\le 1$. If this relative Selmer group is trivial we are done by taking $Q = \emptyset$. Thus, we can assume that it has dimension $1$. Let $\overline{\phi_{1,4m}}$ be a generator of it. By Lemma \ref{BalancedRelativeSelmer}, $\overline{H^1_{\mathcal L_{1,4m,S'}^\perp}(\Gamma_{F,S'},\rho_{1,4m}(\mathfrak g_1^{\der})^*)}$ will also be one-dimensional. Label a generator of it by $\overline{\psi_{1,4m}}$. Then, for $m \le r \le 4m$ we get that  $\overline{H^1_{\mathcal L_{1,r,S'}}(\Gamma_{F,S'},\rho_{1,r}(\mathfrak g_1^{\der}))}$  and $\overline{H^1_{\mathcal L_{1,r,S'}^\perp}(\Gamma_{F,S'},\rho_{1,r}(\mathfrak g_1^{\der})^*)}$ have dimension 1 and are generated by $\overline{\phi_{1,r}}$ and $\overline{\psi_{1,r}}$, respectively, where $\phi_{1,r} \coloneqq \phi_{1,4m} \pmod{\varpi^r}$ and $\psi_{1,r} \coloneqq \psi_{1,4m} \pmod{\varpi^r}$.
    
    As before, Lemma \ref{VanishingofH1FN*} tells us that $\restr{\overline{\psi_{1,2m}}}{\Gamma_{F_{4m}^*}}$ is non-zero, so we can find $\gamma_1 \in \Gamma_{F_{4m}^*}$ such that $\overline{\psi_{1,2m}}(\gamma_1) \neq 0$. Also, using Lemma \ref{H^1ZeroMap} we can find $\gamma_2 \in \Gamma_{F_{4m}^*}$ such that $\phi_{1,2m}(\gamma_2) \neq 0$. Moreover, since $\Gamma_{F_{4m}^*}$ acts trivially on $\rho_{1,2m}(\mathfrak g^{\der})$ we have that $\phi_{1,2m}(\gamma_2^a) = a \cdot \phi_{1,2m}(\gamma_2)$. By choosing a suitable value for $a$ we can assume that $0 \neq \phi_{1,2m}(\gamma_2) \in \varpi^{2m-1}\mathfrak g_1^{\der}$. By Lemma \ref{rhoNProperties}, $\rho_{1,2m}(\mathfrak g_1^{\der})$ and $\overline{\rho_1}(\mathfrak g_1^{\der})^*$ do not share a common $\mathbb F_p[\Gamma_F]$-subquotient. Hence, $F_{4m}^*(\phi_{1,2m})$ and $F_{4m}^*(\overline{\psi_{1,2m}})$ are non-trivial linearly disjoint Galois extensions of $F_{4m}^*$. Therefore, we can find $\gamma \in \Gamma_{F_{4m}^*}$ which will satisfy the above properties simultaneously, i.e. $\overline{\psi_{1,2m}}(\gamma) \neq 0$ and $0 \neq \phi_{1,2m}(\gamma) \in \varpi^{2m-1}\mathfrak g_1^{\der}$. We consider its action on $F_{2m}^*(\phi_{1,2m},\overline{\psi_{1,2m}})$, which will yield an element in $\Gal(F_{2m}^*(\phi_{1,2m},\overline{\psi_{1,2m}})/L)$, where $L \coloneqq F_{2m}^*(\phi_{1,2m},\overline{\psi_{1,2m}}) \cap F_{4m}^*$. We can use $\gamma$ to define a Chebotarev condition in that extension, which will produce trivial primes $\nu$ such that $\phi_{1,2m}(\sigma_\nu)$ is a non-zero element of $\varpi^{2m-1}\mathfrak g^{\der}_1$ and $\overline{\psi_{1,2m}}(\sigma_\nu) \neq 0$.

    Now, as in Proposition \ref{SelectionofPrimes}, relying on Lemma \ref{TorusandRootLemma} for $p \gg_{G_1} 0$ we can find a split maximal torus $T_1$ of $G_1$ and $\alpha_1 \in \Phi(G^0_1,T_1)$ such that $\phi_{1,2m}(\gamma) \notin \ker(\restr{\alpha_1}{\mathfrak t_1}) \oplus\bigoplus (\mathfrak g_1)_\beta$ and $\overline{\psi_{1,2m}}(\gamma) \notin ((\mathfrak g_1)_{\alpha_1})^\perp$. Then we obtain a subset of $Q_{n,4m}$ which will satisfy the properties of Definition \ref{PrimesQ_n} with respect to this specific choice of $T_1$ and $\alpha_1$. This subset will be given by a Chebotarev condition in $F_n^*/F_{4m}^*$. To glue these two Chebotarev conditions it is enough to show that they agree on the intersection $F_n^* \cap F_{2m}^*(\phi_{1,2m},\overline{\psi_{1,2m}})$. As $F_{2m}^*(\phi_{1,2m},\overline{\psi_{1,2m}})$ is the fixed field of cocycles the intersection will be an abelian extension of $F_{2m}^*$, with a Galois group being an abelian quotient of $\Gal(F_n^*/F_{2m}^*)$. For $p \gg_{G_1} 0$ by Lemma \ref{AbelianizationLemma} the abelianization of this group is $\Gal(F_{4m}^*/F_{2m}^*)$, which means that the intersection is contained in $F_{4m}^*$. Thus, $F_n^* \cap F_{2m}^*(\phi_{1,2m},\overline{\phi_{1,2m}}) = F_{4m}^* \cap F_{2m}^*(\phi_{1,2m},\overline{\phi_{1,2m}}) = L$. As both Chebotarev conditions are trivial on $L$ they are compatible. Therefore, we can find a positive density set of primes $\nu \in Q_{n,4m}$ such that $\overline{\psi_{1,2m}}(\sigma_\nu) \notin ((\mathfrak g_1)_{\alpha_1})^\perp$, $\phi_{1,2m}(\sigma_\nu) \notin \ker(\restr{\alpha_1}{\mathfrak t_1}) \oplus \bigoplus (\mathfrak g_1)_\beta$, and $\phi_{1,2m-1}(\sigma_\nu) = 0$. By \cite[Lemma $6.7$]{FKP21} this translates to $\restr{\overline{\psi_{1,2m}}}{\Gamma_{F_\nu}}\notin L_{1,1,\nu}^{\alpha_1,\perp}$, $\restr{\phi_{1,2m}}{\Gamma_{F_\nu}} \notin L_{1,2m,\nu}^{\alpha_1}$, and $\restr{\phi_{1,2m-1}}{\Gamma_{F_\nu}} \in L_{1,2m-1,\nu}^{\alpha_1}$.

    Choose such a prime $\nu$ and set $L_{1,r,\nu}' = L_{1,r,\nu}^{\alpha_1} + L_{1,r,\nu}^{\mathrm{unr}}$. The condition on $\overline{\psi_{1,2m}}$ tells us that $\restr{\overline{\psi_{1,2m}}}{\Gamma_{F_\nu}}$ will have a non-zero image in $L_{1,1,\nu}^{\mathrm{unr},\perp}/L_{1,1,\nu}^{'\perp}$, which by \cite[Lemma $6.7$]{FKP21} is isomorphic to $\mathcal O/\varpi$. We conclude that $\restr{\psi_{1,2m}}{\Gamma_{F_\nu}}$ will generate $L_{1,2m,\nu}^{\mathrm{unr},\perp}/L_{1,2m,\nu}^{'\perp}$, which is isomorphic to $\mathcal O/\varpi^{2m}$, as it has a non-zero image under the restriction map $\mathcal O/\varpi^{2m} \to \mathcal O/\varpi$. Therefore, the restriction to $\Gamma_{F_\nu}$ induces an isomorphism
    
    $$H^1_{\mathcal L_{1,2m,S'}^\perp}(\Gamma_{F,S'},\rho_{1,2m}(\mathfrak g_1^{\der})^*)/H^1_{\mathcal L_{1,2m,S'}^\perp \cup L_{1,2m,\nu}^{'\perp}}(\Gamma_{F,S'},\rho_{1,2m}(\mathfrak g_1^{\der})^*) \simeq L_{1,2m,\nu}^{\mathrm{unr},\perp}/L_{1,2m,\nu}^{'\perp}.$$

    \vspace{2 mm}

    We apply the Greenberg-Wiles formula once with local conditions $\mathcal L_{1,2m,S'} \cup L_{1,2m,\nu}'$ and once with $\mathcal L_{1,2m,S'}$. Dividing the two equations we get

    $$\frac{\left|H^1_{\mathcal L_{1,2m,S'} \cup L_{1,2m,\nu}'}(\Gamma_{F,S' \cup \nu},\rho_{1,2m}(\mathfrak g_1^{\der}))\right|}{\left|H^1_{\mathcal L_{1,2m,S'}}(\Gamma_{F,S},\rho_{1,2m}(\mathfrak g_1^{\der}))\right|} = \frac{\left| H^1_{\mathcal L_{1,2m,S'}^\perp \cup L_{1,2m,\nu}^{'\perp}}(\Gamma_{F,S'},\rho_{1,2m}(\mathfrak g_1^{\der})^*)\right|}{\left| H^1_{\mathcal L_{1,2m,S'}^\perp}(\Gamma_{F,S'},\rho_{1,2m}(\mathfrak g_1^{\der})^*) \right|} \cdot \frac{\left|L_{1,2m,\nu}'\right|}{\left|L_{1,2m,\nu}^{\mathrm{unr}}\right|}.$$

    \vspace{2 mm}

    By the isomorphism the right-hand side becomes $1$. This implies the inclusion

    $$H^1_{\mathcal L_{1,2m,S' \cup \nu}}(\Gamma_{F,S'\cup\nu},\rho_{1,2m}(\mathfrak g_1^{\der})) \subseteq H^1_{\mathcal L_{1,2m,S'} \cup L_{1,2m,\nu}'}(\Gamma_{F,S' \cup \nu},\rho_{1,2m}(\mathfrak g_1^{\der})) = H^1_{\mathcal L_{1,2m,S'}}(\Gamma_{F,S},\rho_{1,2m}(\mathfrak g_1^{\der})).$$

    \vspace{2 mm}

    In summary, the condition $\restr{\overline{\psi_{1,2m}}}{\Gamma_{F_\nu}} \notin L_{1,1,\nu}^{\alpha_1,\perp}$ implies that allowing extra ramification at $\nu$ will not increase the size of the $2m$-th Selmer group attached to $\rho_{1,2m}$. Since $\overline{\psi_{1,2m}} = \overline{\psi_{1,2m-1}}$ we get an analogous inclusion on the level of $(2m-1)$-th relative Selmer group. Reducing modulo $\varpi$ this inclusion yields $\overline{H^1_{\mathcal L_{1,2m-1,S' \cup \nu}}(\Gamma_{F,S'\cup\nu},\rho_{1,2m-1}(\mathfrak g_1^{\der}))} \subseteq \langle \overline{\phi_{1,2m-1}} \rangle$. This is an equality, since $H^1_{\mathcal L_{1,2m-1,S' \cup \nu}}(\Gamma_{F,S'\cup\nu},\rho_{1,2m-1}(\mathfrak g_1^{\der}))$ contains $\phi_{1,2m-1}$ by our choice of $\nu$. 
    
    On the other side, we claim that $\overline{H^1_{\mathcal L_{1,4m,S' \cup \nu}}(\Gamma_{F,S'\cup\nu},\rho_{1,4m}(\mathfrak g_1^{\der}))} = 0$. Suppose that this isn't the case. Then, from the analogous inclusion of $4m$-th relative Selmer groups we get that this space has dimension $1$. In this case $\overline{H^1_{\mathcal L_{1,2m,S' \cup \nu}}(\Gamma_{F,S'\cup\nu},\rho_{1,2m}(\mathfrak g_1^{\der}))}$ must be non-zero, so by the same argument it has dimension $1$. Then, by \cite[Lemma $4.14$]{Nik26} the $(4m,2m)$-relative Selmer group $H^1_{(4m,2m)}(\Gamma_{F,S' \cup \nu},\rho_{1,4m}(\mathfrak g_1^{\der}))$ is a free $\mathcal O/\varpi^{2m}$-module of rank $1$. By our initial assumption $\overline{H^1_{\mathcal L_{1,4m,S'}}(\Gamma_{F,S'},\rho_{1,4m}(\mathfrak g_1^{\der}))} = \overline{H^1_{\mathcal L_{1,2m,S'}}(\Gamma_{F,S'},\rho_{1,2m}(\mathfrak g_1^{\der}))} $ with dimension $1$, so again by \cite[Lemma $4.14$]{Nik26} we have that $H^1_{(4m,2m)}(\Gamma_{F,S'},\rho_{1,4m}(\mathfrak g_1^{\der}))$ is a free $\mathcal O/\varpi^{2m}$-module of rank $1$. The inclusions above yield $H^1_{(4m,2m)}(\Gamma_{F,S' \cup \nu},\rho_{1,4m}(\mathfrak g_1^{\der})) \subseteq H^1_{(4m,2m)}(\Gamma_{F,S'},\rho_{1,4m}(\mathfrak g_1^{\der}))$, which is an equality as both are free $\mathcal O/\varpi^{2m}$-modules of rank $1$. But, this is a contradiction as $\phi_{1,2m}$ is an element of $H^1_{(4m,2m)}(\Gamma_{F,S'},\rho_{1,4m}(\mathfrak g_1^{\der}))$, but not of $H^1_{(4m,2m)}(\Gamma_{F,S' \cup \nu},\rho_{1,4m}(\mathfrak g_1^{\der}))$ based on our choice of prime $\nu$.

    As previously, we can't fully control the size of the relative Selmer group attached to $\rho_{2,r}$ after allowing extra ramification at $\nu$. However, by Lemma \ref{DimensionofRelSelmer}, $\overline{H^1_{\mathcal L_{2,2m-1,S' \cup \nu}}(\Gamma_{F,S' \cup \nu},\rho_{2,2m-1}(\mathfrak g_2^{\der}))}$ has dimension either $0$ or $1$. If the dimension is $0$ we are done by taking $Q = \{\nu\}$. Thus, we assume this relative Selmer group has dimension $1$ and is being generated by $\overline{\phi_{2,2m-1}}$. This assumption implies that the $m$-th relative Selmer group $\overline{H^1_{\mathcal L_{2,m,S' \cup \nu}}(\Gamma_{F,S' \cup \nu},\rho_{2,m}(\mathfrak g_2^{\der}))}$ is non-trivial and in particular one-dimensional. As before, this implies that $H_{(2m-1,m)}^1(\Gamma_{F,S' \cup \nu},\rho_{2,2m-1}(\mathfrak g_2^{\der}))$ is a free $\mathcal O/\varpi^m$-module of rank $1$, which is generated by $\phi_{2,m}$. On the other hand, by our choice of $\nu$ we know that $\overline{H^1_{\mathcal L_{1,2m-1,S' \cup \nu}}(\Gamma_{F,S'\cup\nu},\rho_{1,2m-1}(\mathfrak g_1^{\der}))}$ is one-dimensional and generated by $\overline{\phi_{1,2m-1}}$. Additionally, $\overline{H^1_{\mathcal L_{1,m,S' \cup \nu}}(\Gamma_{F,S'\cup\nu},\rho_{1,m}(\mathfrak g_1^{\der}))}$ has dimension $1$, so again $H^1_{(2m-1,m)}(\Gamma_{F,S' \cup \nu},\rho_{1,2m-1}(\mathfrak g_1^{\der}))$ is a free $\mathcal O/\varpi^m$-module of rank $1$ with generator $\phi_{1,m}$.

    By Proposition \ref{SelectionofPrimes} we can choose a prime $w \in Q_{n,m}$ such that $\restr{\phi_{i,m}}{\Gamma_{F_w}} \notin L_{i,m,w}^{\alpha'_i}$ and $\restr{\overline{\psi_{i,m}}}{\Gamma_{F_w}} \notin L_{i,1,w}^{\alpha'_i,\perp}$ for both $i=1,2$, where $\overline{\psi_{i,m}}$ are the generators of the $m$-th relative dual Selmer groups. Then, after allowing ramification at $w$ the $(2m-1)$-th relative Selmer group attached to both $\rho_{1,2m-1}$ and $\rho_{2,2m-1}$ vanish. The proof of this claim follows by arguing as in our selection of the prime $\nu$. The conditions on the dual cocycles $\overline{\psi_{i,m}}$ will tell us that after allowing extra ramification at $w$ the $(2m-1)$-th relative Selmer group doesn't increase in size. Therefore, $\overline{H^1_{\mathcal L_{i,2m-1,S'\cup\{\nu,w\}}}(\Gamma_{F,S' \cup \{\nu,w\}},\rho_{i,2m-1}(\mathfrak g_i^{\der}))}$ has dimension $0$ or $1$. If it's $0$ we are done. If it has dimension $1$, then the same is true for the $m$-th relative Selmer group and therefore $H^1_{(2m-1,m)}(\Gamma_{F,S' \cup \{\nu,w\}},\rho_{i,2m-1}(\mathfrak g_i^{\der}))$ is a free $\mathcal O/\varpi^m$-module of rank $1$. In particular, it will be equal to $H^1_{(2m-1,m)}(\Gamma_{F,S' \cup \nu},\rho_{i,2m-1}(\mathfrak g_i^{\der}))$, which is a contradiction as $\phi_{i,m}$ is an element of the latter, but not the former. Hence, we are done by taking $Q = \{\nu,w\}$.
    
\end{proof}

Summarizing, we have annihilated relative Selmer group, and consequently relative dual Selmer groups attached to $\rho_{1,4m}$ and $\rho_{2,4m}$ by allowing ramification at the same set of primes $Q$. We can then prove the main result of this section

\begin{thm}
\label{padicLift}

    Let $p \gg_{G_1,G_2} 0$. There exist a finite set $Q$ of trivial primes, a finite extension $\mathcal O'$ of $\mathcal O$, and $p$-adic geometric lifts $\rho_i:\Gamma_{F,S' \cup Q} \to G_i(\mathcal O)$ with multiplier $\mu_i$ of $\rho_{i,n}$, hence of $\overline{\rho_i}$ such that for each $\nu \in Q$ we have that $\restr{\rho_i}{\Gamma_{F_\nu}} \in \Lift_{\restr{\overline{\rho_i}}{\Gamma_{F_\nu}}}^{\mu_i,\alpha_{i,\nu}}(\mathcal O)$ for $i=1,2$.

\end{thm}
\begin{proof}

    Using the vanishing of the $\overline{H^1_{\mathcal L_{i,4m,S' \cup Q}^\perp}(\Gamma_{F,S' \cup Q},\rho_{i,4m}(\mathfrak g_i^{\der})^*)}$ we can produce the desired lifts $\rho_i$ by running the argument in the proof of \cite[Theorem $6.11$]{FKP21} for each of the representations $\overline{\rho_i}$ separately. 
    
\end{proof}

\section{\texorpdfstring{$\GL_2$}{GL\textunderscore 2}-representations with the same semi-simplification}

We now focus on the special case of two residual reducible $\GL_2$-valued representations which have the same semi-simplification. Let $\overline{\rho_1}:\Gamma_{F,S} \to \GL_2(k)$ and $\overline{\rho_2}:\Gamma_{F,S} \to \GL_2(k)$ be two such representations. As explained in \S 1 we can assume they are of the form

$$\overline{\rho_1} = \begin{pmatrix} \overline{\chi} & * \\ 0 & 1\end{pmatrix} \quad \quad \mathrm{and} \quad \quad \overline{\rho_2} = \begin{pmatrix} 1 & *' \\ 0 & \overline{\chi}\end{pmatrix}$$

\vspace{2 mm}

\noindent for some residual character $\overline{\chi}$, where the set $S$ contains all the primes dividing $p$, as well as all the primes at which either $\overline{\rho_i}$ is ramified. Moreover suppose that

\begin{itemize}
    \item $\overline{\chi}$ is an odd character, i.e. $\overline{\chi}(c_\nu) = -1$ for all complex conjugations $c_\nu$ and all infinite places $\nu$.
    \item $\overline{\chi} \neq \overline{\kappa}^{\pm 1}$.
    \item The extension classes $*$ and $*'$ are non-trivial. 
\end{itemize}

\begin{rem}
\label{SameShape}

    Our lifting method relies on Assumptions \ref{GeneralAssumptions}, which as we will see in \S $5.1$ and \S $5.2$ are satisfied whenever $\overline{\rho_1}$ and $\overline{\rho_2}$ have the same semi-simplification. So, even though our primary goal is to deal with the case when $\overline{\rho_1}$ and $\overline{\rho_2}$ are as mentioned above, the same method allows us to answer a version of Question \ref{WeakerQuestion} when both residual representations are of the shape $\begin{pmatrix}
        \overline{\chi} & * \\ 0 & 1
    \end{pmatrix}$. In fact, in this case finding local lifts of both representations which have the same behavior becomes an easier task, which we can use to prove a slightly stronger result. See Remark \ref{CrystallineLifts} for a more detailed discussion about this.
     
\end{rem}

In the case of $\GL_2$-representations the multiplier type $\overline{\mu_i}$ is exactly the determinant of $\overline{\rho_i}$, which in both cases is equal to $\overline{\chi}$. Therefore, fixing a lift $\mu = \mu_1 = \mu_2$ amounts to fixing the determinant of the $p$-adic lifts. We can take $\mu = \kappa^{r-1}\chi_0$ for some integer $r \ge 2$ and a finite-order character $\chi_0$. This is indeed possible by taking $\mu = \kappa^{r-1}[\overline{\chi}\overline{\kappa}^{1-r}]$, where the notation $[\overline{\psi}]$ stands for the Teichm\"uller lift of a residual character $\overline{\psi}:\Gamma_F \to k^\times$. The choice of a common lift $\mu$ of $\det \overline{\rho_1}$ and $\det \overline{\rho_2}$ will ensure that the newforms attached to the $p$-adic lifts $\rho_1$ and $\rho_2$ will have the same Neben character. 

We remark that $* = c_1$ and $*' = \overline{\chi}c_2$ for some cocycles $c_1 \in Z^1(\Gamma_{F,S},k(\overline{\chi}))$ and $c_2 \in Z^1(\Gamma_{F,S},k(\overline{\chi}^{-1}))$. Our assumption that the extension classes $*$ and $*'$ are non-trivial is equivalent to the $1$-cocycles $c_1$ and $c_2$ not being $1$-coboundaries in the respective cohomology groups. Our goal is to apply the results of \S $3$ and \S $4$ in this case. In order to do this we need to verify that $\overline{\rho_1}$ and $\overline{\rho_2}$ satisfy Assumptions \ref{GeneralAssumptions}. We will mostly focus on the construction of $p$-adic local lifts at the primes in $S$ to which we devote the whole of \S $5.2$. For our purposes we can't just take any $p$-adic local lifts at the primes in $S$ and we will have to construct them more carefully. This will take some substantial work and we spend most of \S $5$ on their construction, as well as explaining how they relate to the weight and the level of the newforms attached to the global $p$-adic lifts.

\subsection{Verification of assumptions}

In this subsection we will go one by one through each bullet point of Assumptions \ref{GeneralAssumptions}, except the fourth one, and verify that $\overline{\rho_1}$ and $\overline{\rho_2}$ satisfy them. Most of these computations are based on \cite[Lemma $7.1$]{FKP22}.

\begin{lem}
    \label{Assumption1}

    $H^1(\Gal(K/F),\overline{\rho_i}(\mathfrak g^{\der})^*) = 0$ for $i=1,2$, where $K = F(\overline{\rho_1},\overline{\rho_2},\mu_p)$.

\end{lem}
\begin{proof}

    Let $P$ be a Sylow $p$-subgroup of $\Gal(K/F)$. We claim that the fixed field $K^P$ is equal to $K_0 \coloneqq F(\overline{\chi},\mu_p)$. Let $\sigma \in P$. Since $F(\mu_p)/F$ is a normal subextension of $K/F$, any element of $\Gal(K/F)$ restricts to an element of $\Gal(F(\mu_p)/F)$. Since $[F(\mu_p):F]$ divides $p-1$ we have that $\restr{\sigma}{F(\mu_p)}$ has order dividing both $p$ and $p-1$, hence $\restr{\sigma}{F(\mu_p)} = 1$. From this $F(\mu_p) \subseteq K^P$. On the other hand, $\Gal(F(\overline{\chi})/F) \simeq \Gamma_F/\ker \overline{\chi} \simeq \im \overline{\chi}$. As $\overline{\chi}$ is a $k$-valued multiplicative character its image has size dividing $q-1$, where $q$ is the size of $k$ and hence a $p$-power. Thus, $F(\overline{\chi})/F$ is a normal extension of order coprime to $p$. So, as before $F(\overline{\chi}) \subseteq K^P$. Hence, $K_0 \subseteq K^P$.

    For the reverse inclusion, let  $H \subseteq \Gal(K/F)$ be the fixed subgroup of $K_0$. We've shown that $K^H = K_0 \subseteq K^P$ and hence $P \subseteq H$. It now suffices to show that $H$ is a $p$-group. Indeed, as $P$ is a Sylow $p$-subgroup by the maximality we must have $P = H$ and so $K^P = K_0$. For $\sigma \in H$ we have $\overline{\chi}(\sigma) = 1$ and therefore

    $$c_1(\sigma^2) = \overline{\chi}(\sigma) \cdot c_1(\sigma) + c_1(\sigma) = c_1(\sigma) + c_1(\sigma) = 2c_1(\sigma).$$

    \vspace{2 mm}

    Inductively, $c_1(\sigma^p) = pc_1(\sigma) = 0$, as $\mathrm{char}$ $k = p$. Analogously, $c_2(\sigma^p) = 0$. Hence, $\sigma^p \in \ker c_1, \ker c_2,\ker \overline{\chi}$ and from this $\sigma^p \in \ker \overline{\rho_1},\ker \overline{\rho_2}$. Moreover, $\sigma$ already acts trivially on $\mu_p$, so we conclude that $\sigma^p = 1$ in $\Gal(K/F)$. In particular, $\sigma$ has order dividing $p$, which implies that $H$ is a $p$-group, which is exactly what we wanted. We remark that this also shows that $P$ is the unique Sylow $p$-subgroup of $\Gal(K/F)$ and hence is normal.

    We now consider the inflation-restriction sequence:

    $$0 \longrightarrow H^1(\Gal(K_0/F),\overline{\rho_i}(\mathfrak g^{\der})^*) \longrightarrow H^1(\Gal(K/F),\overline{\rho_i}(\mathfrak g^{\der})^*) \longrightarrow H^1(P,\overline{\rho_i}(\mathfrak g^{\der})^*)^{\Gal(K_0/F)}.$$

    \vspace{2 mm}

    As $\Gal(K_0/F) \simeq \Gal(K/F)/P$ and $P$ is a Sylow $p$-subgroup of $\Gal(K/F)$, the order of this Galois group is coprime to $p$. Therefore, the first term in the sequence vanishes. From this it is enough to show that $H^1(P,\overline{\rho_i}(\mathfrak g^{\der})^*)^{\Gal(K_0/F)} = 0$ for $i=1,2$. We now consider the following subspaces of $\mathfrak{sl}_2$:

    $$F_2 = \left \langle \begin{pmatrix} 0 & 1 \\ 0 & 0 \end{pmatrix} \right\rangle, \quad \quad F_1 = \left \langle \begin{pmatrix} 0 & 1 \\ 0 & 0 \end{pmatrix}, \begin{pmatrix} 1 & 0 \\ 0 & -1 \end{pmatrix} \right\rangle, \quad \quad F_0 = \mathfrak{g}^{\der}.$$

    \vspace{2 mm}

    An easy computation shows that these are submodules of $\overline{\rho_1}(\mathfrak g^{\der})^*$ after identifying $\overline{\rho_1}(\mathfrak g^{\der})^*$ with $\overline{\rho_1}(\mathfrak g^{\der})(1)$ using the trace pairing. Moreover, $F_2 \simeq k(\overline{\chi}\overline{\kappa}), F_1/F_2 \simeq k(\overline{\kappa})$, and $F_0/F_1 \simeq k(\overline{\chi}^{-1}\overline{\kappa})$ as $k[\Gamma_F]$-submodules of $\overline{\rho_1}(\mathfrak g^{\der})^*$. From the short exact sequence $0 \to k(\overline{\chi}\overline{\kappa}) \to F_1 \to k(\overline{\kappa}) \to 0$ of $k[\Gamma_F]$-modules we get an long exact sequence on cohomology

    $$ \cdots \longrightarrow H^1(P,k(\overline{\chi}\overline{\kappa})) \longrightarrow H^1(P,F_1) \longrightarrow H^1(P,k(\overline{\kappa})) \longrightarrow \cdots .$$

    \vspace{2 mm}

    As $\Gamma_{K_0}$ acts trivially via both $\overline{\chi}$ and $\overline{\kappa}$ the $\Gamma_F$-action on these modules factors through $\Gal(K_0/F)$ and hence they are all $k[\Gal(K_0/F)]$-modules. As $\Gal(K_0/F)$ has order prime to $p$ taking $\Gal(K_0/F)$-invariants is an exact functor. Therefore, to show that $H^1(P,F_1)^{\Gal(K_0/F)} = 0$ it is enough to show that $H^1(P,k(\overline{\chi}\overline{\kappa}))^{\Gal(K_0/F)} = H^1(P,k(\overline{\kappa}))^{\Gal(K_0/F)} = 0$. Likewise, from the short exact sequence $0 \to F_1 \to \overline{\rho_i}(\mathfrak g^{\der})^* \to k(\overline{\chi}^{-1}\overline{\kappa}) \to 0$ of $k[\Gamma_F]$-modules we deduce that to get $H^1(P,\overline{\rho_1}(\mathfrak g^{\der})^*)^{\Gal(K_0/F)} = 0$ it suffices to show that $H^1(P,F_1)^{\Gal(K_0/F)} = H^1(P,k(\overline{\chi}^{-1}\overline{\kappa}))^{\Gal(K_0/F)} = 0$. Combining these two results it suffices to show that $H^1(P,k(\psi))^{\Gal(K_0/F)} = 0$ for $\psi = \overline{\chi}\overline{\kappa},\overline{\kappa},\overline{\chi}^{-1}\overline{\kappa}$. Moreover, as $P$ is the fixed group of $K_0$ it acts trivially via both $\overline{\chi}$ and $\overline{\kappa}$ and therefore $H^1(P,k(\psi))^{\Gal(K_0/F)} = \Hom_{\Gal(K_0/F)}(P,k(\psi))$. An analogous computation shows that $F_2 \simeq k(\overline{\chi}^{-1}\overline{\kappa})$, \linebreak $F_1/F_2 \simeq k(\overline{\kappa})$, and $F_0/F_1 \simeq k(\overline{\chi}\overline{\kappa})$ as $k[\Gamma_F]$-submodules of $\overline{\rho_2}(\mathfrak g^{\der})^*$. Similarly $H^1(P,\overline{\rho_2}(\mathfrak g^{\der})^*)^{\Gal(K_0/F)} = 0$ also reduces to showing that $\Hom_{\Gal(K_0/F)}(P,k(\psi)) = 0$ for $\psi = \overline{\chi}\overline{\kappa},\overline{\kappa},\overline{\chi}^{-1}\overline{\kappa}$. 

    In order to show that these homomorphism groups are trivial we will first show that $P$ as an $\mathbb F_p[\Gamma_F]$-module, where the $\Gamma_F$-action is given by conjugation, is a direct sum of copies of $k_0(\overline{\chi})$ and $k_0(\overline{\chi}^{-1})$, where $k_0 \subseteq k$ is the finite extension of $\mathbb F_p$ generated by the values of $\overline{\chi}$. Analogous computations to the ones we did at the beginning of the proof tell us that each of $\Gal(K_1/F)$ and $\Gal(K_2/F)$ has a unique Sylow $p$-subgroup, labeled by $P_1$ and $P_2$, respectively. Furthermore, the fixed field of $P_i$ in $\Gal(K_i/F)$ is also $K_0$. Since $\Gamma_{K_0}$ lies in $\ker \overline{\chi}$ we have that $\restr{c_1}{\Gamma_{K_0}}:\Gamma_{K_0} \to k$ is a group homomorphism. By the definition of $K_1$ we have $\ker(\restr{c_1}{\Gamma_{K_0}}) = \Gamma_{K_1}$. Therefore, we obtain an injection

    $$P_1 = \Gal(K_1/K_0) \simeq \Gamma_{K_0}/\Gamma_{K_1} \xhookrightarrow{c_1} k.$$

    \vspace{2 mm}

    We claim that this is an $\mathbb F_p[\Gamma_F]$-module homomorphism, where the $\Gamma_F$ acts by conjugation on $P_1$ and by the character $\overline{\chi}$ on $k$. Since $\overline{\chi}$ is trivial on $P_1$ as above we have that $c_1(\sigma^2) = 2c_1(\sigma)$ and inductively $c_1(\sigma^a) = ac_1(\sigma)$ for all $a \in \mathbb F_p$ and $\sigma \in P_1$. Therefore, $P_1 \hookrightarrow k$ is an $\mathbb F_p$-vector space homomorphism. Now, for $\sigma \in P_1$ and $\tau \in \Gamma_F$ we have

    \begin{align*}
        c_1(\tau \cdot \sigma) &= c_1(\tau\sigma\tau^{-1}) = c_1(\tau) + \overline{\chi}(\tau)c_1(\sigma\tau^{-1}) = c_1(\tau) + \overline{\chi}(\tau)(c_1(\sigma) + \overline{\chi}(\sigma)c_1(\tau^{-1})) \\
        & =c_1(\tau) + \chi(\tau)c_1(\sigma) + \chi(\tau)c_1(\tau^{-1}) = \overline{\chi}(\tau)c_1(\sigma) = \tau \cdot c_1(\sigma),
    \end{align*}

    \vspace{2 mm}

    \noindent where we used that $\overline{\chi}(\sigma) = 1$ and $0 = c_1(1) = c_1(\tau\tau^{-1}) = c_1(\tau) + \overline{\chi}(\tau)c_1(\tau^{-1})$. Therefore, $c_1$ is $\Gamma_F$-equivariant and hence an $\mathbb F_p[\Gamma_F]$-module homomorphism. Similarly, $\restr{c_2}{\Gamma_{K_0}}: P_2 \hookrightarrow k$ is an $\mathbb F_p[\Gamma_F]$-module homomorphism. We then get an inclusion

    $$P \hookrightarrow P_1 \oplus P_2 \stackrel{c_1 \times c_2}{\longhookrightarrow} k(\overline{\chi}) \oplus k(\overline{\chi}^{-1})$$

    \vspace{2 mm}

    \noindent of $\mathbb F_p[\Gamma_F]$-modules. Since $\overline{\chi}$ is valued in $k_0$ the $\mathbb F_p[\Gamma_F]$-submodule of $k(\overline{\chi})$ generated by a non-zero element of it is isomorphic to $k_0(\overline{\chi})$. On the other hand, the $\Gamma_F$-action via $\overline{\chi}$ factors through $\Gal(F(\overline{\chi})/F)$, which has order coprime to $p$. Thus, by Maschke's Theorem $k(\overline{\chi})$ is a semi-simple $\mathbb F_p[\Gamma_F]$-module and decomposes into a direct sum of copies of $k_0(\overline{\chi})$. Similarly, $k(\overline{\chi}^{-1})$ is a semi-simple $\mathbb F_p[\Gamma_F]$-module and decomposes into a direct sum of copies of $k_0(\overline{\chi}^{-1})$. Then, $P$ must also be a semi-simple $\mathbb F_p[\Gamma_F]$-module by virtue of being an $\mathbb F_p[\Gamma_F]$-submodule of the semi-simple $\mathbb F_p[\Gamma_F]$-module $k(\overline{\chi}) \oplus k(\overline{\chi}^{-1})$. Moreover, it must decompose as a direct sum of copies of $k_0(\overline{\chi})$ and $k_0(\overline{\chi}^{-1})$. 

    Finally, using this decomposition we get that $\Hom_{\mathbb F_p[\Gal(K_0/F)]}(P,k(\psi))$ decomposes as a direct sum of copies of $\Hom_{\mathbb F_p[\Gal(K_0/F)]}(k_0(\overline{\chi}^{\pm 1}),k(\psi))$, and we will show all these homomorphism groups are trivial. We consider two cases:

    \begin{enumerate}
        \item Suppose that $\psi = \overline{\chi}\overline{\kappa}, \overline{\chi}^{-1}\overline{\kappa}$. In this case, the values of $\psi$ will still generate $k_0$ and therefore each non-zero element of $k(\psi)$ will generate an $\mathbb F_p[\Gamma_F]$-module isomorphic to $k_0(\psi)$. Since the action of $\psi$ factors through $\Gal(K_0/F)$, which has order coprime to $p$, as above $k(\psi)$ is a semi-simple $\mathbb F_p[\Gamma_F]$-module which decomposes as a direct sum of copies of $k_0(\psi)$. As $k_0(\overline{\chi}^{\pm 1})$ is an irreducible $\mathbb F_p[\Gamma_F]$-module by Schur's Lemma any non-zero element of $\Hom_{\mathbb F_p[\Gal(K_0/F)]}(k_0(\overline{\chi}^{\pm 1}),k(\psi))$ will give rise to an $\mathbb F_p[\Gamma_F]$-module isomorphism $k_0(\overline{\chi}^{\pm 1}) \simeq k_0(\psi)$.

        Tensoring with $k_0$ over $\mathbb F_p$ we get a $k_0[\Gamma_F]$-module isomorphism $k_0(\overline{\chi}^{\pm1}) \otimes_{\mathbb F_p} k_0 \simeq k_0(\psi) \otimes_{\mathbb F_p} k_0$. Now, $k_0(\overline{\chi}^{\pm 1}) \otimes_{\mathbb F_p} k_0 \simeq \oplus_{\theta \in \Gal(k_0/\mathbb F_p)} k_0(\overline{\chi}^{\pm 1})^\theta$. This follows from the standard algebra isomorphism given by $x \otimes a \to (\theta(a)x)_\theta$ with the $\Gamma_F$-action coming from the first coordinate and is still given by $\overline{\chi}^{\pm 1}$ on each summand, while multiplication by $k_0$ goes through the second coordinate, so on the $\theta$-summand it is given by precomposition with $\theta$. Moreover, we note that $k_0(\theta^{-1} \circ \overline{\chi}^{\pm 1}) \simeq k_0(\overline{\chi}^{\pm 1})^\theta$ as $k_0[\Gamma_F]$-modules via $\theta$. We remark that $\theta$ might act non-trivially on elements in $k_0$, but the fact that multiplication by $k_0$ in $k_0(\overline{\chi}^{\pm 1})^\theta$ is given by precomposition by $\theta$ will undo this action and hence this isomorphism will be $k_0$-linear. From the analogous decomposition of $k_0(\psi) \otimes_{\mathbb F_p} k_0$ we get that $k_0(\overline{\chi}^{\pm 1}) \simeq k_0(\theta \circ \psi)$, as $k_0[\Gamma_F]$-modules. Hence, $\overline{\chi}^{\pm 1} = \theta \circ \psi$ for some $\theta \in \Gal(k_0/\mathbb F_p)$. Evaluating at a complex conjugation $c_\nu$ we get

        $$- 1 = \overline{\chi}^{\pm 1}(c_\nu) = \theta(\psi(c_\nu)) = \theta(-\overline{\kappa}(c_\nu)) = - \overline{\kappa}(c_\nu),$$

        \vspace{2 mm}

        \noindent where we used that $\overline{\chi}(c_\nu) = -1$ by our assumptions, and $\overline{\kappa}$ is valued in $\mathbb F_p^\times$, hence is fixed by $\theta$. This would imply that $\overline{\kappa}(c_\nu) = 1$, i.e. complex conjugation fixes the $p$-th roots of unity, which is obviously false for $p\ge 3$, as $F$ is a totally real number field.

        \item Let $\psi = \overline{\kappa}$. Since $\overline{\kappa}$ is valued in $\mathbb F_p$, in a similar manner as above we get that $k(\psi)$ is a semi-simple $\mathbb F_p[\Gamma_F]$-module which decomposes as a direct sum of copies of $\mathbb F_p(\psi)$. Then, any non-zero element of $\Hom_{\mathbb F_p[\Gal(K_0/F)]}(k_0(\overline{\chi}^{\pm 1}),k(\psi))$ will yield a $\mathbb F_p[\Gamma_F]$-module isomorphism $k_0(\overline{\chi}^{\pm 1}) \simeq \mathbb F_p(\psi)$. This immediately forces $k_0 = \mathbb F_p$ and $\overline{\chi}^{\pm 1} = \psi = \overline{\kappa}$, which is ruled out by our assumptions on $\overline{\chi}$.

    \end{enumerate}
\end{proof}

\begin{lem}
\label{Assumption2}

    $\overline{\rho_i}(\mathfrak g^{\der})$ and $\overline{\rho_j}(\mathfrak g^{\der})^*$ have no common $\mathbb F_p[\Gamma_F]$-subquotient for $i,j=1,2$.
    
\end{lem}
\begin{proof}

    Suppose that $V$ is a common $\mathbb F_p[\Gamma_F]$-subquotient of $\overline{\rho_i}(\mathfrak g^{\der})$ and $\overline{\rho_j}(\mathfrak g^{\der})^*$. Tensoring with $k$ over $\mathbb F_p$ we get that $V \otimes_{\mathbb F_p} k$ is a common $k[\Gamma_F]$-subquotient of $\overline{\rho_i}(\mathfrak g^{\der}) \otimes_{\mathbb F_p} k$ and $\overline{\rho_j}(\mathfrak g^{\der})^* \otimes_{\mathbb F_p} k$. As in Lemma \ref{Assumption1} we get that $\overline{\rho_i}(\mathfrak g^{\der}) \otimes_{\mathbb F_p} k \simeq \oplus_{\theta \in \Gal(k/\mathbb F_p)} \overline{\rho_i}(\mathfrak g^{\der})^\theta$. By the Jordan-H\"older theorem the simple $k[\Gamma_F]$-subquotients of $\overline{\rho_i}(\mathfrak g^{\der}) \otimes_{\mathbb F_p} k$ are the simple $k[\Gamma_F]$-subquotients of $\overline{\rho_i}(\mathfrak g^{\der})^\theta$, as $\theta$ ranges in $\Gal(k/\mathbb F_p)$. Similarly, the simple $k[\Gamma_F]$-subquotients of $\overline{\rho_j}(\mathfrak g^{\der})^* \otimes_{\mathbb F_p} k$ are the simple $k[\Gamma_F]$-subquotients of $(\overline{\rho_j}(\mathfrak g^{\der})^*)^\beta$ as $\beta$ varies in $\Gal(k/\mathbb F_p)$. Thus, by considering a simple $k[\Gamma_F]$-subquotient of $V \otimes_{\mathbb F_p} k$ we get that $\overline{\rho_i}(\mathfrak g^{\der})^\theta$ and $(\overline{\rho_j}(\mathfrak g^{\der})^*)^\beta$ share a common simple $k[\Gamma_F]$-subquotient for some $\theta,\beta \in \Gal(k/\mathbb F_p)$. As seen in Lemma \ref{Assumption1} the simple $k[\Gamma_F]$-subquotients of $\overline{\rho_i}(\mathfrak g^{\der})^\theta$ are $k(\psi)^\theta \simeq k(\theta^{-1} \circ \psi)$, where $\psi = \overline{\chi},\overline{\chi}^{-1}$, or the trivial character. Similarly, the simple $k[\Gamma_F]$-subquotients of $(\overline{\rho_j}(\mathfrak g^{\der})^*)^\beta$ are $k(\psi')^\beta \simeq k(\beta^{-1} \circ \psi')$, where $\psi' = \overline{\chi}\overline{\kappa},\overline{\chi}^{-1}\overline{\kappa}$, or $\overline{\kappa}$. The common simple $k[\Gamma_F]$-subquotient implies that $\psi = \theta \circ \psi'$ for some $\theta \in \Gal(k/\mathbb F_p)$. We rule out all the options on a case-to-case basis. 

    \begin{itemize}
        \item If $\psi$ is equal to the trivial character then we must have that $\psi'$ is also the trivial character. This implies $\overline{\kappa}$ is equal to $\overline{\chi},\overline{\chi}^{-1}$, or the trivial character. The first two options are explicitly ruled out by our assumptions, while the third one is clearly impossible. Thus, $\psi$ can't be the trivial character. 

        \item Now let $\psi = \overline{\chi}$. If $\psi' = \overline{\kappa}$, then it will be an $\mathbb F_p^\times$-valued character, which will imply $\overline{\chi} = \overline{\kappa}$, which contradicts our assumptions on $\overline{\chi}$. If $\psi' = \overline{\chi}\overline{\kappa}$ or ${\overline{\chi}}^{-1}\overline{\kappa}$, then evaluating at a complex conjugation $c_\nu$

        $$-1 = \overline{\chi}(c_\nu) = \theta(\overline{\chi}^{\pm 1}(c_\nu)\overline{\kappa}(c_\nu)) = \theta(-\overline{\kappa}(c_\nu)) = -\overline{\kappa}(c_\nu).$$

        \vspace{2 mm}

        Thus, $\overline{\kappa}(c_\nu) = 1$, which is a contradiction, as $\overline{\kappa}$ doesn't fix the $p$-th roots of unity for $p \ge 3$, as $F$ is a totally real number field. 

        \item Finally, we consider $\psi = \overline{\chi}^{-1}$. As in the previous case if $\psi' = \overline{\kappa}$ we get that $\overline{\chi}^{-1} = \overline{\kappa}$, which contradicts our assumptions. And if $\psi' = \overline{\chi}\overline{\kappa}$ or $\overline{\chi}^{-1}\overline{\kappa}$ then we get a contradiction by evaluating at a complex conjugation $c_\nu$.
        
    \end{itemize}

    Having discarded all the possible options we conclude that no such simple $k[\Gamma_F]$-subquotient exists. Hence, $\overline{\rho_i}(\mathfrak g^{\der})$ and $\overline{\rho_j}(\mathfrak g^{\der})^*$ have no common $\mathbb F_p[\Gamma_F]$-subquotient for $i,j = 1,2$.
    
\end{proof}

\begin{lem}
\label{Assumption3}

Neither $\overline{\rho_i}(\mathfrak g^{\der})$ nor $\overline{\rho_i}(\mathfrak g^{\der})^*$ contains the trivial representation for $i=1,2$.
    
\end{lem}
\begin{proof}

    We first show that $\overline{\rho_1}(\mathfrak g^{\der})$ doesn't contain the trivial representation. In our case $\mathfrak g^{\der} = \mathfrak{sl}_2$ and so its elements will be $(2 \times 2)$ matrices with trace $0$, and $\sigma$ acts on it by conjugation with $\overline{\rho_1}(\sigma)$. Suppose that $\Gamma_F$ acts trivially on  $\begin{pmatrix} x & y \\ z & -x \end{pmatrix}$. Then, for $\sigma \in \Gamma_F$ we have

    \begin{align*}
        \sigma \cdot \begin{pmatrix} x & y \\ z & -x \end{pmatrix} &= \begin{pmatrix}
            \overline{\chi}(\sigma) & c_1(\sigma) \\ 0 & 1 \end{pmatrix} \begin{pmatrix}
            x & y \\ z & - x \end{pmatrix} \begin{pmatrix}
            \overline{\chi}(\sigma) & c_1(\sigma) \\ 0 & 1 \end{pmatrix}^{-1} \\
            & = \begin{pmatrix} x + \overline{\chi}^{-1}(\sigma)c_1(\sigma)z & \overline{\chi}(\sigma)y - 2c_1(\sigma)x - \overline{\chi}(\sigma)c_1(\sigma)^2z \\ \overline{\chi}^{-1}(\sigma)z & -x - \overline{\chi}^{-1}(\sigma)c_1(\sigma)z\end{pmatrix}.
    \end{align*}

    \vspace{2 mm}

    By comparing the bottom left entries we get that $z = \overline{\chi}^{-1}(\sigma)z$ for all $\sigma \in \Gamma_F$. But $\overline{\chi} \neq 1$, which means that $z=0$. Then, the $\Gamma_F$-action simplifies to

    $$\sigma \cdot \begin{pmatrix} x & y \\ 0 & -x \end{pmatrix} = \begin{pmatrix} x & \overline{\chi}(\sigma)y - 2c_1(\sigma)x \\ 0 & -x \end{pmatrix}.$$

    \vspace{2 mm}

    We now take $\sigma$ to be an element of $\Gamma_{K_0}$, but not of $\Gamma_{K_1}$. Such an element exists by the computations in Lemma \ref{Assumption1}. Indeed, we have that $K_1 = K_0(c_1)$, and $K_1/K_0$ is a non-trivial extension as $c_1$ is not a $1$-coboundary. Therefore, $\overline{\chi}(\sigma) = 1$ and $c_1(\sigma) \neq 0$. Comparing the top right entries yields $2c_1(\sigma)x = 0$. As $p \neq 2$ and $c_1(\sigma) \neq 0$ we must have that $x = 0$. Finally, the $\Gamma_F$-action becomes

    $$\sigma \cdot \begin{pmatrix} 0 & y \\ 0 & 0 \end{pmatrix} = \begin{pmatrix} 0 & \overline{\chi}(\sigma)y \\ 0 & 0 \end{pmatrix},$$

    \vspace{2 mm}

    \noindent which can't be trivial since $\overline{\chi} \neq 1$. Hence, we conclude that $\overline{\rho_1}(\mathfrak g^{\der})$ does not contain the trivial representation. We prove the claim for $\overline{\rho_1}(\mathfrak g^{\der})^*$ in a similar manner. We first identify $\overline{\rho_1}(\mathfrak g^{\der})^*$ with $\overline{\rho_1}(\mathfrak g^{\der})(1)$ using the trace pairing. Then the analogous computations still hold, relying on the fact that $\Gamma_{K_0}$ has trivial image under $\overline{\chi}\overline{\kappa}$ and $\overline{\chi} \neq \overline{\kappa}$, by our initial assumption. 

    Finally, we can easily adapt the same computations to prove the claim for $i=2$. Indeed, $\overline{\rho_2} \sim \begin{pmatrix} \overline{\chi}^{-1} & c_2 \\ 0 & 1 \end{pmatrix}$ after twisting by $\overline{\chi}$. Then, the same computations hold with $\overline{\chi}^{-1}$ in place of $\overline{\chi}$, as $\overline{\chi} \neq 1,\overline{\kappa}^{-1}$, the cocycle $c_2$ is not a coboundary, and $\Gamma_{K_0}$ has trivial image under both $\overline{\chi}^{-1}$ and $\overline{\chi}^{-1}\overline{\kappa}$.
     
\end{proof}

\begin{lem}
\label{Assumption4}
 
 $\overline{\rho_i}$ is odd for $i=1,2$, i.e. for every infinite place $\nu$ of $F$, $h^0(\Gamma_{F_\nu},\overline{\rho}(\mathfrak g^{\der})) = 1$.

\end{lem}
\begin{proof}

    Let $c_\nu$ be a complex conjugation in $\Gamma_F$. As $c_\nu$ has order $2$ and the cocycles $c_i$ are valued in $k$, which has odd characteristic we get that $c_i(c_\nu) =0$ for both $i=1,2$. Using our assumption that $\overline{\chi}$ is odd we get that the image of $c_\nu$ under both $\overline{\rho_1}$ and $\overline{\rho_2}$ is conjugate to $\begin{pmatrix} 1 & 0 \\ 0 &-1 \end{pmatrix}$. Now let $\begin{pmatrix} x & y \\ z & -x \end{pmatrix}$ be fixed by the action of $c_\nu$. We then have:

    $$\begin{pmatrix} x & y \\ z & -x \end{pmatrix} = \begin{pmatrix} 1 & 0 \\ 0 & -1 \end{pmatrix}\begin{pmatrix} x & y \\ z & -x \end{pmatrix}\begin{pmatrix} 1 & 0 \\ 0 & -1 \end{pmatrix} = \begin{pmatrix} x & -y \\ -z & -x \end{pmatrix}.$$

    \vspace{2 mm}

    Comparing the off-diagonal entries we get $2y = 2z = 0$, which implies that $y=z =0$, as $p \neq 2$. From this we conclude that for both $i=1,2$ the $\Gamma_{F_\nu}$-invariants of $\overline{\rho_i}(\mathfrak g^{\der})$ are exactly the diagonal $2\times 2$ matrices with trace zero, which span a subspace of dimension $1$.
    
\end{proof}

\subsection{Local lifts at primes in \texorpdfstring{$S$}{S}}

Our goal is to produce $p$-adic lifts $\rho_1$ and $\rho_2$ of $\overline{\rho_1}$ and $\overline{\rho_2}$, respectively, whose associated Weil-Deligne representations at each prime will have the same Artin conductor. We recall that at primes $\nu \in S$ we modeled our global lift $\rho_i$ using the given local lifts $\rho_{i,\nu}$. Therefore, for $\nu \in S$ we want to find local lifts $\rho_{1,\nu}$ and $\rho_{2,\nu}$, such that $\WD(\rho_{1,\nu})$ and $\WD(\rho_{2,\nu})$ will have the same Artin conductor. We do this first for primes $\nu \nmid p$. For such primes we will prove something stronger. In particular, we will show that we can choose $\rho_{1,\nu}$ and $\rho_{2,\nu}$ so that they have the same strong inertial type. We first recall the definition of a strong inertial type.

\begin{defn}
\label{StrongInertialType}

Let $L$ be a finite extension of $\mathbb Q_\ell$. An $n$-dimensional \textit{strong inertial type} $\tau$ over $L$ is an equivalence class of pair $(r_\tau,N_\tau)$ such that

\begin{itemize}
    \item $r_\tau:I_L \to \GL_n(\overline{E})$ is a representation with open kernel.
    \item $N_\tau$ is a nilpotent $n \times n$ matrix with coefficients in $\overline{E}$.
    \item $(r_\tau,N_\tau)$ extends to a Weil-Deligne representation of $\Gamma_L$.
\end{itemize}

Two such pairs are equivalent if they are conjugate by an element in $\GL_n(\overline{E})$. We say that a representation $\Gamma_L \to \GL_n(\overline{E})$ has strong inertial type $\tau$ if the restriction to $I_L$ of the associated Weil-Deligne representation is equivalent to $\tau$. 
    
\end{defn}

In order to establish the existence of the desired local lifts for $\nu \nmid p$ we will use the computations in \cite[\S $5$]{Sho16}. Before doing this as in \cite[\S $3.1$]{Sho16} we define some two-dimensional strong inertial types which will appear often in the computations. All of these will be trivial on $\Gamma_{F_\nu^{\mathrm{tame},p}}$ and therefore they will be completely determined by the image of $r(\tau_\nu)$ and the nilpotent matrix $N$. Set $a_\nu = v_p(N(\nu)-1)$ and $b_\nu = v_p(N(\nu)+1)$. We then define

\begin{itemize}
    \item $\tau_{\zeta,s}$ by $r(\tau_\nu) = \begin{pmatrix} \zeta & 0 \\ 0 & \zeta\end{pmatrix}$ and $N = 0$, where $\zeta$ is a $p^{a_\nu}$-th root of unity.
    \item $\tau_{\zeta,ns}$ by $r(\tau_\nu) = \begin{pmatrix} \zeta & 0 \\ 0 & \zeta\end{pmatrix}$ and $N = \begin{pmatrix} 0 & 1 \\ 0 & 0\end{pmatrix}$, where $\zeta$ is a $p^{a_\nu}$-th root of unity.
    \item $\tau_{\zeta_1,\zeta_2}$ by $r(\tau_\nu) = \begin{pmatrix} \zeta_1 & 0 \\ 0 & \zeta_2\end{pmatrix}$ and $N = 0$, where $\zeta_1$ and $\zeta_2$ are distinct $p^{a_\nu}$-th roots of unity.
    \item $\tau_{\xi}$ by $r(\tau_\nu) = \begin{pmatrix} \xi & 0 \\ 0 & \xi^{-1}\end{pmatrix}$ and $N = 0$, where $\xi$ is a $p^{b_\nu}$-th root of unity.
\end{itemize}

\vspace{2 mm}

\begin{prop}
\label{LocalLiftsnmidp}

For each $\nu \nmid p \in S$ and $i=1,2$, there exist local lifts $\rho_{i,\nu}:\Gamma_{F_\nu} \to \GL_2(\mathcal O)$ of $\restr{\overline{\rho_i}}{\Gamma_{F_\nu}}$ such that $\rho_{1,\nu}$ and $\rho_{2,\nu}$ have the same strong inertial type. 
    
\end{prop}

\begin{proof}

For each $\restr{\overline{\rho_i}}{\Gamma_{F_\nu}}$ we will use the results of \cite[\S $5$]{Sho16} to give an explicit description of the local lifting rings $R_{\restr{\overline{\rho_i}}{\Gamma_{F_\nu}}}^{\sqr}$ and its strong inertial type $\tau$ quotients $R_{\restr{\overline{\rho_i}}{\Gamma_{F_\nu}}}^{\sqr,\tau}$. Choosing an $\mathcal O$-point in such a quotient will give us a $p$-adic lift of $\restr{\overline{\rho_i}}{\Gamma_{F_\nu}}$ with a strong inertial type $\tau$. It will be clear from the explicit description that such quotients will have an $\mathcal O$-point whenever they are non-zero.

We will first show that we can assume the cocycles $c_1$ and $c_2$ are trivial on $\Gamma_{F_\nu^{\mathrm{tame},p}}$. The restriction $\restr{c_1}{\Gamma_{F_\nu}}$ will correspond to a cohomology class in $H^1(\Gamma_{F_\nu},k(\overline{\chi}))$. We consider the inflation-restriction sequence

$$0 \longrightarrow H^1\left(\Gal(F_\nu^{\mathrm{tame},p}/F_\nu),k(\overline{\chi})^{\Gamma_{F_\nu^{\mathrm{tame},p}}}\right) \longrightarrow H^1(\Gamma_{F_\nu},k(\overline{\chi})) \longrightarrow H^1(\Gamma_{F_\nu^{\mathrm{tame},p}},k(\overline{\chi})).$$

\vspace{2 mm}

But, $\Gamma_{F_\nu^{\mathrm{tame},p}}$ has no non-trivial finite $p$-quotients, so the last cohomology group is trivial. Therefore, $\restr{c_1}{\Gamma_{F_\nu}}$ is cohomologous to a cocycle which is trivial on $\Gamma_{F_\nu^{\mathrm{tame},p}}$. Let $m - \sigma \cdot m$ be the coboundary which is the difference between the two cocycles. Conjugating by $\begin{pmatrix} 1 & m \\ 0 & 1 \end{pmatrix}$ allows us to assume that $c_1$ is trivial on $\Gamma_{F_\nu^{\mathrm{tame},p}}$. By a similar reasoning we assume the same is true for $c_2$. 

\begin{itemize}
    \item We first consider the case when $\overline{\chi}$ is non-trivial on $\Gamma_{F_\nu^{\mathrm{tame},p}}$. By our assumption on the cocycles $c_1$ and $c_2$ we have that $\restr{\overline{\rho_i}}{\Gamma_{F_\nu^{\mathrm{tame},p}}} \simeq \overline{\chi} \oplus 1$ for both $i=1,2$. As this restriction is non-scalar we note that all possible strong inertial types for lifts of $\restr{\overline{\rho_i}}{\Gamma_{F_\nu}}$ are of the form $\tau = (r_\tau,0)$. Indeed, if $\rho_{i,\nu}$ is a $p$-adic lift of $\restr{\overline{\rho_i}}{\Gamma_{F_\nu}}$ with strong inertial type $\tau$ we have that $r_\tau(g) = \rho_{i,\nu}(g)$ for $g \in \Gamma_{F_\nu^{\mathrm{tame},p}}$. As $(r_\tau,N_\tau)$ extends to a Weil-Deligne representation, for $g \in \Gamma_{F_\nu^{\mathrm{tame},p}}$ we have that $\rho_{i,\nu}(g)$ commutes with $N_\tau$. This will force $\rho_{i,\nu}(g)$ to be scalar for each $g \in \Gamma_{F_\nu^{\mathrm{tame},p}}$ if $N_\tau \neq 0$, which is impossible, as $\restr{\overline{\rho_i}}{\Gamma_{F_\nu^{\mathrm{tame},p}}}$ is non-scalar. 
    
    Then by \cite[Proposition $5.1$]{Sho16} we have

    $$R_{\restr{\overline{\rho_i}}{\Gamma_{F_\nu}}}^{\sqr} \simeq \frac{\mathcal O \llbracket X_1,X_2,Y_1,Y_2,Z_1,Z_2 \rrbracket}{((1+X_1)^{p^{a_\nu}}-1,(1+X_2)^{p^{a_\nu}} - 1)},$$

    \vspace{2 mm}

    \noindent and its $p^{2a_\nu}$ irreducible components are exactly the quotients $R_{\restr{\overline{\rho_i}}{\Gamma_{F_\nu}}}^{\sqr,\tau}$ for the strong inertial types $\tau = ((\restr{[\overline{\chi}]}{I_{F_\nu}}) \otimes f_1) \oplus (\restr{1}{I_{F_\nu}} \otimes f_2),0)$, where $f_1,f_2$ run through all pairs of characters $I_{F_\nu} \to \overline{E}^\times$ which extend to $\Gamma_{F_\nu}$ and reduce to the trivial character. From this description we can easily choose $p$-adic lifts $\rho_{1,\nu}$ and $\rho_{2,\nu}$ which will have the same strong inertial type, for example $(\restr{[\overline{\chi}]}{I_{F_\nu}} \oplus \restr{1}{I_{F_\nu}},0)$.

    \item We now assume that $\overline{\chi}$ is trivial on $\Gamma_{F_\nu^{\mathrm{tame},p}}$. This will be an overarching assumption in the rest of the cases and we will not mention it explicitly anymore. We note that it implies that $\overline{\chi}$ is unramified. Indeed, as $\tau_\nu$ generates a pro-$p$-group and $\overline{\chi}$ is valued in $k^\times$ we must have that $\overline{\chi}(\tau_\nu) = 1$. Combining this with our assumption that $\overline{\chi}$ is trivial on $\Gamma_{F_\nu^{\mathrm{tame},p}}$ we conclude that $\overline{\chi}$ is trivial on $I_{F_\nu}$.
    
    We now suppose that $\overline{\chi}(\sigma_\nu) \neq 1,N(\nu),N(\nu)^{-1}$. This means that the eigenvalues of the image of $\sigma_\nu$ under both $\overline{\rho_1}$ and $\overline{\rho_2}$ have ratio which is different that $1,N(\nu)$, or $N(\nu)^{-1}$. Then, by \cite[Proposition $5.3$]{Sho16} we have that

     $$R_{\restr{\overline{\rho_i}}{\Gamma_{F_\nu}}}^{\sqr} \simeq \frac{\mathcal O \llbracket A,B,P,Q,X,Y \rrbracket }{((1+P)^{p^{a_\nu}}-1,(1+Q)^{p^{a_\nu}}-1)}.$$

     \vspace{2 mm}

     Moreover, for a $p^{a_\nu}$-th root of unity $\zeta$ we have $R_{\restr{\overline{\rho_i}}{\Gamma_{F_\nu}}}^{\sqr,\tau_{\zeta,s}} \simeq \mathcal O \llbracket A,B,X,Y \rrbracket$; and if $N(\nu) \equiv 1 \pmod p$ for distinct $p^{a_\nu}$-th roots of unity $\zeta_1,\zeta_2$ we have $R_{\restr{\overline{\rho_i}}{\Gamma_{F_\nu}}}^{\sqr,\tau_{\zeta_1,\zeta_2}} \simeq \mathcal O \llbracket A,B,P,X,Y \rrbracket /(1 + P - \zeta_1)(1+P-\zeta_2)$. For the rest of the strong inertial types the quotients $R_{\restr{\overline{\rho_i}}{\Gamma_{F_\nu}}}^{\sqr,\tau}$ are zero. We choose the $p$-adic lifts $\rho_{1,\nu}$ and $\rho_{2,\nu}$ so that they both are of strong inertial type $\tau_{\zeta,s}$ for some $p^{a_\nu}$-th root of unity. 

     \item Now, suppose that $N(\nu) \not \equiv \pm 1 \pmod p$. As most of the options were covered in the previous case we only need to consider the cases when $\overline{\chi}(\sigma_\nu) = 1,N(\nu)$, or $N(\nu)^{-1}$. We take a look at each of them separately. 

     \begin{itemize}
         \item Suppose that $\overline{\chi}(\sigma_\nu) = 1$. Using the relation between $\sigma_\nu$ and $\tau_\nu$ we have

         \begin{align*}
             \begin{pmatrix} 1 & N(\nu)c_1(\tau_\nu) \\ 0 & 1\end{pmatrix} = \begin{pmatrix} 1 & c_1(\sigma_\nu) \\ 0 & 1\end{pmatrix}\begin{pmatrix} 1 & c_1(\tau_\nu) \\ 0 & 1\end{pmatrix}\begin{pmatrix} 1 & -c_1(\sigma_\nu) \\ 0 & 1\end{pmatrix} = \begin{pmatrix} 1 & c_1(\tau_\nu) \\ 0 & 1\end{pmatrix}.
         \end{align*}

        \vspace{2 mm}

         Comparing the top right entries we get that $(N(\nu)-1)c_1(\tau_\nu) = 0$. By our assumption we have that $N(\nu) \not \equiv 1 \pmod p$, so this means that $c_1(\tau_\nu) = 0$, and therefore $\overline{\rho_1}$ is unramified at $\nu$. A similar computation also tells us that $\overline{\rho_2}$ is unramified at $\nu$. Therefore, by \cite[Proposition $5.5.(1)$]{Sho16} we have that $R_{\restr{\overline{\rho_i}}{\Gamma_{F_\nu}}}^{\sqr} \simeq R_{\restr{\overline{\rho_i}}{\Gamma_{F_\nu}}}^{\sqr,\tau_{1,s}} \simeq \mathcal O \llbracket P,Q,R,S \rrbracket$. Hence, we can find $p$-adic lifts $\rho_{1,\nu}$ and $\rho_{2,\nu}$ so that both have strong inertial type $\tau_{1,s}$.

        \item Now suppose that $\overline{\chi}(\sigma_\nu) = N(\nu)^{-1}$. Again, using the relation between $\sigma_\nu$ and $\tau_\nu$ we get

        \begin{align*}
            \begin{pmatrix} 1 & N(\nu)c_1(\tau_\nu) \\ 0 & 1\end{pmatrix} &= \begin{pmatrix} N(\nu)^{-1} & c_1(\sigma_\nu) \\ 0 & 1\end{pmatrix}\begin{pmatrix} 1 & c_1(\tau_\nu) \\ 0 & 1\end{pmatrix}\begin{pmatrix} N(\nu) & -N(\nu)c_1(\sigma_\nu) \\ 0 & 1\end{pmatrix} \\ &= \begin{pmatrix} 1 & N(\nu)^{-1}c_1(\tau_\nu) \\ 0 & 1\end{pmatrix}.
        \end{align*}

        \vspace{2 mm}

        Comparing the top right entries we get $(N(\nu)^2-1)c_1(\tau_\nu) = 0$. As $N(\nu) \not \equiv \pm 1 \pmod p$ we must have that $c_1(\tau_\nu) = 0$, i.e. $\overline{\rho_1}$ is unramified at $\nu$. Therefore, $\restr{c_1}{\Gamma_{F_\nu}}$ is an inflation of an unramified class $H^1(\Gamma_{F_\nu}/I_{F_\nu},k(\overline{\chi}))$. As $\sigma_\nu$ acts on $k(\overline{\chi})$ as multiplication by $N(\nu)^{-1}$, and $N(\nu)^{-1}-1$ is a unit in $k$ this cohomology group is trivial. Hence, we can assume that $\restr{c_1}{\Gamma_{F_\nu}} = 0$. After twisting by $\overline{\chi}$ using \cite[Proposition $5.5.(2)$]{Sho16} we have that
        
        $$R_{\restr{\overline{\rho_1}}{\Gamma_{F_\nu}}}^{\sqr} \simeq \frac{\mathcal O \llbracket X_1,X_2,X_3,X_4,X_5 \rrbracket}{(X_1X_2)},$$

        \vspace{2 mm}

        \noindent with the quotients by the two minimal primes being $R_{\restr{\overline{\rho_1}}{\Gamma_{F_\nu}}}^{\sqr,\tau_{1,s}}$ and $R_{\restr{\overline{\rho_1}}{\Gamma_{F_\nu}}}^{\sqr,\tau_{1,ns}}$. On the other side, the relation between $\sigma_\nu$ and $\tau_\nu$ will be automatically satisfied for any choice of a cocycle $c_2$. If $\overline{\rho_2}$ is not ramified at $\nu$ then as it was the case for $\overline{\rho_1}$ we have that

        $$R_{\restr{\overline{\rho_2}}{\Gamma_{F_\nu}}}^{\sqr} \simeq \frac{\mathcal O \llbracket X_1,X_2,X_3,X_4,X_5 \rrbracket}{(X_1X_2)},$$

        \vspace{2 mm}

        \noindent with the quotients by the two minimal primes being $R_{\restr{\overline{\rho_2}}{\Gamma_{F_\nu}}}^{\sqr,\tau_{1,s}}$ and $R_{\restr{\overline{\rho_2}}{\Gamma_{F_\nu}}}^{\sqr,\tau_{1,ns}}$. If $\overline{\rho_2}$ is ramified at $\nu$ then again by \cite[Proposition $5.5.(2)$]{Sho16} we have that $R_{\restr{\overline{\rho_2}}{\Gamma_{F_\nu}}}^{\sqr} \simeq R_{\restr{\overline{\rho_2}}{\Gamma_{F_\nu}}}^{\sqr,\tau_{1,ns}} \simeq \mathcal O \llbracket B,P,X,Y \rrbracket$. In summary, we can always choose the $p$-adic lifts $\rho_{1,\nu}$ and $\rho_{2,\nu}$ to be of strong inertial type $\tau_{1,ns}$.

        \item Suppose that $\overline{\chi}(\sigma_\nu) = N(\nu)$. Then we get the same results as in the previous subcase with the roles of $\overline{\rho_1}$ and $\overline{\rho_2}$ reversed. In particular, $\overline{\rho_2}$ will be unramified at $\nu$ and it will have $p$-adic lifts of strong inertial type $\tau_{1,s}$ and $\tau_{1,ns}$. On the other hand, $\overline{\rho_1}$ could ramify at $\nu$ in which case it has $p$-adic lifts only of strong inertial type $\tau_{1,ns}$, while if it doesn't it has $p$-adic lifts of strong inertial types $\tau_{1,s}$ and $\tau_{1,ns}$. In either case, as above we can choose the $p$-adic lifts $\rho_{1,\nu}$ and $\rho_{2,\nu}$ to be of strong inertial type $\tau_{1,ns}$.
        
     \end{itemize}

     \item Assume that $N(\nu) \equiv -1 \pmod p$. Then, the only options that have not yet been considered are $\overline{\chi}(\sigma_\nu) = 1$ and $\overline{\chi}(\sigma_\nu) = -1$.

     \begin{itemize}
         \item Let $\overline{\chi}(\sigma_\nu) = 1$. Then as in the previous case this condition implies that $(N(\nu)-1)c_1(\tau_\nu) = 0$, and as $N(\nu) \not \equiv 1 \pmod p$ we get that $c_1(\tau_\nu) = 0$, i.e. $\overline{\rho_1}$ isn't ramified at $\nu$. Similarly, $\overline{\rho_2}$ isn't ramified at $\nu$. Thus, by \cite[Proposition $5.6.(1)$]{Sho16} we have that $R_{\restr{\overline{\rho_i}}{\Gamma_{F_\nu}}}^{\sqr} \simeq R_{\restr{\overline{\rho_i}}{\Gamma_{F_\nu}}}^{\sqr,\tau_{1,s}} \simeq \mathcal O \llbracket P,Q,R,S \rrbracket$. Hence, we can find $p$-adic lifts $\rho_{1,\nu}$ and $\rho_{2,\nu}$ so that both have strong inertial type $\tau_{1,s}$.

         \item Suppose that $\overline{\chi}(\sigma_\nu) = -1$. Computations as above tell us that for any cocycle $c_i$, $\overline{\rho_i}$ will preserve the relation between $\sigma_\nu$ and $\tau_\nu$, so there are no restriction on the cocycle $c_i$. If $c_i(\tau_\nu) \neq 0$ by \cite[Proposition $5.6.(2)(a)$]{Sho16} we have that $R_{\restr{\overline{\rho_i}}{\Gamma_{F_\nu}}}^{\sqr,\tau_{1,ns}}$ and $R_{\restr{\overline{\rho_i}}{\Gamma_{F_\nu}}}^{\sqr,\tau_{\xi}}$ are isomorphic to $\mathcal O \llbracket P,Q,X,Y \rrbracket$, where we recall that $\xi$ is any $p^{b_\nu}$-th root of unity. Moreover, these are the only possible strong inertial types for lifts of $\restr{\overline{\rho_i}}{\Gamma_{F_\nu}}$. On the other side, if $c_i(\tau) = 0$, then by \cite[Proposition $5.6.(2)(b)$]{Sho16} the possible strong inertial types for lifts of $\restr{\overline{\rho_i}}{\Gamma_{F_\nu}}$ are $\tau_{1,s}, \tau_{1,ns}$, and $\tau_{\xi}$. In particular

         $$R_{\restr{\overline{\rho_i}}{\Gamma_{F_\nu}}}^{\sqr,\tau_{1,ns}} \simeq \frac{\mathcal O \llbracket X_1,X_2,X_3,X_4,X_5,X_6 \rrbracket}{((X_1,X_3) \cap (X_2,X_3 - (N(\nu)+1)))},$$

         \vspace{2 mm}

         \noindent whose spectrum is the scheme theoretic union of two formally smooth components that do not intersect in the generic fiber. From this we conclude that we can always choose $p$-adic lifts $\rho_{1,\nu}$ and $\rho_{2,\nu}$ which will have strong inertial type $\tau_{1,ns}$.

      \end{itemize}

      \item Finally, we consider the case $N(\nu) \equiv 1 \pmod p$. The only possibility that we need to consider is $\overline{\chi}(\sigma_\nu) = 1$. We claim that we can choose $p$-adic lifts $\rho_{1,\nu}$ and $\rho_{2,\nu}$ both with strong inertial type $\tau_{1,ns}$. This follows immediately from \cite[Proposition $5.8$]{Sho16}. If either $c_i(\sigma_\nu) \neq 0$ or $c_i(\tau_\nu) \neq 0$ we have that

      $$R_{\restr{\overline{\rho_i}}{\Gamma_{F_\nu}}}^{\sqr,\tau_{1,ns}} \simeq \mathcal O \llbracket A,B,R,T \rrbracket.$$

      \vspace{2 mm}

      On the other side, if $c_i(\sigma_\nu) = c_i(\tau_\nu) = 0$ we have that $R_{\restr{\overline{\rho_i}}{\Gamma_{F_\nu}}}^{\sqr,\tau_{1,ns}}$ is a domain of relative dimension $4$ over $\mathcal O$, which moreover has a formally smooth generic fiber. This tells us that in this case we can always choose $\rho_{i,\nu}$ to be a $p$-adic lift of $\restr{\overline{\rho_i}}{\Gamma_{F_\nu}}$ with strong inertial type $\tau_{1,ns}$.

\end{itemize}

\end{proof}

We recall that throughout the lifting method of \S $3$ and \S $4$ we have fixed a multipler type $\mu_i$ and we were interested in producing lifts of $\overline{\rho_i}$ of multipler type $\mu_i$. In the $\GL_2$-case this translates to producing lifts with a prescribed determinant $\mu$, which we fixed at the beginning of the argument. Therefore, as in the fourth bullet point of Assumption \ref{GeneralAssumptions} it is important that the local lifts produced by Proposition \ref{LocalLiftsnmidp} all have determinant $\mu$. The next results tells us that this is always possible. 

\begin{cor}
\label{LocalLiftsFixedDeterminant}

For each $\nu \nmid p \in S$ and $i=1,2$ there exist local lifts $\rho_{i,\nu}:\Gamma_{F_\nu} \to \GL_2(\mathcal O)$ of $\restr{\overline{\rho_i}}{\Gamma_{F_\nu}}$ such that $\rho_{1,\nu}$ and $\rho_{2,\nu}$ both have determinant $\mu$ and have the same strong inertial type.
    
\end{cor}
\begin{proof}

    By Proposition \ref{LocalLiftsnmidp} there exist such lifts $\rho_{1,\nu}$ and $\rho_{2,\nu}$ which have the same strong inertial type. It remains to arrange the condition on the determinant. We do this by twisting by a suitable character. 

    We consider the characters $\psi_i:\Gamma_{F_\nu} \to \mathcal O^\times$ given by $\psi_i = \mu \cdot (\det(\rho_{i,\nu}))^{-1}$ As both $\mu$ and $\det(\rho_{i,\nu})$ are lifts of $\det(\overline{\rho_{i,\nu}}) = \overline{\chi}$ we have that $\psi_i$ is a lift of the trivial character. Hence, its image lands inside $1 + \varpi\mathcal O$. As $p \ge 3$, this allows us to take a square root of $\psi_i$. Indeed, $\sqrt{1 + x}$ is well-defined on $\varpi\mathcal O$, as its Taylor series converges for $x \in \varpi\mathcal O$. We then get a character $\phi_i$ such that $\phi_i^2 = \psi_i$. We then define $\rho'_{i,\nu} \coloneqq \rho_{i,\nu} \otimes \phi_i$. We immediately note that $\rho_{i,\nu}'$ is a lift of $\restr{\overline{\rho_i}}{\Gamma_{F_\nu}}$ and $\det(\rho'_{i,\nu}) = \det(\rho_{i,\nu})\phi_i^2 = \det(\rho_{i,\nu})\psi = \mu$, which is exactly what we wanted. 

    It remains to see how twisting by $\phi_i$ will affect the strong inertial type of $\rho_{i,\nu}$. We first note that by the virtue of being a lift of the trivial character, as explained in \cite[Lemma $2.5$]{Sho16} the character $\phi_i$ will be trivial on $\Gamma_{F_\nu^{\mathrm{tame},p}}$ and $\im(\restr{\phi_i}{I_{F_\nu}})$ will consist of $p^{a_\nu}$-th roots of unity. Let $\tau_i = (r_{\tau_i},N_{\tau_i})$ and $\tau_i' = (r_{\tau_i'},N_{\tau_i'})$ be the strong inertial types of $\rho_{i,\nu}$ and $\rho_{i,\nu}'$, respectively. Now, by Grothendieck's monodromy theorem there exists a finite-index subgroup $H_i$ of $I_{F_\nu}$ such that for $h \in H_i$ we have $\rho_{i,\nu}(h)$ is unipotent. We can use this to compute $N_{\tau_i}$. In particular $N_{\tau_i} = t_p(h)^{-1}\log(\rho_{i,\nu}(h))$ for any $h \in H_i$ such that $t_p(h) \neq 0$, where $t_p:I_{F_\nu} \to \mathbb Z_p$ is a surjective group homomorphism, which is unique up to multiplication by a unit. On the other hand, for $h \in H_i$ we have $\rho_{i,\nu}'(h^{a_\nu}) = \phi_i(h)^{a_\nu}\rho_{i,\nu}(h^{a_\nu}) = \rho_{i,\nu}(h^{a_\nu})$, which is unipotent. Choosing $h \in H_i$ s.t. $t_p(h) \neq 0$ we can compute $N_{\tau_i'}$. We get

    $$N_{\tau_i'} = t_p(h^{a_\nu})^{-1}\log(\rho_{i,\nu}'(h^{a_\nu})) = t_p(h^{a_\nu})^{-1}\log(\rho_{i,\nu}(h^{a_\nu})) = N_{\tau_i}.$$

    \vspace{2 mm}

    Therefore, twisting by a character whose restriction to $I_{F_\nu}$ has finite image will not affect the monodromy operator of the associated Weil-Deligne representation. On the other side, as $N_{\tau_i} = N_{\tau_i'}$ we have $r_{\tau_i'} = \phi_i \cdot r_{\tau_i}$. As $\det(\rho_{i,\nu}) = \det(r_{\tau_i})$ for elements in the inertia group $I_{F_\nu}$ we have that $\restr{\phi_i}{I_{F_\nu}}$ is uniquely determined by $r_{\tau_i}$ and $\mu$. Therefore, the strong inertial type $\tau_i'$ are uniquely determined by $\tau_i$ and $\mu$. As $\tau_1 = \tau_2$ we have get that the representations $\rho_{1,\nu}'$ and $\rho_{2,\nu}'$ have the same strong inertial type. 
    
\end{proof}

We now shift our attention to primes $\nu$ that divide $p$. At such primes we want to produce lifts $\rho_{i,\nu}$ of $\restr{\overline{\rho_i}}{\Gamma_{F_\nu}}$ which will be de Rham. By \cite{Ber02} such representations are potentially semi-stable, i.e. they become semi-stable after restricting to the absolute Galois group of a finite extension $L$ of $F_\nu$. Then, as in \cite[\S $1.3$]{Fon94} we can associate a Weil-Deligne representation to the $(\varphi,N,\Gal(L/F_\nu))$-module $D_{\mathrm{st},L}(\rho_{i,\nu}) \coloneqq (B_{\mathrm{st}} \otimes_{\mathbb Q_p} V_{\rho_{i,\nu}})^{\Gamma_L}$, where $V_{\rho_{i,\nu}}$ is the two-dimensional $E$-vector space on which $\Gamma_{F_\nu}$ acts via $\rho_{i,\nu}$. The associated Weil-Deligne representation doesn't depend on the choice of $L$ (see \cite[Lemma $2.2.1.2$]{BM02}). We use this Weil-Deligne representation to define the strong inertial type of the lifts $\rho_{i,\nu}$. In comparison to the case for primes not dividing $p$, for $\nu \mid p$ we will produce lifts $\rho_{1,\nu}$ and $\rho_{2,\nu}$ whose associated Weil-Deligne representations have the same Artin conductor, but are not necessarily of the same inertial type. For that we will need the following technical results:

\begin{lem}
\label{ArbitraryArtinConductor}

    For every integer $n \ge 1$ there exists a character $\psi:\Gamma_{F_\nu} \to (\mathcal O_n)^\times$ which is a lift of the trivial character and has Artin conductor $\mathfrak f(\psi) \ge n$, where $\mathcal O_n \coloneqq \mathcal O[\zeta_{p^n}]$. Moreover, $\psi$ is trivial modulo $\varpi_n^m$, where $\varpi_n$ is a uniformizer of $\mathcal O_n$ and $m \coloneqq v_{\mathcal O_n}(\zeta_{p^n}-1)$, but is non-trivial modulo $\varpi_n^{m+1}$.
    
\end{lem}
\begin{proof}

   We adjoin $p^a$-th roots of unity $\mathbb \zeta_{p^a}$ to $F_\nu$. We then have that $\Gal(F_\nu(\zeta_{p^a})/F_\nu) \simeq \mathbb Z/m \times \mathbb Z/p^b$, where $m$ is some integer dividing $p-1$ and $b \le a-1$ with both depending on $F_\nu$. By taking $a$ large enough we can assume that $b = n$. Therefore, $\Gal(F_\nu(\zeta_{p^a})/F_\nu)$, and subsequently $\Gamma_{F_\nu}$ have quotient isomorphic to $\mathbb Z/p^n$. We then map it isomorphically to the $p^n$-th roots of unity in $\mathcal O_n$ to obtain our character $\psi$. As $\zeta_{p^n}-1$ is a factor of $p$ in $\mathcal O_n$ we have that it has trivial reduction in $k$, the residue field of $\mathcal O_n$. Therefore, $\overline{\psi} = 1$, i.e. $\psi$ is a lift of the trivial character. In fact, from this explicit description we have that $\psi$ is trivial modulo $\varpi_n^m$, but not modulo $\varpi_n^{m+1}$. It remains to show $\psi$ has large enough Artin conductor. 

   The extension $F_\nu(\zeta_{p^a})/F_\nu$ will be totally ramified with the factor $\mathbb Z/p^n$ corresponding to the wild inertia subgroup. Moreover, the ramification groups $G_i$ of $\Gal(F_\nu(\zeta_{p^a})/F_\nu)$ will decrease in size by a factor of $p$, and each of these drops occurs when $i$ is a power of $p$. By construction $\psi$ will be non-trivial on any non-zero ramification group. Therefore

   $$\mathfrak f(\psi) = \sum_{i=0}^\infty \frac{\mathrm{codim } \, E'(\psi)^{G_i}}{[G_0:G_i]} = 1 + \sum_{j=0}^{n-1} \,\,\,\sum_{i=p^j}^{p^{j+1}-1} \,\,\frac{1}{[G_0:G_i]} = 1 + \sum_{j=0}^{n-1} \,\,\, \sum_{i=p^j}^{p^{j+1}-1} \,\,\frac{1}{mp^{j}} = 1 + \sum_{j=0}^{n-1} \frac{(p-1)p^j}{mp^j} = 1 + \frac{(p-1)n}m.$$

   \vspace{2 mm}

   As $m \mid p-1$, the Artin conductor of $\psi$ will be at least $n$.
    
\end{proof}

\begin{lem}
\label{ProductofCharConductor}

    Let $\psi,\psi':\Gamma_{F_\nu} \to \mathcal O^\times$ be finitely ramified characters. Assume $\mathfrak f(\psi) > \mathfrak f(\psi')$. Then $\mathfrak f(\psi\psi') = \mathfrak f(\psi)$.

\end{lem}
\begin{proof}

    We can view the characters as maps from $\Gal(F_\nu(\psi,\psi')/F_\nu)$, and use its ramification groups to compute their Artin conductors. We note that $E(\psi)^{G_i}$ will have dimension $0$ or $1$. Moreover, if the dimension is $0$ for some $n$, then it will be $0$ for $i \ge n$. Therefore, from the assumption $\mathfrak f(\psi) > \mathfrak f(\psi')$ using the definition of the Artin character we deduce that there exists an integer $i$ such that $\restr{\psi'}{G_i} = 1$, but $\restr{\psi}{G_i} \neq 1$. From this we deduce that $\restr{\psi\psi'}{G_i} = 1$ if and only if $\restr{\psi}{G_i} = 1$. Therefore:

    $$\mathfrak f(\psi\psi') = \sum_{i=0}^\infty \frac{\mathrm{codim } \, E(\psi\psi')^{G_i}}{[G_0:G_i]} = \frac{\mathrm{codim } \, E(\psi)^{G_i}}{[G_0:G_i]} = \mathfrak f(\psi).$$
    
\end{proof}

\begin{cor}
    \label{TameandWild}

    Let $\psi:\Gamma_{F_\nu} \to \mathcal O^\times$ be a finitely ramified character with non-trivial image on wild inertia. Let $\chi:\Gamma_{F_\nu} \to \mathcal O^\times$ be a finitely tamely ramified character. Then $\mathfrak f(\psi\chi) = \mathfrak f(\psi)$.
    
\end{cor}
\begin{proof}
    
    As $\chi$ is tamely ramified we have $\mathfrak f(\chi) = 1$ or $0$, with the latter being the case if $\chi$ is unramified. On the other side, as $\psi$ is wildly ramified from the definition of the Artin conductor we have that $\mathfrak f(\psi) > 1$. Therefore, by Lemma \ref{ProductofCharConductor} we have $\mathfrak f(\psi\chi) = \mathfrak f(\psi)$.
    
\end{proof}

We now have the following key result which is based on \cite[Lemma $7.2$]{FKP22}, which in turn generalizes \cite[Lemma $6.1.6$]{BLGG12}.

\begin{prop}
\label{LocalLiftsmidp}

For each $\nu \mid p$ and $i=1,2$, there exist a finite extension $\mathcal O'$ of $\mathcal O$, and a potentially crystalline lift $\rho_{i,\nu}:\Gamma_{F_\nu} \to \GL_2(\mathcal O')$ of $\restr{\overline{\rho_i}}{\Gamma_{F_\nu}}$ with determinant $\mu$ such that

$$\rho_{1,\nu} = \begin{pmatrix} \kappa^{r-1}\chi_1 & * \\ 0 & \chi_2\end{pmatrix} \quad \quad \text{ and } \quad\quad \rho_{2,\nu} = \begin{pmatrix} \kappa^{r-1}\chi_1' & * \\ 0 & \chi_2'\end{pmatrix},$$

\vspace{2 mm}

\noindent where $\chi_1,\chi_2,\chi_1',\chi_2'$ are finitely ramified characters such that $\mathfrak f(\chi_1) = \mathfrak f(\chi_1')$ and $\mathfrak f(\chi_2) = \mathfrak f(\chi_2')$.

\end{prop}
\begin{proof}

    Any lifts of the given forms will be automatically potentially crystalline for $r \ge 3$. As remarked in the proof of \cite[Lemma $6.1.6$]{BLGG12} this is also true when $r=2$, $\chi_1 \neq \chi_2$, and $\chi_1' \neq \chi_2'$. This claim follows immediately from the computations in \cite[Example $3.9$]{BK90} and the remark preceding \cite[Proposition $3.8$]{BK90}.

    In order for $\rho_{1,\nu}$ of that form to be a lift of $\restr{\overline{\rho_1}}{\Gamma_{F_\nu}}$ we need to choose $\chi_1$ and $\chi_2$ so that $\overline{\chi_1} = \overline{\chi}\overline{\kappa}^{1-r}$, $\overline{\chi_2} = 1$, and $\chi_1\chi_2 = [\overline{\chi}\overline{\kappa}^{1-r}]$. In addition, we will ask for $\chi_1$ and $\chi_2$ to be distinct characters. We also need to choose $\chi_1'$ and $\chi_2'$ satisfying the analogous conditions, i.e. $\chi_1'$ and $\chi_2'$ are distinct finitely ramified characters such that $\overline{\chi'_1} = \overline{\kappa}^{1-r}$, $\overline{\chi'_2} = \overline{\chi}$, and $\chi_1'\chi_2' = [\overline{\chi}\overline{\kappa}^{1-r}]$

    We first consider the case $\restr{\overline{\chi}}{\Gamma_{F_\nu}} \neq \overline{\kappa}, \overline{\kappa}^{-1}$. We will show that for any choice of characters $\chi_1,\chi_2,\chi_1'$, and $\chi_2'$ that satisfy the aforementioned conditions we can produce the wanted lifts $\rho_{1,\nu}$ and $\rho_{2,\nu}$. Let $\chi_1$ and $\chi_2$ be such characters. To produce $\rho_{1,\nu}$ of the desired form that will lift $\restr{\overline{\rho_1}}{\Gamma_{F_\nu}}$ we need to show that we can lift the modulo $\varpi$ extension class $\restr{c_1}{\Gamma_{F_\nu}}$ to a $p$-adic extension class. It would suffice to show that the map $H^1(\Gamma_{F_\nu},\mathcal O(\kappa^{r-1}\chi_1\chi_2^{-1})) \to H^1(\Gamma_{F_\nu},k(\overline{\chi}))$ induced by reducing modulo $\varpi$ is surjective. We start with the short exact sequence of $\mathcal O[\Gamma_{F_\nu}]$-modules

    $$0 \longrightarrow \mathcal O(\kappa^{r-1}\chi_1\chi_2^{-1}) \xlongrightarrow{\,\cdot \varpi \, } \mathcal O(\kappa^{r-1}\chi_1\chi_2^{-1}) \longrightarrow k(\overline{\chi}) \longrightarrow 0.$$

    \vspace{2 mm}

    It induces a long exact sequence on cohomology which tells us that the cokernel of the map above lies in $H^2(\Gamma_{F_\nu},\mathcal O(\kappa^{r-1}\chi_1\chi_2^{-1}))$. By Tate duality this is the Pontryagin dual of $H^0(\Gamma_{F_\nu},(E/\mathcal O)(\kappa^{2-r}\chi_1^{-1}\chi_2))$. We have $\overline{\kappa^{2-r}\chi_1^{-1}\chi_2} = \overline{\kappa}\overline{\chi}^{-1} \neq 1$, by our assumption. Therefore, the $\Gamma_{F_\nu}$-action isn't trivial even on elements in $\varpi^{-1}\mathcal O$, so it has no fixed points. We conclude that $H^0(\Gamma_{F_\nu},(E/\mathcal O)(\kappa^{r-1}\chi_1\chi_2^{-1})) = 0$, which gives us the desired surjection. A similar argument, relying on the fact that $\overline{\kappa}\overline{\chi} \neq 1$ tells us that we produce a lift $\rho_{2,\nu}$ of $\restr{\overline{\rho_2}}{\Gamma_{F_\nu}}$ for any choice of $\chi_1'$ and $\chi_2'$ that satisfy the conditions. 

    By taking Teichm\"uller lifts, i.e. $\chi_1 = [\overline{\chi}\overline{\kappa}^{1-r}]$, $\chi_2 = 1$, $\chi_1' = [\overline{\kappa}^{1-r}]$, and  $\chi_2' = [\overline{\chi}]$ it is easy to see that characters satisfying the specified conditions exist. Now, from Lemma \ref{ArbitraryArtinConductor} by enlarging $\mathcal O$, but still keeping the same notation, we can produce a finitely ramified character $\psi:\Gamma_{F_\nu} \to \mathcal O^\times$ which is a lift of the trivial character and has Artin conductor $\mathfrak f(\psi) \ge 2$. We then replace $\chi_1$ with $\chi_1\psi$, $\chi_2$ with $\chi_2\psi^{-1}$, $\chi_1'$ with $\chi_1'\psi$, and $\chi_2'$ with $\chi_2'\psi^{-1}$. As $\overline{\psi} = 1$ this will not affect the reductions modulo $\varpi$. Clearly, the characters are still distinct and the determinant of the lifts stays the same. Also, as all the characters are finitely ramified the same will be true for their products. Thus, these twists will still satisfy the above conditions and we can produce the wanted lifts $\rho_{1,\nu}$ and $\rho_{2,\nu}$. Moreover, as each of the residual characters are valued in $k^\times$, their Teichm\"uller lifts are all tamely ramified. On the other side, $\psi$ is wildly ramified by construction, so by Corollary \ref{TameandWild} we have that $\mathfrak f(\chi_1\psi) = \mathfrak f(\chi_2\psi^{-1}) = \mathfrak f(\chi_1'\psi) = \mathfrak f(\chi_2'\psi^{-1}) = \mathfrak f(\psi) \ge 2$. 

    Now, suppose that $\restr{\overline{\chi}}{\Gamma_{F_\nu}} = \overline{\kappa}$. As seen in the previous case, for this $\overline{\chi}$ we can always produce a $p$-adic lift of $\restr{\overline{\rho_2}}{\Gamma_{F_\nu}}$ for any choice of $\chi_1'$ and $\chi_2'$ that satisfy the conditions stated at the beginning. On the other side, this is not always possible for $\overline{\rho_1}$ and we have to choose the characters $\chi_1$ and $\chi_2$ more carefully. We make a preliminary choice of $\chi_1 = [\overline{\chi}\overline{\kappa}^{1-r}]$ and $\chi_2 = 1$, which will be subject to a change. Immediately from the beginning we will ask for $\kappa^{2-r}\chi_1^{-1}\chi_2$ to be non-trivial. As $\chi_1$ and $\chi_2$ are finitely ramified character this is always true for $r \neq 2$. When $r=2$ instead of the previous choice we take $\chi_1$ to be the unramified character that sends $\sigma_\nu$ to any non-trivial element of $1+\varpi\mathcal{O}$ and $\chi_2$ to be its inverse. Clearly, $\chi_1$ and $\chi_2$ still satisfy the desired conditions and $\kappa^{2-r}\chi_1^{-1}\chi_2$ is non-trivial. Now, let $L$ be the span of the class $\restr{c_1}{\Gamma_{F_\nu}}$ inside $H^1(\Gamma_{F_\nu},k(\overline{\kappa}))$, and let $H$ be the subspace of $H^1(\Gamma_{F_\nu},k)$ that annihilates $L$ under the Tate duality pairing. As in the previous case we have an exact sequence

    $$\cdots \longrightarrow H^1(\Gamma_{F_\nu},\mathcal O(\kappa^{r-1}\chi_1\chi_2^{-1})) \longrightarrow H^1(\Gamma_{F_\nu},k(\overline{\kappa})) \xlongrightarrow{\, \delta_1 \, } H^2(\Gamma_{F_\nu},\mathcal O(\kappa^{r-1}\chi_1\chi_2^{-1})) \longrightarrow \cdots.$$

    \vspace{2 mm}

    As in the previous case we want to show that $\restr{c_1}{\Gamma_{F_\nu}}$ lies in the image of the first map. By the exactness we need to show that $\delta_1(L) = 0$. On the other hand, we consider the short exact sequence of $\mathcal O[\Gamma_{F_\nu}]$-modules

    $$0 \longrightarrow k \longrightarrow (E/\mathcal O)(\kappa^{2-r}\chi_1^{-1}\chi_2) \xlongrightarrow{\,\cdot \varpi \, } (E/\mathcal O)(\kappa^{2-r}\chi_1^{-1}\chi_2) \longrightarrow 0,$$

    \vspace{2 mm}

    \noindent which gives us the connecting homomorphism $\delta_0:H^0(\Gamma_{F_\nu},(E/\mathcal O)(\kappa^{2-r}\chi_1^{-1}\chi_2)) \to H^1(\Gamma_{F_\nu},k)$. We note that $\delta_1(L) = 0$ is equivalent to $\im \delta_0 \subseteq H$. Indeed, $\delta_1(L) = 0$ is equivalent to $H^0(\Gamma_{F_\nu},(E/\mathcal O)(\kappa^{2-r}\chi_1^{-1}\chi_2))$ being the annihilator of $\delta_1(L)$, which in turn is equivalent to $\im \delta_0$ being inside $H$, the annihilator of $L$. Let $n \ge 1$ be the largest integer such that $\kappa^{2-r}\chi_1^{-1}\chi_2 \pmod{\varpi^n}$ is trivial. Earlier we have arranged for $\kappa^{2-r}\chi_1^{-1}\chi_2$ to be a non-trivial lift of the trivial character, so such $n \ge 1$ exists. Thus, we can write $\kappa^{2-r}\chi_1^{-1}\chi_2 = 1 + \varpi^n\alpha$ for some function $\alpha:\Gamma_{F_\nu} \to \mathcal O$. The choice of the integer $n$ means that $\overline{\alpha} \coloneqq \alpha \pmod{\varpi}$ is non-zero. Immediately from its definition we note that $\overline{\alpha}:\Gamma_{F_\nu} \to k$ is an additive character. We now twist $\chi_1$ and $\chi_2$ by characters so that $\overline{\alpha}$ is ramified. We remark that we perform this step even if $\overline{\alpha}$ is already ramified. By Lemma \ref{ArbitraryArtinConductor} for any integer $a$ there is finitely wildly ramified character $\psi:\Gamma_{F_\nu} \to (\mathcal O_a)^\times$ which is a lift of the trivial character. We can moreover write $\psi = 1 + \varpi_a^m\psi'$, where $\varpi_a$ is a uniformizer of $\mathcal O_a$, $m = v_{\mathcal O_a}(\zeta_{p^a}-1)$, and $\overline{\psi'}:\Gamma_{F_\nu} \to k$ is a non-trivial and ramified additive character. We claim that by taking $a$ large enough we can assume that $m < n \cdot v_{\mathcal O_a}(\varpi)$. We note that
    
    $$m =  v_{\mathcal O_a}(\zeta_{p^a}-1) = \frac{v_{\mathcal O_a}(p)}{p^{a-1}(p-1)} = \frac{e_{\mathcal O_a/\mathbb Z_p}}{p^{a-1}(p-1)} = \frac{e_{\mathcal O_a/\mathcal O}}{p^{a-1}(p-1)}  \cdot e \le e,$$
    
    \vspace{2 mm}
    
    \noindent where we have used that $\mathcal O_a/\mathcal O$ is an extension of degree at most $p^{a-1}(p-1)$, and recall that $e$ is the ramification index of $\mathcal O$ over $\mathbb Z_p$. This shows that $m$ is uniformly bounded, independent of the integer $a$. This shows that by taking $a$ large enough we can assume that $m \le e < n \cdot v_{\mathcal O_a}(\varpi)$. We now choose such an integer $a$ and fix the character $\psi$ produced by Lemma \ref{ArbitraryArtinConductor} for that integer $a$. We replace $\chi_1$ with $\chi_1\psi^{-1}$ and $\chi_2$ with $\chi_2\psi$. Then

    $$\kappa^{2-r}(\chi_1\psi^{-1})^{-1}(\chi_2\psi) = \kappa^{2-r}\chi_1^{-1}\chi_2\psi^2 = (1 + \varpi^n\alpha)(1 + \varpi_a^m\psi')^2 = 1 + \varpi_a^m(2\psi' + \varpi_a(\cdots)),$$

    \vspace{2 mm}

    \noindent where we used that $m < n \cdot v_{\mathcal O_a}(\varpi)$. in the last equality. Therefore, the $k$-valued additive character associated to $\kappa^{2-r}(\chi_1\psi^{-1})^{-1}(\chi_2\psi)$ will be $2\overline{\psi'}$, which is ramified. In order not to overburden the notation we will continue writing $\chi_1$ and $\chi_2$ after twisting, $\mathcal O$ for $\mathcal O_a$, $\varpi$ for $\varpi_a$, and $\overline{\alpha}$ for the now ramified $k$-valued additive character. 

    We now note that the image of $\delta_0$ is the span of $\overline{\alpha} \in \Hom(\Gamma_{F_\nu},k) = H^1(\Gamma_{F_\nu},k)$. Unpacking the definition of the connecting homomorphism $\delta_0$ and identifying $k$ with $(\varpi^{-1}\mathcal O)/\mathcal O$, for $c \in H^0(\Gamma_{F_\nu},(E/\mathcal O)(\kappa^{2-r}\chi_1^{-1}\chi_2))$ we have that $(\delta_0(c))(\sigma) = \overline{\varpi^nc} \cdot \overline{\alpha}(\sigma)$. We note that as $\kappa^{2-r}\chi_1^{-1}\chi_2$ is non-trivial modulo $\varpi^n$ we must have that $c \in \varpi^{-n}\mathcal O$, hence $\varpi^nc \in \mathcal O$, and it makes sense to talk about its reduction modulo $\varpi$. Using this observation it is enough to show that we can choose $\chi_1$ and $\chi_2$, and consequently $\overline{\alpha}$ such that the latter is contained in $H$. Suppose that this is not the case for the $\chi_1$ and $\chi_2$ we have initially chosen and modified above, as otherwise we are done. We note that then $L \neq 0$ and it has to be a line in $H^1(\Gamma_{F_\nu},k(\overline{\kappa}))$, and $H$ will be a hyperplane in $H^1(\Gamma_{F_\nu},k)$. We now consider $2$ cases:

    \begin{itemize}
        \item Suppose that $\restr{\overline{\rho_1}}{\Gamma_{F_\nu}}$ is tr\'es ramifi\'e. By \cite[Lemma $3.1.$(ii)]{HMR24} this means that $H$ doesn't contain the unramified line in $H^1(\Gamma_{F_\nu},k)$. For $\overline{x} \in k$ we write $u_{\overline{x}}:\Gamma_{F_\nu} \to k$ for the unramified additive character that sends $\sigma_\nu$ to $\overline{x}$. As $H$ is a hyperplane that contains neither $\overline{\alpha}$, nor the unramified line we can find a unique $\overline{x} \neq 0$ such that $\overline{\alpha} + u_{\overline{x}} \in H$. We let $\overline{y} = \overline{x}/2$, and lift $\overline{y}$ to $y \in \mathcal O$. We then replace $\chi_1$ with $\chi_1u_{(1+\varpi^ny)^{-1}}$ and $\chi_2$ with $\chi_2u_{1+\varpi^ny}$, where $u_x:\Gamma_{F_\nu} \to \mathcal O^\times$ is the unramified multiplicative character that sends $\sigma_\nu$ to $x$. Clearly, these characters will still satisfy the conditions stated at the beginning. Now

        $$\kappa^{2-r}(\chi_1u_{(1+\varpi^ny)^{-1}})^{-1}(\chi_2u_{1+\varpi^ny}) = \kappa^{2-r}\chi_1^{-1}\chi_2(u_{1+\varpi^ny})^2 = (1 + \varpi^n\alpha)(u_{1+\varpi^ny})^2.$$

        \vspace{2 mm}

        We note that this character is still trivial modulo $\varpi^{n-1}$. On the other side, as $\overline{\alpha}$ is ramified it will be non-trivial modulo $\varpi^n$. Hence, we can express it as $1+\varpi^n\alpha'$ with $\overline{\alpha'}$ being a non-trivial $k$-valued additive character. Restricting to the inertia subgroup $I_{F_\nu}$ we get that $\overline{\alpha'} = \overline{\alpha} = \overline{\alpha} + u_{\overline{x}}$. Evaluating at $\sigma_\nu$ we get

        $$1+ \varpi^n\alpha'(\sigma_\nu) = (1 + \varpi^n\alpha(\sigma_\nu))(1+\varpi^ny)^2 = 1 + \varpi^n(\alpha(\sigma_\nu) + 2y + \varpi^ny^2 + \varpi^{2n}y^2\alpha(\sigma_\nu)).$$

        \vspace{2 mm}

        Thus, $\overline{\alpha'}(\sigma_\nu) = \overline{\alpha}(\sigma_\nu) + 2\overline{y} = \overline{\alpha}(\sigma_\nu) + \overline{x} = \overline{\alpha}(\sigma_\nu) + u_{\overline{x}}(\sigma_\nu)$. We deduce that $\overline{\alpha'} = \overline{\alpha} + u_{\overline{x}}$, which is an element of $H$ by our choice of $\overline{x}$. 

        \item Suppose that $\restr{\overline{\rho_1}}{\Gamma_{F_\nu}}$ is peu ramifi\'e. Again, as explained in \cite[Lemma $3.1.$(ii)]{HMR24}  we have that $H$ contains the unramified line in $H^1(\Gamma_{F_\nu},k)$. By making a ramified extension of $\mathcal O$, if necessary, we can assume that $n \ge 2$. We now replace $\chi_1$ with $\chi_1u_{(1+\varpi)^{-1}}$ and $\chi_2u_{1+\varpi}$. These new characters will still have the wanted properties. Then

        $$\kappa^{2-r}(\chi_1u_{(1+\varpi)^{-1}})^{-1}(\chi_2u_{1+\varpi}) = \kappa^{2-r}\chi_1^{-1}\chi_2(u_{1+\varpi})^2 = (1+\varpi^n\alpha)(u_{1+\varpi})^2.$$

        \vspace{2 mm}

        As $\kappa^{2-r}\chi_1^{-1}\chi_2$ is trivial modulo $\varpi^2$, while $u_{1+\varpi}$ is not we can express the character as $1+\varpi\alpha'$ with $\overline{\alpha'}$ a non-trivial $k$-valued additive character on $\Gamma_{F_\nu}$. Moreover, $\overline{\alpha'}$ will be unramified, hence contained in $H$.
        
    \end{itemize}

    Summarizing, we can find characters $\chi_1$ and $\chi_2$ satisfying the wanted properties, which are moreover products of tamely ramified characters and $\psi$, for which we get a $p$-adic lift $\rho_{1,\nu}$ of the desired form. We can then take $\chi_1' = [\overline{\kappa}^{1-r}]\psi$ and $\chi_2' = [\overline{\chi}]\psi^{-1}$. As noted above, since $\restr{\overline{\chi}}{\Gamma_{F_\nu}} \neq \overline{\kappa}^{-1}$ we can produce the lift $\rho_{2,\nu}$ using this choice of $\chi_1'$ and $\chi_2'$. By Corollary \ref{TameandWild} we note that $\mathfrak f(\chi_1) = \mathfrak f(\chi_2) = \mathfrak f(\chi_1') = \mathfrak f(\chi_2') = \mathfrak f(\psi) \ge 2$. 

    Finally, we consider the case $\restr{\overline{\chi}}{\Gamma_{F_\nu}} = \overline{\kappa}^{-1}$. We are in the same situation as in the previous case, but with the roles of $\overline{\rho_1}$ and $\overline{\rho_2}$ reversed. Writing $\overline{\rho_2} = \overline{\chi} \otimes \begin{pmatrix} \overline{\chi}^{-1} & * \\ 0 & 1\end{pmatrix}$, we can produce a lift of the second factor as in the previous case, replacing $\overline{\chi}$ with $\overline{\chi}^{-1}$ in the argument. We then twist it by $[\overline{\chi}]$ to produce a lift $\rho_{2,\nu}$. As previously, $\chi_1'$ and $\chi_2'$ will be products of tamely ramified characters and a wildly ramified character $\psi$ produced by Lemma \ref{ArbitraryArtinConductor}. As $\restr{\overline{\chi}}{\Gamma_{F_\nu}} \neq \overline{\kappa}$ we can take $\chi_1' = [\overline{\chi}\overline{\kappa}^{1-r}]\psi$ and $\chi_2 = \psi^{-1}$ and produce a lift $\rho_{1,\nu}$. Again, $\mathfrak f(\chi_1) = \mathfrak f(\chi_2) = \mathfrak f(\chi_1') = \mathfrak f(\chi_2') = \mathfrak f(\psi) \ge 2$ by Corollary \ref{TameandWild}
    
\end{proof}

We take $\rho_{1,\nu}$ and $\rho_{2,\nu}$ to be local $p$-adic lifts produced by the proposition. By enlarging $\mathcal O$, if needed, we can assume that they are all valued in $\GL_2(\mathcal O)$. The lift $\rho_{i,\nu}$ is an ordinary potentially crystalline lift of $\restr{\overline{\rho_i}}{\Gamma_{F_\nu}}$ with Hodge-Tate weight $\{0,r-1\}$ and determinant $\mu$. As the Hodge-Tate weights are distinct these lifts are Hodge-Tate regular and therefore they satisfy the fourth bullet point of Assumptions \ref{GeneralAssumptions}. Moreover, $\rho_{1,\nu}$ and $\rho_{2,\nu}$ have the same Hodge type \textbf{v}. As mentioned above $\rho_{1,\nu}$ and $\rho_{2,\nu}$ might not have the same strong inertial type at primes above $p$. The next result tells us that their associated Weil-Deligne representations will have the same Artin conductor. 

\begin{prop}
\label{WDArtinConductor}

Let $\mathrm{WD}(\rho_{i,\nu})$ be the Weil-Deligne representation associated to the $\rho_{i,\nu}$. Then 

$$\mathfrak f(\mathrm{WD}(\rho_{1,\nu})) = \mathfrak f(\mathrm{WD}(\rho_{2,\nu})).$$
    
\end{prop}
\begin{proof}

    As both lifts $\rho_{1,\nu}$ and $\rho_{2,\nu}$ are potentially crystalline the monodromy operator $N$ in the respective associated Weil-Deligne representations will be $0$. Therefore, we have $\mathrm{WD}(\rho_{i,\nu}) = (r_{i,\nu},0)$ for $i=1,2$, where $r_{i,\nu}: W_{F_\nu} \to \GL_2(\overline{E})$ is a representation whose restriction to $I_{F_\nu}$ has open kernel.
    
    For the lift $\rho_{1,\nu}$ we have an exact sequence $0 \to E(\kappa^{r-1}\chi_1) \to V_{\rho_{1,\nu}} \to E(\chi_2) \to 0$ for some finitely ramified characters $\chi_1$ and $\chi_2$. The functor $D_{\mathrm{st},L}$ doesn't depend on the field $L$ as long as $\restr{\rho_{i,\nu}}{\Gamma_L}$ is a semi-stable representation. In particular, we take $L$ such that the restrictions $\restr{\chi_1}{\Gamma_L}$ and $\restr{\chi_2}{\Gamma_L}$ are unramified. The functor $D_{\mathrm{st},L}$ is exact in the category of semi-stable representations, so as $\restr{\rho_{1,\nu}}{\Gamma_L}$ is semi-stable we get a short exact sequence

    $$0 \longrightarrow D_{\mathrm{st},L}(E(\kappa^{r-1}\chi_1)) \longrightarrow D_{\mathrm{st},L}(V_{\rho_{1,\nu}}) \longrightarrow D_{\mathrm{st},L}(E(\chi_2)) \longrightarrow 0.$$

    \vspace{2 mm}

    We note that $D_{\mathrm{st},L}(\bullet)$ will be a free $(L_0 \otimes_{\mathbb Q_p} E)$-module, where $L_0$ is the maximal unramified subextension in $L/\mathbb Q_p$. Moreover, the $\Gal(L/F_\nu)$-action on $D_{\mathrm{st},L}(\bullet)$ will be $E$-linear, but $L_0$-semilinear. Restricting to the inertia subgroup of $\Gal(L/F_\nu)$ we equip $D_{\mathrm{st},L}(\bullet)$ with an $I_{F_\nu}$-action by letting $I_L$ act trivially. As $L_0/\mathbb Q_p$ is unramified this action will be trivial on $L_0$, and hence $(L_0 \otimes_{\mathbb Q_p} E)$-linear. Taking larger $E$, if needed, we can assume that $E$ contains all the embedding of $L_0$ into $\overline{E}$. Thus, from $L_0 \otimes_{\mathbb Q_p} E \simeq \prod_{\alpha: L_0 \hookrightarrow E} E$ we get a decomposition $D_{\mathrm{st},L}(\bullet) = \prod_{\alpha:L_0 \hookrightarrow E} D_\alpha(\bullet)$. Then, the $E[I_{F_\nu}]$-module $D_\alpha(\bullet)$, which by \cite[Lemma $2.2.1.2$]{BM02} is independent of the choice of $\alpha$, is exactly $\restr{r}{I_{F_\nu}}$, where $r:W_{F_\nu} \to \GL_2(\overline{E})$ is the representation appearing in $\mathrm{WD}(\bullet)$.

    As the $I_{F_\nu}$-action factors through a finite quotient $I_{F_\nu}/I_L$, by Maschke's theorem the above sequence splits when restricted to $I_{F_\nu}$. Therefore, $D_{\alpha}(V_{\rho_{1,\nu}}) \simeq D_\alpha(E(\kappa^{r-1}\chi_1)) \oplus D_\alpha(E(\chi_2))$, as $E[I_{F_\nu}]$-modules. We will now describe these $1$-dimensional spaces. We have $D_{\mathrm{st},L}(E(\kappa^{r-1}\chi_1)) = (B_{\mathrm{st}} \otimes_{\mathbb Q_p} E(\kappa^{r-1}\chi_1))^{\Gamma_L} \simeq L_0t^{1-r} \otimes_{\mathbb Q_p} E$, where $t$ is a uniformizer of $B_{\dR}^+$ on which $\Gamma_L$ acts as the cyclotomic character. This action cancels out the action of $\kappa^{r-1}$ on the second factor. Hence, the $\Gal(L/F_\nu)$-action on it will be given by multiplication by $\chi_1$. Therefore $D_{\mathrm{st},L}(E(\kappa^{r-1}\chi_1))$ as a $(L_0 \otimes_{\mathbb Q_p} E)$-module endowed with a $I_{F_\nu}$-action will be isomorphic to $L_0 \otimes_{\mathbb Q_p} E(\chi_1)$. In particular, $D_{\alpha}(E(\kappa^{r-1}\chi_1)) \simeq E(\chi_1)$. A similar computation tells us that $D_{\alpha}(E(\chi_2)) \simeq E(\chi_2)$. Thus, $D_\alpha(V_{\rho_{1,\nu}}) \simeq E(\chi_1) \oplus E(\chi_2)$, and $\restr{r_{1,\nu}}{I_{F_\nu}} = \chi_1 \oplus \chi_2$. Analogously, $\restr{r_{2,\nu}}{I_{F_\nu}} = \chi_1' \oplus \chi_2'$. 

    Following \cite[\S $8.12$]{Del73}, the Artin conductor $\mathfrak f(\mathrm{WD}(\rho_{i,\nu}))$ of the Weil-Deligne representation $\mathrm{WD}(\rho_{i,\nu})$ is given by $\mathfrak f(r_{i,\nu}) + \dim V_{r_{i,\nu}}^{I_{F_\nu}} - \dim V_{r_{i,\nu},N}^{I_{F_\nu}}$, where $V_{r_{i,\nu},N}$ is the subspace of $V_{r_{i,\nu}}$ on which $N$ acts as the zero map. In our case $N=0$ so the last two terms cancel out. Therefore

    $$\mathfrak f(\mathrm{WD}(\rho_{1,\nu})) = \mathfrak f(r_{1,\nu}) =  \sum_{i=0}^\infty \frac{\mathrm{codim } \, V_{r_{1,\nu}}^{G_i}}{[G_0:G_i]} = \sum_{i=0}^\infty \frac{\mathrm{codim } \, E(\chi_1)^{G_i} + \mathrm{codim } \, E(\chi_2)^{G_i}}{[G_0:G_i]} = \mathfrak f(\chi_1) + \mathfrak f(\chi_2).$$

    \vspace{2 mm}

    Similarly, $\mathfrak f(\mathrm{WD}(\rho_{2,\nu})) = \mathfrak f(\chi_1') + \mathfrak f(\chi_2')$. As $\chi_1$ and $\chi_1'$, and $\chi_2$ and $\chi_2'$ have the same Artin conductors we deduce that $\mathfrak f(\mathrm{WD}(\rho_{1,\nu})) = \mathfrak f(\mathrm{WD}(\rho_{2,\nu}))$.
 
\end{proof}

\begin{rem}
\label{CrystallineLifts}

    When $F = \mathbb Q$ we can write $\overline{\chi} = \overline{\kappa}^{r-1}\overline{\chi_0}$ for some unramified character $\overline{\chi_0}$. With this choice of $r$ by using Proposition $\ref{LocalLiftsmidp}$ we can produce a crystalline lift of $\restr{\overline{\rho_1}}{\Gamma_{\mathbb Q_p}}$ by simply taking $\chi_1 = [\overline{\chi_0}]$ and $\chi_2 = 1$, avoiding the need to involve the ramified character $\psi$. However, in general Proposition \ref{LocalLiftsmidp} doesn't simultaneously produce a crystalline lift for $\restr{\overline{\rho_2}}{\Gamma_{\mathbb Q_p}}$ for this choice of $r$. Our inability to produce crystalline lifts of $\restr{\overline{\rho_1}}{\Gamma_{\mathbb Q_p}}$ and $\restr{\overline{\rho_2}}{\Gamma_{\mathbb Q_p}}$ with the same Hodge-Tate weights means that the common level $N$ of the modular forms associated to $\overline{\rho_1}$ and $\overline{\rho_2}$ will be divisible by $p$.

    On the other hand, if we consider the case when both $\overline{\rho_1}$ and $\overline{\rho_2}$ have the same shape, as in Remark \ref{SameShape}, then for this choice of $r$ we can produce crystalline lifts of $\restr{\overline{\rho_1}}{\Gamma_{\mathbb Q_p}}$ and $\restr{\overline{\rho_2}}{\Gamma_{\mathbb Q_p}}$ with the same Hodge-Tate weights. This means that we can produce local lifts at $p$ which will not only have the same Artin conductor, but the same strong inertial type, which in this case turns out to be the trivial one. Therefore, $\overline{\rho_1}$ and $\overline{\rho_2}$ will be modular of the same weight $r$, the same Neben character $[\overline{\chi_0}]$, and the same level $N$, which moreover will not be divisible by $p$.
    
\end{rem}

\subsection{Main Result}

The representations $\overline{\rho_1},\overline{\rho_2}:\Gamma_{F,S} \to \GL_2(k)$ satisfy Assumptions \ref{GeneralAssumptions}. Thus, we can produce $p$-adic lifts $\rho_1$ and $\rho_2$, which are lifts of $\overline{\rho_1}$ and $\overline{\rho_2}$, respectively, as in Theorem \ref{padicLift}. When $F = \mathbb Q$ by the results of Skinner-Wiles (\cite[Theorem A]{SW99}) and Lue Pan (\cite[Theorem $7.1.1$]{Pan18}) we know that each of these lifts will be modular, i.e. will be isomorphic to a representation attached to some modular form. Our goal is to show that $\rho_1$ and $\rho_2$ are associated to newforms of the same level and the same weight. The level of the newforms will be determined by the Artin conductor of the Weil-Deligne representations $\mathrm{WD}(\restr{\rho_i}{\Gamma_{\mathbb Q_\ell}})$ for each prime $\ell$ of $\mathbb Q$, while the weight by the Hodge-Tate weights of $\restr{\rho_i}{\mathbb Q_p}$. We now analyze how each of these parameters is affected by the lifting method. We remark that even though our motivation comes from modularity results over $\mathbb Q$ most of the results can be proven for totally real number fields $F$, where in the absence of modularity results we can interpret them in terms of the local behavior of the lifts $\rho_1$ and $\rho_2$. We prove them in full generality. We will only need to restrict to the case $F = \mathbb Q$ when we invoke Ramanujan's conjecture to control the behavior of the lifts at the primes where we allow extra ramification during the lifting process.

We recall that the lifting method consisted of two parts. The doubling method produces lifts of $\overline{\rho_1}$ and $\overline{\rho_2}$ modulo high enough powers of $\varpi$. As seen in Theorem \ref{LiftingModn} it allows us to precisely describe the local behavior of the lifts $\rho_{1,\nu}$ and $\rho_{2,\nu}$ at each prime in a finite set $S'$. On the other side, the relative deformation method produces $p$-adic lifts $\rho_1$ and $\rho_2$ by allowing extra ramification at some finite set of primes $Q$. Unlike in the doubling method we can't exactly prescribe the local behavior of these lifts. Instead, for each prime $\nu \in S' \cup Q$ and and for all large enough integers $n$ we define a class of lifts $D_{i,n,\nu} \subseteq \Lift_{\restr{\overline{\rho_i}}{\Gamma_{F_\nu}}}(\mathcal O/\varpi^n)$. For these classes of lifts the fibers $D_{i,n+4m,\nu} \to D_{i,n,\nu}$ are stable under the preimages in $Z^1(\Gamma_{F_\nu},\rho_{i,4m}(\mathfrak g^{\der}))$ of the local conditions $L_{i,4m,\nu}$. We eventually have that $\restr{\rho_i}{\Gamma_{F_\nu}} \in D_{i,\nu} \coloneqq \varprojlim D_{i,n,\nu} \subseteq \Lift_{\restr{\overline{\rho_i}}{\Gamma_{F_\nu}}}(\mathcal O)$. Our goal is to show that the local behavior of lifts inside $D_{i,\nu}$ for each $\nu \in S' \cup Q$ and $i=1,2$ is uniform. The classes $D_{i,n,\nu}$ are defined in multitude of ways, mainly by \cite[Lemma $3.5$]{FKP21} and \cite[Proposition $4.7$]{FKP21}. Therefore, we will need to prove the claim on a case-by-case basis. 

\begin{itemize}

\item We first note that by construction, $\rho_1$ and $\rho_2$ will be unramified away from $S' \cup Q$. Therefore, at each such prime the Weil-Deligne representations $\mathrm{WD}(\restr{\rho_i}{\Gamma_{F_\nu}})$ will be unramified with monodromy operator $N=0$. Hence, $\mathfrak f(\mathrm{WD}(\restr{\rho_i}{\Gamma_{F_\nu}})) = 0$ for both $i=1,2$. 

\item We now focus on primes in $S'$. We remember that before we even started the lifting method we have enlarged our set $S$ for technical reasons. Also, during the production of the modulo $\varpi^n$ lifts $\rho_{1,n}$ and $\rho_{2,n}$ in Theorem \ref{LiftingModn} we have added some extra primes to $S'$ in order to make sure that the lifts will have maximal image. At each of these primes $\nu$ we take $D_{i,n,\nu}$ to be the class of unramified lifts. Thus, $\restr{\rho_i}{\Gamma_{F_\nu}}$ will be unramified and as above the Weil-Deligne representations $\mathrm{WD}(\restr{\rho_i}{\Gamma_{F_\nu}})$ will be unramified with monodromy operator $N=0$. Therefore, $\mathfrak f(\mathrm{WD}(\restr{\rho_i}{\Gamma_{F_\nu}})) = 0$ for such primes $\nu$. 

\item For primes $\nu$ in the original set $S$ we have produced local lifts $\rho_{i,\nu}$ as in \S $5.2$. We recall that the lifts $\rho_{1,\nu}$ and $\rho_{2,\nu}$ are chosen so that if $\nu \nmid p$ by Corollary \ref{LocalLiftsFixedDeterminant} they have the same strong inertial type; and if $\nu \mid p$ by Proposition \ref{WDArtinConductor} the Weil-Deligne representations $\mathrm{WD}(\rho_{1,\nu})$ and $\mathrm{WD}(\rho_{2,\nu})$ have the same Artin conductor. Then, as explained at the beginning of \S $4$ we choose an irreducible component of $R_{\restr{\rho_i}{\Gamma_{F_\nu}}}^{\sqr,\mu}[1/\varpi]$ if $\nu \nmid p$, or of $R_{\restr{\rho_i}{\Gamma_{F_\nu}}}^{\sqr,\mu,\tau,\textup{\textbf{v}}}[1/\varpi]$ if $\nu \mid p$ containing $\rho_{i,\nu}$. We choose the class of lifts $D_{i,n,\nu}$ as in \cite[Proposition $4.7$]{FKP21}. $D_{i,\nu}$ will consist of lifts of $\restr{\overline{\rho_i}}{\Gamma_{F_\nu}}$ that correspond to points inside an $\varpi$-adic open neighborhood of $\rho_{i,\nu}$ in the chosen irreducible component. We claim that these can be chosen in such a manner that lifts in $D_{i,\nu}$ will have the same strong inertial type. The inertial type (this is just the restriction $\restr{r}{I_{F_\nu}}$, without the monodromy operator N) is constant on irreducible components of the generic fiber (see \cite[Theorem $3.3.1$]{Boc13}). Therefore, it remains to show that we can choose the monodromy operator $N$ to be constant for lifts in $D_{i,\nu}$. 

\begin{itemize}

\item If $N = 0$ in $\mathrm{WD}(\rho_{i,\nu})$ we can work with $R_{\restr{\rho_i}{\Gamma_{F_\nu}}}^{\sqr,\mu,N=0}[1/\varpi]$ ($R_{\restr{\rho_i}{\Gamma_{F_\nu}}}^{\sqr,\mu,\tau,\textup{\textbf{v}},N=0}[1/\varpi]$ if $\nu \mid p$) instead of with $R_{\restr{\rho_i}{\Gamma_{F_\nu}}}^{\sqr,\mu}[1/\varpi]$ ($R_{\restr{\rho_i}{\Gamma_{F_\nu}}}^{\sqr,\mu,\tau,\textup{\textbf{v}}}[1/\varpi]$ if $\nu \mid p)$. These are exactly the union of irreducible components of the generic fibers for which the monodromy operator $N$ vanishes identically. Points in these rings correspond to potentially unramified lifts if $\nu \nmid p$, and potentially crystalline lifts if $\nu \mid p$. We then use \cite[Theorem $3.3.8$]{BG19} instead of \cite[Theorem $3.3.3$]{BG19} in the dimension computations of \cite[Proposition $4.7$]{FKP21}. Thus, we produce classes of lifts $D_{i,\nu}$ for which the monodromy operator $N$ of the associated Weil-Deligne representations will be $0$. Hence, all lifts in $D_{i,\nu}$ have the same strong inertial type. 

\item Now, assume that $N \neq 0$ in $\WD(\rho_{i,\nu})$. We note that in this case $\nu \nmid p$, since by Proposition \ref{LocalLiftsmidp} the lifts $\rho_{i,\nu}$ for $\nu \mid p$ are potentially crystalline. As $N$ is a $2 \times 2$ nilpotent matrix it has to be a conjugate of $\begin{pmatrix} 0 & 1 \\ 0 & 0\end{pmatrix}$. Therefore, it suffices to show that we can choose $D_{i,\nu}$ such that for each lift inside this class the monodromy operator of its associated Weil-Deligne representation doesn't vanish. The next lemma shows us that this is indeed possible. 

\begin{lem}
\label{NonzeroMonodromy}

    Let $\nu \nmid p$. Suppose that $\WD(\rho) = (r,N)$ with $N\neq0$ for some $\rho:\Gamma_{F_\nu} \to \GL_n(\mathcal O)$. Then, there exists a positive integer $m$, depending only on $n$ and $\mathcal O$ such that: if for any $\rho':\Gamma_{F_\nu} \to \GL_n(\mathcal O)$ we have $\rho \equiv \rho' \pmod{\varpi^m}$, then $N' \neq 0$, where $\WD(\rho') = (r',N')$.
    
\end{lem}
\begin{proof}

We will first show that we can find a uniform upper bound on the order of the finite-order elements of $\GL_n(\mathcal O)$. Any finite-order element of $\GL_n(\mathcal O)$ will be diagonalizable when viewed as an element of $\GL_n(\overline{E})$. In particular, its eigenvalues will be roots of unity. On the other side, these roots of unity will satisfy a degree $n$ polynomial over $\mathcal O$. As $\mathcal O$ is a finite extension of $\mathbb Z_p$ it contains finitely many roots of unity, so we conclude that there are finitely many roots of unity that satisfy a degree $n$ polynomial over $\mathcal O$. Therefore, there are finitely many options for the eigenvalues of finite-order elements of $\GL_n(\mathcal O)$, which gives us an upper bound on the order of those elements. Label this bound by $M$.

By Grothendieck's monodromy theorem $\restr{\rho}{H}$ will be unipotent for some finite-index subgroup $H$ of $I_{F_\nu}$. In particular $\rho(\tau_\nu^a)$ will be a unipotent matrix for some positive integer $a$. We can use $\rho(\tau_\nu^a)$ to compute $N$ in $\WD(\rho)$ and by our assumption we have that $\rho(\tau_\nu^a)$ is a non-trivial unipotent matrix. The same will be true for $\rho(\tau_\nu^a),\rho(\tau_\nu^{2a}),\dots,\rho(\tau_\nu^{Ma})$. Therefore, we can find a large enough integer $m$ such that $\rho(\tau^{ia}) \not \equiv 1 \pmod{\varpi^m}$ for all $1 \le i \le M$. We claim that this value of $m$ would satisfy the statement of the lemma. Let $\rho':\Gamma_{F_\nu} \to \GL_n(\mathcal O)$ be a representation such that $\rho \equiv \rho' \pmod{\varpi^m}$. As above we can find a positive integer $b$ such that $\rho'(\tau_\nu^b)$ is unipotent, which we can use to compute $N'$. We claim it can't be trivial. For the sake of contradiction, suppose it is trivial. Then $\rho'(\tau_\nu)$ will be a finite-order element of $\GL_n(\mathcal O)$, and so is $\rho'(\tau_\nu^a)$. By what we have shown above it has order $\le M$, so $\rho'(\tau_\nu^{ia}) = 1$ for some $i \le M$. But then $\rho(\tau_\nu^{ia}) \equiv \rho'(\tau_\nu^{ia}) \equiv 1 \pmod{\varpi^m}$, which contradicts our choice of $m$. Therefore, $\rho'(\tau_\nu^b)$ is a non-trivial unipotent matrix, hence $N' \neq 0$.
    
\end{proof}

For each $\nu \in S$ such that $N \neq 0$ in $\WD(\rho_{i,\nu})$ we feed the integer $m_\nu$ coming from the lemma into the lower bounds on $n$ in the lifting method. In particular, we guarantee that for each such $m_\nu$ we will produce the lifts $\rho_{i,m_\nu}:\Gamma_{F,T_{m_\nu}} \to \GL_2(\mathcal O/\varpi^{m_\nu})$ by the doubling method. Therefore, each lift inside $D_{i,\nu}$ will have the same reduction modulo $\varpi^{m_\nu}$ as the lift $\rho_{i,\nu}$, which was produced beforehand. As in our discussion before the lemma this means that every lift in $D_{i,\nu}$ will have the same strong inertial type. 

\end{itemize}

\noindent As the Artin conductor of the Weil-Deligne representation depends only on the strong inertial type, from our choice of the local lifts in \S $5.2$ we deduce that for all $\nu \in S$

$$\mathfrak f(\WD(\restr{\rho_1}{\Gamma_{F_\nu}})) = \mathfrak f(\WD(\rho_{1,\nu})) = \mathfrak f(\WD(\rho_{2,\nu})) = \mathfrak f(\WD(\restr{\rho_2}{\Gamma_{F_\nu}})).$$

\vspace{2 mm}

\begin{rem}
\label{LocallySmoothPoints}

In order to apply \cite[Proposition $4.7$]{FKP21} we need to make sure that the local lifts we produced in \S $5.2$ correspond to formally smooth points in the generic fiber of the suitable lifting ring. If they are not, then as explained during the production of the local conditions $L_{i,n,\nu}$ at the beginning of \S $4$, we replace them with formally smooth lifts which are equal to the initial lifts modulo any prescribed power of $\varpi$. We need to be careful not to change the strong inertial type of the local lifts during this process. We can do this as the formally smooth points are dense in the generic fibers by \cite[Lemma $4.9$]{FKP21}. In particular, we can choose this formally smooth point to be in the same irreducible component as the given lift, so the inertial type doesn't change. We can guarantee that the monodromy operator $N$ stays the same by arguing as above. In particular, if $N=0$ in the inertial type of the given lift, then we can choose this smooth point to lie in an irreducible component on which the monodromy operator vanishes identically. If $N \neq 0$ in the inertial type of the given lift, by Lemma \ref{NonzeroMonodromy} if we choose a formally smooth point close enough to the given lift we preserve the non-vanishing of the monodromy operator. 
    
\end{rem}

\item We will now show the same equality for the rest of the primes in $S'$, i.e. the primes that were added in $S'$ during the doubling method and at which both $\rho_1$ and $\rho_2$ ramify. We first consider the primes that satisfy part $(1a)$ of Theorem \ref{LiftingModn}. For such a prime $\nu$ we have a $p$-adic local lift $\rho_{i,\nu}$ which is a formally smooth point in the generic fiber and we produce the classes of lifts $D_{i,\nu}$ by \cite[Proposition $4.7$]{FKP21}. Arguing as in the case of primes in $S$ it is enough to show that $\rho_{1,\nu}$ and $\rho_{2,\nu}$ have the same strong inertial type. We have that $\rho_{i,n} \in \Lift_{\restr{\overline{\rho_i}}{\Gamma_{F_\nu}}}^{\mu,\alpha_{i,\nu}}(\mathcal O)$ and in particular $\rho_{i,\nu}(\tau_\nu)$ is a non-trivial unipotent element in $U_{\alpha_{i,\nu}}(\mathcal O)$. We use it to compute the monodromy operator $N_i$ of $\WD(\rho_{i,\nu}) = (r_i,N_i)$. We immediately deduce that $N_i \sim \begin{pmatrix} 0 & 1 \\ 0 & 0\end{pmatrix}$ and $\restr{r_i}{I_{F_\nu}}$ is trivial. We conclude that $\rho_{1,\nu}$ and $\rho_{2,\nu}$ will have the same strong inertial type, which will be equal to $\tau_{1,ns}$. We get 

$$\mathfrak f(\WD(\restr{\rho_1}{\Gamma_{F_\nu}}))= \mathfrak f(\WD(\restr{\rho_2}{\Gamma_{F_\nu}})) = 1.$$

\vspace{2 mm}

\item The rest of the primes in $S'$ are the ones that satisfy part $(1b)$ of Theorem \ref{LiftingModn}. For such primes $\nu$ we define the classes of lifts $D_{i,\nu,n}$ by hand, in particular $D_{i,\nu,n} = \Lift_{\restr{\overline{\rho_i}}{\Gamma_{F_\nu}}}^{\mu,\alpha_{i,\nu}}(\mathcal O/\varpi^n)$. Therefore, $D_{i,\nu} = \Lift_{\restr{\overline{\rho_i}}{\Gamma_{F_\nu}}}^{\mu,\alpha_{i,\nu}}(\mathcal O)$, and as seen in the previous paragraph the strong inertial types for lifts in $D_{i,\nu}$ are $\tau_{1,ns}$ and $\tau_{1,s}$, depending on whether the lift ramifies or not. From Theorem \ref{LiftingModn} we know if $\nu$ was added during the $n$-th step of the doubling method, then $\rho_{i,n}$ would ramify at $\nu$. Hence, $\rho_i$ ramifies at $\nu$ and its strong inertial type is $\tau_{1,ns}$. As above $\mathfrak f(\WD(\restr{\rho_1}{\Gamma_{F_\nu}}))= \mathfrak f(\WD(\restr{\rho_2}{\Gamma_{F_\nu}})) = 1$.

\item Finally, we consider the primes $\nu \in Q$. These are the primes that were added with purpose of annihilating some relative dual Selmer groups. These primes were selected from the auxiliary sets $Q_{n,m}$. For such primes we set $D_{i,\nu,n} = \Lift_{\restr{\overline{\rho_i}}{\Gamma_{F_\nu}}}^{\mu,\alpha_{i,
\nu}}(\mathcal O/\varpi^n)$, and subsequently $D_{i,\nu} = \Lift_{\restr{\overline{\rho_i}}{\Gamma_{F_\nu}}}^{\mu,\alpha_{i,
\nu}}(\mathcal O)$. As before, the possible strong inertial types for lifts in these sets are $\tau_{1,ns}$ and $\tau_{1,s}$. In order to specify this strong inertial types more precisely from now on we restrict ourselves to the case $F = \mathbb Q$, and we write $\ell$, instead of $\nu$ for its finite places. We claim that in this case our lifts $\rho_1$ and $\rho_2$ are ramified at primes in $Q$. This follows from Ramanujan's conjecture, which we know is true for automorphic representations coming from cuspidal modular forms. In that case it will be a consequence of the Weil conjectures, which were proven by Deligne (\cite{Del74}). By \cite[Theorem A]{SW99} and \cite[Theorem $7.1.1$]{Pan18} we know that the lifts $\rho_i$ are isomorphic to a representation attached to some modular form. Therefore, $\rho_i \simeq \rho_{\pi_i,\iota_i}$, where $\pi_i$ is a cuspidal automorphic representation on $\GL_2(\mathbb A_{\mathbb Q})$ coming from a holomorphic modular form, and $\iota_i:\overline{\mathbb Q_p} \simeq \mathbb C$ is some field isomorphism. If $\rho_i$ is not ramified at some prime $\ell$ by the local Langlands correspondence the local component $\pi_{i,\ell}$ will be an unramified principal series $\mathrm{n-Ind}_{B(\mathbb Q_\ell)}^{\GL_2(\mathbb Q_\ell)}(\chi_{1,i,\ell} \otimes \chi_{2,i,\ell})$ for some unramified characters $\chi_{1,i,\ell}$ and $\chi_{2,i,\ell}$ on $\mathbb Q_\ell^\times$, where $\mathrm{n-Ind}$ stands for the normalized induction. In particular, $\mathrm{rec}_\ell(\pi_{i,\ell}) = \mathrm{rec}_\ell(\chi_{1,i,\ell}) \oplus \mathrm{rec}_\ell(\chi_{2,i,\ell})$. By Ramanujan's conjecture $\pi_{i,\ell}$ is tempered, which means that the characters $\chi_{1,i,\ell}$ and $\chi_{2,i,\ell}$ have the same absolute value. In particular, $\chi_{1,i,\ell}(\ell)$ and $\chi_{2,i,\ell}(\ell)$, the images of the eigenvalues of $\rho_i(\sigma_\ell)$ under $\iota_i$ will have the same absolute value. This is impossible for $\ell \in Q$. Indeed, for $\ell \in Q$ we have $\restr{\rho_i}{\Gamma_{\mathbb Q_\ell}} \in D_{i,\ell}$, and by construction the image of $\sigma_\ell$ under elements in $D_{i,\ell}$ has eigenvalues $\ell$ and $1$, and their image under $\iota_i$ do not have the same absolute value. Therefore, both $\rho_1$ and $\rho_2$ are ramified at all primes in $Q$ and moreover they both have inertial types $\tau_{1,ns}$ at primes $\ell \in Q$. 

\end{itemize}

We recall that at the beginning of \S 3 we have assume that $p$ is a large enough number with respect to the root data of $G_1$ and $G_2$, so that the the statement of all the lemmas in Appendix $\hyperref[Appendix]{A}$ hold. We now explicitly compute the lower bounds on $p$ in the special case $G_1 = G_2 = \GL_2$.

\begin{lem}
\label{BoundsVerification}

For $G = GL_2$ all the results in Appendix $\hyperref[Appendix]{A}$ are satisfied when $p \ge 5$.

\end{lem}
\begin{proof}

We will go through each of the lemmas in Appendix $\hyperref[Appendix]{A}$ that require $p$ to be a large integer. By focusing on the part of the proof that requires this we will show that all of them are satisfied for $G=\GL_2$ when $p \ge 5$.

\begin{itemize}

	\item In the proof of Proposition \ref{SelectionofPrimes} we applied Lemma \ref{NonInclusionLemma} with $n=4$. To be able to do this we need all of $\{1,2,3,4\}$ to be invertible modulo $\varpi$. As mentioned during the proof of Proposition \ref{SelectionofPrimes} this will be satisfied for $p \ge 5$.
	\item In Theorem \ref{LiftingModn} we used Lemma \ref{TorusElementLemma} in order to produce suitable elements inside a maximal split torus. In the proof of Lemma \ref{TorusElementLemma} we needed $p$ to be large enough so that a subspace of the torus is not a union of hyperplanes. However, as $\GL_2$ has a single positive root this is the empty union. In fact, the claim about the rest of the roots $\beta$ holds vacuously. Therefore, we only need $2$ to be invertible modulo $\varpi$, which is true for $p \ge 3$.
	\item In Proposition \ref{SelectionofPrimes} and Proposition \ref{1-0Annihilation} we needed to find a suitable maximal split torus $T$ and a root $\alpha$. We did that by invoking Lemma \ref{TorusandRootLemma} with $r=s=1$. When $G^{\der} = \SL_2$ we can make our initial choice $T_1$ to be the diagonal maximal torus, and $\alpha_1$ the positive root that corresponds to the upper-triangular Borel subgroup. As $p \ge 3$ is very good for $\SL_2$ the trace pairing is perfect, so we can use it to identify $\mathfrak{sl}_2$ with $(\mathfrak{sl}_2)^*$. Then, for a choice of non-zero elements $A,B \in \mathfrak{sl}_2$ we want to find $g \in \SL_2(k)$ so that 
        
        $$\Ad(g^{-1})A \notin \begin{pmatrix} 0 & * \\ * & 0\end{pmatrix} \quad \quad \text{and} \quad \quad \Ad(g^{-1})B \notin \begin{pmatrix} * & * \\ 0 & *\end{pmatrix}.$$

        \vspace{2 mm}

        As the adjoint action is given by conjugation a quick computation tells us that the elements $g \in \SL_2(k)$ we want to avoid are cut out by two degree $2$ polynomials in $3$ variables. Therefore, as explained in the proof of Lemma \ref{TorusandRootLemma} there are at most $4q^2$ such elements. On the other side, $\SL_2(k)$ has order $q^3-q$. As $q^3-4q^2-q$ is positive for $q \ge 5$ we conclude that when $p \ge 5$ we can always find $g \in \SL_2(k)$ as wanted.
        
        \item In the proofs of Proposition \ref{cokernelgenerators} and Proposition \ref{DoublingMethod} we used Lemma \ref{RootSpaceSpan} to conclude that $\mathfrak{sl}_2$ is equal to the $k$-span of its root vectors. In Lemma \ref{RootSpaceSpan} we assumed $p \gg_G 0$ so that the simple factors of $\mathfrak g^{\der}$ are irreducible $k[G(k)]$-modules. We will show that $\mathfrak{sl}_2$ is an irreducible $\mathbb F_p[\SL_2(k)]$-module for $p \ge 3$. Let $\{E,F,H\}$ be the standard basis of $\mathfrak{sl}_2$. Then, conjugating by matrices of the form $(1+aE)$ and $(1+aF)$ for some $a\in k$, it easy to see that any non-zero $\mathbb F_p[\SL_2(k)]$-submodule of $\mathfrak{sl}_2$ must contain the basis elements $\{E,F,H\}$, hence it is equal to the whole of $\mathfrak{sl}_2$.
        
        \item During the proofs of Lemma \ref{LinearDisjointness}, Lemma \ref{VanishingofH1FN*}, and Proposition \ref{1-0Annihilation} we used Lemma \ref{AbelianizationLemma}, for which we need $p \gg_G 0$ so that the simple factors of $\mathfrak g^{\der}$ are irreducible $\mathbb F_p[G(k)]$-modules. As checked in the previous bullet point for $G=\SL_2$ this is true for $p \ge 3$.

\end{itemize}

\end{proof}

Summarizing we have the following result

\begin{thm}
\label{MainResult}

Assume $p \ge 5$. Let $\overline{\rho_1},\overline{\rho_2}:\Gamma_{\mathbb Q,S} \to \GL_2(k)$ be two non-semi-simple residual representations such that

    $$\overline{\rho_1} = \begin{pmatrix} \overline{\chi} & * \\ 0 & 1\end{pmatrix} \quad \quad \mathrm{and} \quad \quad \overline{\rho_2} = \begin{pmatrix} 1 & *' \\ 0 & \overline{\chi}\end{pmatrix}$$
    
    \vspace{2 mm}
    
\noindent for some odd character $\overline{\chi}$, which moreover is different than $\overline{\kappa}^{\pm 1}$. Then $\overline{\rho_1}$ and $\overline{\rho_2}$ are modular of the same weight $r \ge 2$, the same level $N \ge 1$, and the same Neben character $[\overline{\chi}\overline{\kappa}^{1-r}]$.

\end{thm}
\begin{proof}

	We recall that at the beginning of \S $3$ we assumed that $p$ is large enough so that all results of Appendix $\hyperref[Appendix]{A}$ hold. As verified in Lemma \ref{BoundsVerification} when $G_1=G_2 = \GL_2$ these are all satisfied for $p \ge 5$. In addition to this we assume that $p$ is very good prime for $G_1^{\der}$ and $G_2^{\der}$. As $\SL_2$ has Dynkin diagram $A_1$ the only not very good prime for it is $2$. Furthermore, as checked in \S $5.1$ and \S $5.2$ Assumptions \ref{GeneralAssumptions} will be satisfied for $\overline{\rho_1}$ and $\overline{\rho_2}$ when $p \ge 3$.

	This allow us to produce $p$-adic lifts $\rho_1$ and $\rho_2$ of $\overline{\rho_1}$ and $\overline{\rho_2}$, respectively, such that from the discussion preceding the theorem, the Weil-Deligne representations $\WD(\restr{\rho_1}{\Gamma_{\mathbb Q_\ell}})$ and $\WD(\restr{\rho_2}{\Gamma_{\mathbb Q_\ell}})$ have the same Artin conductor for each $\ell \in \mathbb Q$. In fact, we have shown that for $\ell \neq p$ the restrictions of these Weil-Deligne representations to the inertia groups $I_{\mathbb Q_\ell}$ are isomorphic. By \cite[Theorem A]{SW99} and \cite[Theorem $7.1.1$]{Pan18} the $p$-adic lift $\rho_i$ will be associated to a newform $f_i$ of some level $N_i$, some weight $r_i$, and some Neben character $\varepsilon_i$. To $f_i$ we can attach an irreducible automorphic $\GL_2(\mathbb A_{\mathbb Q})$-representation $\pi_i$, which will have conductor $N_i$. Now, by the local Langlands correspondence for each prime $\ell$, the local component $\pi_{i,\ell}$ will correspond to a $2$-dimensional Frobenius semi-simple Weil-Deligne representation of $W_{\mathbb Q_\ell}$. By the naturality of the local Langlands correspondence (see \cite[\S $3$]{Tay04}) this Weil-Deligne representation will have the same conductor as $\pi_{i,\ell}$, which in turn is equal to the highest power of $\ell$ that divides the level $N_i$. On the other side, by the local-global compatibility (\cite[Theorem A]{Car86}, \cite{Sai97}) this Weil-Deligne representation will be isomorphic to the Frobenius semi-simplification of $\WD(\restr{\rho_i}{\Gamma_{\mathbb Q_\ell}})$. As the Weil-Deligne representation is already semi-simple on the inertia group the semi-simplification will not affect its Artin conductor. Therefore, we deduce that
    
    $$N_1 = \prod_{\ell \text{ prime}} \ell^{\mathfrak f(\WD(\restr{\rho_1}{\Gamma_{\mathbb Q_\ell}}))} = \prod_{\ell \text{ prime}} \ell^{\mathfrak f(\WD(\restr{\rho_2}{\Gamma_{\mathbb Q_\ell}}))} = N_2,$$

    \vspace{2 mm}

    \noindent and the claim about the level follows. On the other side, from the lifting method the lifts $\restr{\rho_i}{\Gamma_{\mathbb Q_p}}$ will have the same Hodge-Tate type as the local lifts $\rho_{i,p}$ coming from Lemma \ref{LocalLiftsmidp}. The local lifts are Hodge-Tate regular with Hodge-Tate weights $\{0,r-1\}$ for the integer $r \ge 2$ we fixed at the beginning. From the local-global compatibility at $p$ we get that $r_1 = r_2 = r$. Finally, we can compute the Neben character $\varepsilon_i$ using the determinant of $\rho_i$. By construction $\rho_1$ and $\rho_2$ have the same determinant $\mu = \kappa^{r-1}[\overline{\chi}\overline{\kappa}^{1-r}]$. Since $r_1 = r_2 =r$ we deduce that $\varepsilon_1 = \varepsilon_2 = [\overline{\chi}\overline{\kappa}^{1-r}]$. Thus, we conclude that $\rho_1$ and $\rho_2$ will be attached to newforms of the same weight $r\ge 2$, the same level $N \ge 1$, and the same Neben character $[\overline{\chi}\overline{\kappa}^{1-r}]$.
    
\end{proof}

\appendix

\section{Lemmas}

\label{Appendix}

In this section we assume that $G$ is a group as described in \S 2.

\begin{lem}
\label{NonInclusionLemma}

Let $R_1,R_2,\dots,R_n$ be rings such that $\{1,2,\dots,n\}$ are all invertible in $R_i$ for all $i$. For each $i$ let $M_i$ be an $R_i$-module, and $M_i'$ an $R_i$-submodule of $M_i$. Suppose that for every $i \le n$ we are given $x_i \in M_i$ such that $x_i \notin M_i'$. Then, for every given $y_i \in M_i$ there exists an integer $0 \le a \le n$ such that $a \cdot x_i + y_i \notin M_i'$ for all $i \le n$.

\end{lem}
\begin{proof}

For a fixed $i \le n$ we consider elements of the form $y_i + a\cdot x_i \in M_i$ where $a$ ranges through elements in the set $\{0,1,\dots,n\}$. We note that at most one of these elements can be an element of $M_i'$. Indeed, if $y_i + a\cdot x_i$ and $y_i + a'\cdot x_i$ are elements of $M_i'$ for distinct values $a$ and $a'$ in the specified range we have that $(a-a') \cdot x_i \in M_i'$. But, $0 < |a-a'| \le n$, so by our assumption $a-a'$ is invertible in $R_i$, which implies that $x_i \in M_i'$, contradicting our hypothesis. Therefore, for the result of the lemma to be true for that individual $i\le n$ we need to avoid at most one option for $a$ from the set above. In total, we need to avoid at most $n$ options for $a$. As we have $n+1$ options, for at least one value $a\in \{0,1,\dots,n\}$ we have that $y_i + a\cdot x_i \notin M_i'$ for all $i \le n$.
    
\end{proof}

\begin{lem}
\label{DensityLemma}

For some $N \ge 1$, let $\mathcal C$ be a positive upper-density subset of all $N$-tuples of primes in a number field $F$. Suppose that we have decomposed $\mathcal C$ as a finite disjoint union of subsets $\mathcal C_1,\mathcal C_2,\dots,\mathcal C_n$. Then, at least one $\mathcal C_i$ has positive upper-density.

\end{lem}
\begin{proof}

    For the sake of contradiction suppose that no $\mathcal C_i$ has positive upper-density. This in particular means that each of these sets has Dirichlet density $0$. As $\mathcal C$ is a disjoint union of finitely many Dirichlet density zero sets it has to have Dirichlet density $0$ itself. But, this contradicts the fact that it has positive upper-density. Hence, at least one $\mathcal C_i$ has positive upper-density. 
    
\end{proof}

\begin{lem}[cf. {\cite[Lemma $5.13$]{FKP21}}]
\label{TorusElementLemma}

    Let $s$ and $q$ be positive integers such that $q \equiv 1 \pmod{\varpi^s}$, but $q \not \equiv 1 \pmod{\varpi^{s+1}}$. Suppose $p \gg_G 0$. Then, for any split maximal torus $T$ of $G$, $\alpha \in \Phi(G^0,T)$, and any integer $n > s$ there exists an element $t_n \in T(\mathcal O/\varpi^n)$ such that $t_n$ is trivial modulo $\varpi^s$, $\alpha(t_n) \equiv q \pmod{\varpi^n}$, and $\beta(t_{s+1}) \not \equiv 1 \pmod{\varpi^{s+1}}$ for any $\beta \in \Phi(G^0,T)$, where $t_{s+1} \coloneqq t_n \pmod{\varpi^{s+1}}$. Moreover, we can prescribe the value of $t_n$ in $T(\mathcal O/\varpi^n)/T^{\der}(\mathcal O/\varpi^n)$.
    
\end{lem}
\begin{proof}

    We first show that such an element exists for $n=s+1$. We note that $q$ is a square modulo $\varpi^{s+1}$. Indeed

    $$\left(1 + \frac{q-1}2\right)^2 = 1 + (q-1) + \frac{(q-1)^2}4 = q \pmod{\varpi^{s+1}},$$

    \vspace{2 mm}

    \noindent where we used that $q-1$ is divisible by $\varpi^s$, so its square vanishes modulo $\varpi^{s+1}$. Set $q^{1/2} \coloneqq 1 + \frac{q-1}2$. In this step we will use the exponential isomorphism $\exp: \mathfrak g \otimes_{\mathcal O} \varpi^s\mathcal O/\varpi^{s+1}\mathcal O \to \ker(G(\mathcal O/\varpi^{s+1}) \to G(\mathcal O/\varpi^s))$. We now consider elements of the form $t_b = \exp(\varpi^sb)\alpha^\vee(q^{1/2})$, where $\alpha^\vee$ is the coroot associated to $\alpha$ and we view $b$ as an element of  $\ker(\restr{\alpha}{\mathfrak t \otimes k})$. As $b \in \mathfrak t \otimes k$ we have that $t_b \in T(\mathcal O/\varpi^{s+1})$. On the other hand

    $$\alpha(t_b) = \alpha(\exp(\varpi^sb))\alpha(\alpha^\vee(q^{1/2})) = \exp(\varpi^s\alpha(b))q^{\frac 12 \langle \alpha,\alpha^\vee \rangle} = q \pmod{\varpi^{s+1}},$$

    \vspace{2 mm}

    \noindent where we interchangeably view $\alpha$ as a root of both the Lie group and the associated Lie algebra. Using the decomposition $\mathfrak g = \mathfrak g^{\der} \oplus \mathfrak z_{G}$ we can write $b = b' + z$ and hence $\exp(\varpi^sb) = \exp(\varpi^sb')\exp(\varpi^sz)$. The second factor will lie in $Z_{G^0}(\mathcal O/\varpi^s)$, so it will not affect the value of $t_b$ under any of the roots. Thus, we can assume that $b \in \mathfrak t^{\der} \otimes k$. We note that we are free to choose any value for $z$, which can be used to prescribe the value of $t_b$ in $G(\mathcal O/\varpi^{s+1})/G^{\der}(\mathcal O/\varpi^{s+1})$.

    Now, it remains to choose $b \in \ker(\restr{\alpha}{\mathfrak t^{\der} \otimes k})$ in such a manner that the conditions on the rest of the roots in $\Phi(G^0,T)$ are satisfied. Suppose that $\beta(t_b) = 1$. Then we have

    $$1 = \beta(t_b) = \beta(\exp(\varpi^sb))\beta(\alpha^\vee(q^{1/2})) = \exp(\varpi^s\beta(b))q^{\frac 12 \langle \beta,\alpha^\vee \rangle } \pmod{\varpi^{s+1}}.$$

    \vspace{2 mm}

    Taking logarithm from both sides, i.e. the inverse of the exponential map we get

    $$0 = \varpi^s\beta(b) + \frac 12 \langle \beta,\alpha^\vee \rangle(q-1) \pmod{\varpi^{s+1}} \implies \beta(b) = \frac{1-q}{2\varpi^s}\langle \beta,\alpha^\vee \rangle \pmod{\varpi},$$

    \vspace{2 mm}

    \noindent where we used the fact that $\log(q) = q-1$. Now let $\Phi$ be the subset of $\Phi(G^0,T)$ consisting of roots $\beta$ for which there exists $b_\beta \in \ker(\restr{\alpha}{\mathfrak t^{\der} \otimes k})$ which satisfies the equation above. Then we want to choose $b \in \ker(\restr{\alpha}{\mathfrak t^{\der} \otimes k})$ that will lie in the complement of the union of hyperplanes

    $$\bigcup_{\beta \in \Phi} (b_\beta + \ker(\restr{\beta}{\ker(\restr{\alpha}{\mathfrak t^{\der} \otimes k})}).$$

    \vspace{2 mm}

    We remark that this is indeed a union of hyperplanes. Indeed, $\ker(\restr{\beta}{\ker(\restr{\alpha}{\mathfrak t^{\der} \otimes k})})$ is a hyperplane inside $\ker(\restr{\alpha}{\mathfrak t^{\der} \otimes k})$ as long as $\beta \neq \pm \alpha$, and $\pm \alpha \notin \Phi$ as that would imply that $q-1$ is divisible by $\varpi^{s+1}$, which contradicts our assumption. The number of such hyperplanes is bounded in terms of the size of $\Phi(G^0,T)$, which in turn only depends on the Dynkin type of $G^{\der}$. Hence, for $p \gg_G 0$ this union can't be equal to the whole of $\ker(\restr{\alpha}{\mathfrak t^{\der} \otimes k})$. We then choose $b$ in its complement and set $t_{s+1} = t_b$. By construction it is clear that this element will have the desired properties.

    Having constructed $t_{s+1}$ we can inductively construct such elements modulo higher powers of $\varpi$. Let $n \ge s+1$ and suppose that $t_n \in T(\mathcal O/\varpi^n)$ is an element that satisfies the given conditions. As the torus is formally smooth we can find a lift $t_{n+1}' \in T(\mathcal O/\varpi^{n+1})$ of $t_n$. It remains to arrange $\alpha(t_{n+1}') \equiv q \pmod{\varpi^{n+1}}$, as the rest of the conditions are satisfied automatically. By the inductive hypothesis, $\alpha(t_n) \equiv q \pmod{\varpi^n}$. Hence, we can write $\alpha(t_{n+1}') \equiv q(1+a\varpi^n) \pmod{\varpi^{n+1}}$. Set $t_{n+1} \coloneqq t_{n+1}'\alpha^\vee(1 - \frac a2 \varpi^n)$. Then

    $$\alpha(t_{n+1}) = \alpha(t_{n+1}') \cdot \alpha\left(\alpha^\vee \left( 1 - \frac a2 \varpi^s \right)\right) \equiv q(1+a\varpi^n)\left(1 - \frac a2 \varpi^n \right)^2 \equiv q(1+a\varpi^n)(1-a\varpi^n) \equiv q \pmod{\varpi^{n+1}}.$$

    \vspace{2 mm}

    Then, $t_{n+1}$ will be an element of $T(\mathcal O/\varpi^{n+1})$ as desired. As before, multiplying by $\exp(\varpi^nz)$ for $z \in \mathfrak z_{G} \otimes k$ we can prescribe the value of $t_{n+1}$ in $G(\mathcal O/\varpi^{n+1})/G^{\der}(\mathcal O/\varpi^{n+1})$, given we have already done that modulo $\varpi^n$.
    
\end{proof}

\begin{lem}
\label{GeneralLiftLemma}

Assume that $\nu \nmid p$. Let $\lambda_n:\Gamma_{F_\nu} \to G(\mathcal O/\varpi^n)$ be a morphism such that $\lambda_n \pmod{\varpi^s}$ is trivial for some $s < n$, $\lambda_n(\sigma_\nu) \in T(\mathcal O/\varpi^n)$ for some split maximal torus $T$ of $G$, $\alpha(\lambda_n(\sigma_\nu)) \equiv N(\nu) \not \equiv 1 \pmod{\varpi^n}$ for some $\alpha \in \Phi(G^0,T)$, $\lambda_n(\tau_\nu) \in U_\alpha(\mathcal O/\varpi^n)$, and $\beta(\lambda_n(\sigma_\nu)) \not \equiv 1 \pmod{\varpi^{s+1}}$ for all $\beta \in \Phi(G^0,T)$. Then we can find a lift $\lambda_{n+1}:\Gamma_{F_\nu} \to G(\mathcal O/\varpi^{n+1})$ of $\lambda_n$ that satisfies the same properties modulo $\varpi^{n+1}$. Moreover, we are free to prescribe the image of $\lambda_{n+1}(\sigma_\nu)$ in $G(\mathcal O/\varpi^{n+1})/G^{\der}(\mathcal O/\varpi^{n+1})$ as long as its reduction modulo $\varpi^n$ is equal to the image of $\lambda_n(\sigma_\nu)$ in $G(\mathcal O/\varpi^{n})/G^{\der}(\mathcal O/\varpi^{n})$.
    
\end{lem}
\begin{proof}

Set $t_n \coloneqq \lambda_n(\sigma_\nu)$. As in the inductive step of the proof of Lemma \ref{TorusElementLemma} using $q = N(\nu)$ we can produce a lift $t_{n+1} \in T(\mathcal O/\varpi^{n+1})$ of $t_n$, so that by setting $\lambda_{n+1}(\sigma_\nu) = t_{n+1}$ the conclusion of the lemma will be satisfied. Pick a lift $y_{n+1} \in \mathfrak g_\alpha \otimes_{\mathcal O} \mathcal O/\varpi^{n+1}$ of $y_n$. We then define a lift $\lambda_{n+1}:\Gamma_{F_\nu} \to G(\mathcal O/\varpi^{n+1})$ which factors through $\Gal(F_\nu^{\mathrm{tame},p}/F_\nu)$, and is given by $\lambda_{n+1}(\sigma_\nu) = t_{n+1}$ and $\lambda_{n+1}(\tau_\nu) = u_\alpha(y_{n+1})$. Clearly, $\lambda_{n+1}$ will satisfy the wanted properties, so it only remains to show that it is well-defined. In particular, we need to show it preserves the fundamental relation between $\sigma_\nu$ and $\tau_\nu$. This is indeed the case as

$$\lambda_{n+1}(\sigma_\nu\tau_\nu\sigma_\nu^{-1}) = t_{n+1}u_\alpha(y_{n+1})t_{n+1}^{-1} = u_{\alpha}(\ad_{t_{n+1}}(y_{n+1})) = u_{\alpha}(\alpha(t_{n+1})\cdot y_{n+1}) = u_{\alpha}(y_{n+1})^{\alpha(t_{n+1})} = \lambda_{n+1}(\tau_\nu^{N(\nu)}).$$
    
\end{proof}

\begin{lem}
\label{TorusandRootLemma}

Suppose that $G$ is simple and connected. Let $T$ be a split maximal torus in $G$, and $\alpha \in \Phi(G,T)$. Given non-zero $A_1,\dots,A_r \in \mathfrak g \otimes_{\mathcal O} k$, and $B_1,\dots,B_s \in (\mathfrak g \otimes_{\mathcal O} k)^*$, for $p \gg_G 0$ (depending on $r$ and $s$, as well) there exists $g \in G(k)$ such that $\Ad(g^{-1})A_i \notin \ker(\restr{\alpha}{\mathfrak t}) \oplus \bigoplus_\beta \mathfrak g_\beta$, and $\Ad(g^{-1})B_j \notin \mathfrak g_\alpha^\perp$ for all $i \le r$ and $j \le s$.   
    
\end{lem}
\begin{proof}

    Let $\Phi_{A_i}$ be the proper closed subscheme of $G$ defined by the condition $\Ad(g^{-1})A_i \in \ker(\restr{\alpha}{\mathfrak t}) \oplus \bigoplus_\beta \mathfrak g_\beta$. Analogously, we define $\Phi_{B_j}$. Let $Y$ be the union of all these subschemes. We want to show that the set of $k$-points of the complement $G \setminus Y$ is non-empty. By the  Bruhat's decomposition (\cite[Corollary $14.14$]{Bor91}) $G$ contains an open subset $X$ that is isomorphic to the open subset $\mathbb G_m^{\dim(T)} \times \mathbb A^{|\Phi(G,T)|}$ of an affine space $\mathbb A^{d_G}$, where $d_G \coloneqq \dim(G)$. Thus

    $$|X(k)| \ge (q-1)^{\dim(T)} \cdot q^{|\Phi(G,T)|},$$

    \vspace{2 mm}

    \noindent where $q$ is the cardinality of the field $k$. From the explicit description of $\Phi_{A_i}$ and $\Phi_{B_j}$ we get that $\Phi_{A_i} \cap X$ and $\Phi_{B_j} \cap X$ are cut out in $\mathbb A^{d_G}$ by equations of degree at most $d$, depending only on the root datum of $G$. We remark that the coordinates coming from $\mathbb G_m$ factors might have negative powers, but after clearing the denominators we can make sense of them over $\mathbb A^{d_G}$. Now, any polynomial of degree $d$ in $a$ variables over a field $k$ will have at most $dq^{a-1}$ solutions. Indeed, we are free to choose any values for $(a-1)$ of the variables, reducing it to a single-variable polynomial in $k$, which we know has no more solutions than its degree. From this we deduce that

    $$|(X \cap Y)(k)| \le (r+s)dq^{d_G-1}.$$

    \vspace{2 mm}

    Combining these two bounds we get

    $$|(G\setminus Y)(k)| \ge (q-1)^{\dim(T)}\cdot q^{|\Phi(G,T)|} - (r+s)dq^{d_G-1} \ge \frac{1}{2^{\dim(T)}}q^{d_G} - (r+s)dg^{d_G-1}.$$

    \vspace{2 mm}

    As the first term has the highest degree for $p \gg_G 0$ the right-hand side will be positive. Hence the $k$-points of the complement are non-empty.
    
\end{proof}

\begin{lem}
\label{RootSpaceSpan}

Let $p \gg_G 0$. For every $\mathbb F_p$-subspace $U$ of $\mathfrak g^{\der} \otimes k$ not equal to the whole of $\mathfrak g^{\der} \otimes k$ there is a split maximal torus $T$, a root $\alpha \in \Phi(G^0,T)$, and a root vector $X \in \mathfrak g_\alpha$ such that $X \pmod{\varpi} \notin U$.

\end{lem}
\begin{proof}

	It is enough to show that the $\mathbb F_p$-span, or equivalently the $k$-span, of all root vectors in $\mathfrak g^{\der} \otimes k$ is equal to the whole of $\mathfrak g^{\der} \otimes k$. We first note the $k$-span of root vectors in $\mathfrak g \otimes k$ is a $k[G(k)]$-module under the adjoint action, because it sends a root vector to another root vector, possibly with respect to another another split maximal torus and root for it. For $p \gg_G 0$ the simple factors of $\mathfrak g^{\der} \otimes k$ are irreducible $k[G(k)]$-modules, so it suffices to show that the $k$-span of root vectors has a non-trivial projection on each of the simple factors of $\mathfrak g^{\der} \otimes k$. Directly from a root decomposition of $\mathfrak g^{\der}$ we get that each simple factor of $\mathfrak g^{\der} \otimes k$ contains a non-trivial root vector, giving us what we want.
  
\end{proof}

\begin{lem}
\label{AbelianizationLemma}

Suppose $p \gg_G 0$. For $m \ge n \ge 1$, let $G_{m,n}$ denote the kernel of the modulo $\varpi^n$ reduction $G(\mathcal O/\varpi^m) \to G(\mathcal O/\varpi^n)$. Then

$$G_{m,n}^{\mathrm{ab}} \simeq \begin{cases} G_{m,n} &; \text{if } m \le 2n \\ G_{2n,n} &; \text{if } m \ge 2n  \end{cases}.$$
    
\end{lem}
\begin{proof}

    Let $m \le 2n$. Using the isomorphism coming from the exponential map we have that $\mathfrak g \otimes_{\mathcal O} \varpi^n\mathcal O/\varpi^m\mathcal O \simeq G_{m,n}$. As this is an isomorphism of groups we have that $G_{m,n}$ is abelian and the claim for $m \le 2n$ follows.

    For $m \ge 2n$, as $G$ is smooth from the kernel-cokernel sequence we have that $G_{m,n}/G_{m,2n} \simeq G_{2n,n}$, which by above we know is abelian. Thus $[G_{m,n},G_{m,n}] \subseteq G_{m,2n}$. It remains to prove the reverse inclusion. For a fixed $n$ we will prove this by induction on $m$. The base case $m = 2n$ is already proven. Now, let $ m > 2n$ and suppose that $[G_{m-1,n},G_{m-1,n}] = G_{m-1,2n}$. For $g \in G_{m,2n}$ by the inductive hypothesis $g \pmod{\varpi^{m-1}}$ is a commutator in $G_{m-1,n}$. Hence, using the smoothness of $G$ we can find $g_1,g_2 \in G_{m,n}$ such that $g[g_1,g_2] \equiv 1 \pmod{\varpi^{m-1}}$. This means that $g[g_1,g_2] \in G_{m,m-1}$ and therefore it suffices to show that $G_{m,m-1} \subseteq [G_{m,n},G_{m,n}]$. We now consider $G_{m,m-1} \cap [G_{m,n},G_{m,n}]$. We claim that this is an $\mathbb F_p[G(k)]$-submodule of $G_{m,m-1}$. By the exponential map $G_{m,m-1} \simeq \mathfrak g \otimes_{\mathcal O} \varpi^{m-1}\mathcal O/\varpi^{m}\mathcal O \simeq \mathfrak g \otimes_{\mathcal O} k$, which is an abelian group of exponent $p$. Thus, by being a subgroup of it $G_{m,m-1} \cap [G_{m,n},G_{m,n}]$ is an $\mathbb F_p$-vector space. The $G(k)$ action on $G_{m,m-1}$ is given by lifting to $G(\mathcal O/\varpi^m)$ and conjugating. Since $z[x,y]z^{-1} = [zxz^{-1},zyz^{-1}]$ we get that $G_{m,m-1} \cap [G_{m,n},G_{m,n}]$ is closed under conjugation by $G(k)$ and hence it is an $\mathbb F_p[G(k)]$-module.

    For $p \gg_G 0$ the simple factors of $\mathfrak g^{\der} \otimes_{\mathcal O} k$ are irreducible $\mathbb F_p[G(k)]$-modules. The exponential map is $G(k)$-equivariant, so the same is true for $G_{m,m-1}$. Therefore, to show that $G_{m,m-1} \cap [G_{m,n},G_{m,n}]$ is equal to the whole of $G_{m,m-1}$ it suffices to show that $G_{m,m-1} \cap [G_{m,n},G_{m,n}]$ has a non-trivial projection on each of the simple factors of $\mathfrak g^{\der} \otimes_{\mathcal O} k$ under the inverse of the exponential map. For a fixed simple factor of $\mathfrak g^{\der} \otimes_{\mathcal O} k$ let $\mathfrak g_\alpha$ be a root space contained in it for some root $\alpha \in \Phi(G^0,T)$ for some split maximal torus $T$. As in Lemma \ref{TorusElementLemma} we can find $t \in T(\mathcal O/\varpi^m)$ which is trivial modulo $\varpi^{m-n-1}$ and $\alpha(t) \not \equiv 1 \pmod{\varpi^{m-n}}$. Now, let $Y_\alpha$ be a basis of the root space $\mathfrak g_\alpha$ over $\mathcal O$. Then

    $$[t,u_\alpha(\varpi^nY_\alpha)] = tu_\alpha(\varpi^nY_\alpha)t^{-1}u_\alpha(-\varpi^nY_\alpha) = u_\alpha(\Ad_t(\varpi^nY_\alpha)-\varpi^nY_\alpha) = u_{\alpha}((\alpha(t)-1)\varpi^nY_\alpha).$$

    \vspace{2 mm}

    As $\alpha(t)-1$ is divisible by $\varpi^{m-n-1}$, but not by $\varpi^{m-n}$ this is a non-trivial element of $G_{m,m-1}$. As $m-n-1 \ge n$ we have that $t$ is trivial modulo $\varpi^n$ and the same is true for $u_{\alpha}(\varpi^nY_\alpha)$. Therefore, $[t,u_{\alpha}(\varpi^nY_\alpha)]$ is a non-trivial element of $G_{m,m-1} \cap [G_{m,n},G_{m,n}]$. By the discussion above we conclude that this group is equal to $G_{m,m-1}$, which implies that $G_{m,m-1} \subseteq [G_{m,n},G_{m,n}]$. This yields $[G_{m,n},G_{m,n}] = G_{m,2n}$ and $G_{m,n}^{\mathrm{ab}} \simeq G_{2n,n}$

\end{proof}

\bibliographystyle{alpha}
\bibliography{biblio}

\end{document}